\documentclass{article}
\usepackage{amsmath}
\usepackage{amssymb}
\usepackage{amsthm}
\usepackage[margin=2.5cm]{geometry}
\usepackage{bbm}
\usepackage{amsfonts}
\usepackage{mathtools}
\usepackage[mathscr]{euscript}
\usepackage[utf8]{inputenc}
\usepackage{color}
\usepackage{enumerate} 
\usepackage[shortlabels]{enumitem} 
\usepackage{bm}
\usepackage{graphicx}
\usepackage[hidelinks]{hyperref}
\usepackage[font=small,labelfont=bf]{caption}
\usepackage{tikz, float} 
\usepackage{authblk}

\providecommand{\keywords}[1]{
  \par\addvspace\baselineskip 
  \noindent{\scriptsize\textbf{\textit{Keywords---}} #1}
}
\providecommand{\MSC}[1]{
  \noindent{\scriptsize\textbf{\textit{MSC---}} #1}\par\addvspace\baselineskip
}

\newcommand\dr{\; \mathrm{d}}
\newcommand\maxim{\mathcal{M}}
\newcommand\gener{\mathcal{L}}
\newcommand\low[1]{\left\langle #1 \right\rangle}
\DeclareMathOperator\sgn{sgn}

\newcommand\st{\,\mathrm{s.t.}\,}

\newcommand\prob[2][]{\mathbb{P}_{#1}\left(#2\right)}

\newcommand\probT[2][]{\tilde{\mathbb{P}}_{#1}\left(#2\right)}
\newcommand\probTWhen[3][]{\tilde{\mathbb{P}}_{#1}\left(#2 \,\middle|\, #3\right)}

\newcommand\expect[2][]{\mathbb{E}_{#1}\left[#2\right]}

\newcommand\expectT[2][]{\tilde{\mathbb{E}}_{#1}\left[#2\right]}
\newcommand\expectTWhen[3][]{\tilde{\mathbb{E}}_{#1}\left[#2 \,\middle|\, #3\right]}

\newcommand\var[2][]{\mathrm{Var}_{#1}\left(#2\right)}

\newcommand\indep{\perp\!\!\!\perp}
\newcommand\indic{\mathbbm{1}}

\newcommand\GrtrPath{\Gamma_{\scalebox{0.5}{$\geq$}}}
\newcommand\LssrPath{\Gamma_{\scalebox{0.5}{$\leq$}}}

\newtheorem{theorem}{Theorem}[section]
\newtheorem{corollary}[theorem]{Corollary}
\newtheorem{lemma}[theorem]{Lemma}
\newtheorem{proposition}[theorem]{Proposition}

\theoremstyle{definition}
\newtheorem{definition}[theorem]{Definition}
\newtheorem{remark}[theorem]{Remark}

\newtheorem{example}[theorem]{Example}

\numberwithin{equation}{section}

\allowdisplaybreaks

\title{Crystal Growth on Locally Finite Partially Ordered Sets}

\author{Tanner J. Reese\thanks{University of Arizona, \texttt{treese1@arizona.edu}} \ and Sunder Sethuraman\thanks{University of Arizona, \texttt{sethuram@arizona.edu}}}

\begin{document}
\maketitle

\begin{abstract}
We consider a Markovian growth process on a partially ordered set $ \Lambda $,
equivalent to last passage percolation (LPP) with independent (not necessarily identical)
exponentially distributed weights on the elements of $ \Lambda $.
Such a process includes inhomogeneous exponential LPP on the Euclidean lattice $ \mathbb{N}_0^d $.  
We give non-asymptotic bounds on the mean and variance, as well as higher, central, and exponential moments of the passage time $ \tau_A $
to grow any set $ A \subseteq \Lambda $ in terms of characteristics of $ A $.  
We also give a limit shape theorem when $ \Lambda $ is equipped with a monoid structure.
Methods involve making use of the backward equation associated to the Markovian evolution
and comparison inequalities with respect to the time-reversed generator.
\end{abstract}


\MSC{60K35, 06A06}

\tableofcontents

\section{Introduction}
We study a `crystal growth' process on locally finite partially ordered sets (posets),
which includes Euclidean lattices.
From a broader viewpoint, as is well-known, such growths are of interest in material science and other applications.
The model we consider maps to a `last passage percolation' model
with independent (inhomogeneous) exponential weights on the poset.

Our aim in this article is to give non-asymptotic bounds on the moments and variances
of the passage times $ \tau_A $ for the formation of sets $ A $, in terms of characteristics of $ A$.
We also consider a law of large numbers (LLN) shape limit theorem for $ \tau_{A^n} / n $ when the poset is a monoid
(see below for definitions), allowing one to define the iterates $ A^n $.
In the Euclidean lattice context,
even when $ d = 2 $ and the weights are independent and exponential with homogeneous rates,
in which case, there is a wealth of celebrated results,
our bounds on the mean, variance, higher moments, and moment generating function
of $ \tau_A $, for arbitrary $ A $, appear new.
\medskip 

Although we describe our results for an arbitrary poset $ \Lambda $,
one may like to keep in mind the standard $ d $-dimensional lattice context
where $ \Lambda = \mathbb{N}_0^d \subseteq \mathbb{Z}^d $ (where $\mathbb{N}_0 =\{0,1,\ldots\}$)
and for all $ \alpha, \beta \in \mathbb{N}_0^d $
with $ \alpha = (\alpha_1, \ldots, \alpha_d) $ and $ \beta = (\beta_1, \ldots, \beta_d) $,
\begin{equation*}
\alpha \leq \beta \;\text{ if and only if }\; \alpha_i \leq \beta_i \text{ for all } 1 \leq i \leq d.
\end{equation*}
We will consider a Markov process $ X_t $ that describes a crystal growth in $ \Lambda $.
In this process, more carefully defined and generalized to locally finite posets in Section \ref{sect:growth-process},
we require that all of the sites `below' $ \alpha \in \Lambda $
must already be present before the process can add $ \alpha $ to its growth with a certain rate $ \lambda_\alpha $.

For any subset $ S \subseteq \Lambda $, we say $ S $ is a \textbf{lower set}
when for all $ \beta \in S $,
if $ \alpha \in \Lambda $ with $ \alpha \leq \beta $ then $ \alpha \in S $ (see Figure \ref{fig:maximal-sets} for an example).
We denote by $ L(\Lambda) $ the set of finite lower sets in $ \Lambda $.
Note that $ L(\Lambda) $ is also a poset under inclusion $ \subseteq $.
For any $ \beta \in \Lambda $ and $ B \in L(\Lambda) $, we say
\begin{equation*}
\low{\beta} := \{\alpha \in \Lambda \;:\; \alpha \leq \beta\} \in L(\Lambda)
\quad \text{and} \quad
\low{B} := \{A \in L(\Lambda) \;:\; A \subseteq B\} \subseteq L(\Lambda).
\end{equation*}

Because of the specifications of the growth,
for all $ t \geq 0 $, we have $ X_t \in L(\Lambda) $.
Since $ L(\Lambda) $ is a poset, we may define the stopping times of $ X_t $:
For any $ A \in L(\Lambda) $, we say
\begin{equation} \label{eq:stop-time-defn}
\tau_A := \inf\{t \in [0, \infty) \,:\, A \subseteq X_t\}
= \sup\{t \in [0, \infty) \,:\, A \nsubseteq X_t\}
\end{equation}
and by abuse of notation, $ \tau_\alpha = \tau_{\low{\alpha}} $ for $ \alpha \in \Lambda $.
Because $ X_t $ is increasing with respect to the partial order,
for any $ t \in [0, \infty) $ and $ A \in L(\Lambda) $,
the following two events are equivalent:
\begin{equation} \label{eq:stop-time-dual}
\{A \subseteq X_t\} = \{\tau_A \leq t\}.
\end{equation}
Importantly, such a set equivalence will allow us later
to use the backward equation of the process $ X_t $
to characterize the distribution function $ \prob{\tau_A \leq t} $.

When $ \Lambda $ is a general poset, the associated process $ X_t $ is a natural generalization of
the crystal growth model on $ \Lambda = \mathbb{N}_0^d $.
Note, when $ \Lambda= \mathbb{N}_0^d $, the group structure of $ \mathbb{Z}^d $
allows us to define $ A + B $ for any $ A, B \in L(\mathbb{Z}^d) $.
Then, for any integer $ n \geq 0 $,
we have $ n \cdot A = \{\alpha_1 + \ldots + \alpha_n \,:\, \alpha_1, \ldots, \alpha_n \in A\} $
and a limiting shape function $ g(\alpha) = \lim_{n \to \infty} \tau_{n \cdot \low{\alpha}} / n $ can be formulated.
In a general poset $ \Lambda $,
one needs a way to generalize the operation $ + $ in $ \mathbb{Z}^d $ to $ \Lambda $.
A natural way to do this is to assume there exists a monoid structure on $ \Lambda $.
More specifically, there is an associative binary operation $ (x, y) \mapsto xy : \Lambda \times \Lambda \to \Lambda $
along with an identity $ 1_\Lambda \in \Lambda $.
These notions are developed and stated in Section \ref{sect:monoids}.
General references to posets and monoids include \cite{Jaco12} and \cite{Stan12}.

We mention a couple of examples to keep in mind, beyond the Euclidean case $ \Lambda = \mathbb{N}_0^d $.

\begin{example}
\label{exm:lattice-ord}
A large class of partially ordered monoids that directly generalize $ \mathbb{N}_0^d $
are the positive cones $ C $ of finite dimensional vector spaces
intersected with lattices $ \overline{\Lambda} $.
Suppose $ \leq_C $ is a partial order on $ \mathbb{R}^d $ defined by a cone $ C \subseteq \mathbb{R}^d $
(i.e. $ x \leq_C y $ when $ y - x \in C $)
and let $ \overline{\Lambda} \subseteq \mathbb{R}^d $ be a discrete subgroup.
Then, we define $ \Lambda = \overline{\Lambda} \cap C $ which is
the non-negative cone of $ \geq_C $ when restricted to $ \overline{\Lambda} $.
When $ C $ is the non-negative octant in $ d $-dimensions and $ \overline{\Lambda} = \mathbb{Z}^d $,
we recover the simple lattice model $ \mathbb{N}_0^d $.

\begin{figure}[H] \centering
\begin{tikzpicture}[line width=0.4mm]
\draw[<->] (0, -1.4) -- (0, 4.4);
\draw[<->] (-1.4, 0) -- (4.4, 0);

\fill[red, semitransparent] (0, 0) -- (4.4, 2.2) -- (4.4, 4.4) -- (2.2, 4.4) -- cycle;
\draw[red] (2.2, 4.4) -- (0, 0) -- (4.4, 2.2);
\node[red] at (4.4, 4.6) {$ C $};

\foreach \x in {-1, ..., 4}
  \foreach \y in {-1, ..., 4}
    \filldraw[black] (\x, \y) circle (2pt);
\end{tikzpicture}
\caption{The positive cone $ C $ in $ \overline{\Lambda} = \mathbb{Z}^2 $}
\label{fig:cone-in-lattice}
\end{figure}
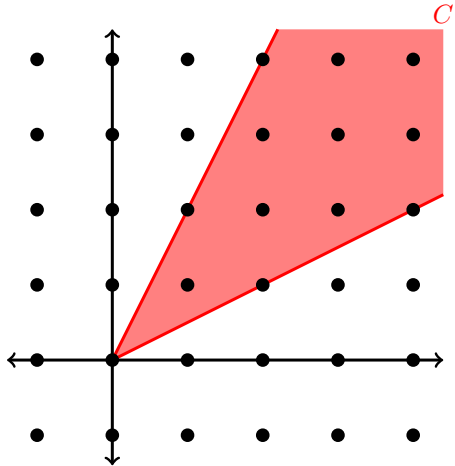

To give another example, we can take $ \overline{\Lambda} $
to be the free group on $ d $ generators $ S $.
Define
\begin{equation*}
\Lambda := \{a_1 \cdots a_n \in \overline{\Lambda} \,:\, a_1, \ldots, a_n \in S\}
\end{equation*}
so that $ \Lambda $ are the elements in $ \overline{\Lambda} $
corresponding to words not containing inverses of any of the generators.
Then, one may think of $ \Lambda $ as a directed tree
where each element has exactly $ d $ elements directly above it,
and ordering is defined by the existence of directed paths in the tree.
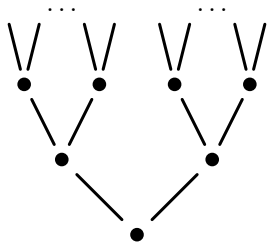
\begin{figure}[H] \centering
\begin{tikzpicture}[line cap=round, line width=0.4mm]

\filldraw (0, 0) circle (2pt);
\draw (-0.2, 0.2) -- (-0.8, 0.8);
\filldraw (-1, 1) circle (2pt);
\draw (0.2, 0.2) -- (0.8, 0.8);
\filldraw (1, 1) circle (2pt);
\draw (-1.1, 1.2) -- (-1.4, 1.8);
\filldraw (-1.5, 2) circle (2pt);
\draw (-0.9, 1.2) -- (-0.6, 1.8);
\filldraw (-0.5, 2) circle (2pt);
\draw (0.9, 1.2) -- (0.6, 1.8);
\filldraw (0.5, 2) circle (2pt);
\draw (1.1, 1.2) -- (1.4, 1.8);
\filldraw (1.5, 2) circle (2pt);

\draw (-1.55, 2.2) -- (-1.7, 2.8);
\draw (-1.45, 2.2) -- (-1.3, 2.8);
\draw (-0.55, 2.2) -- (-0.7, 2.8);
\draw (-0.45, 2.2) -- (-0.3, 2.8);
\draw (0.45, 2.2) -- (0.3, 2.8);
\draw (0.55, 2.2) -- (0.7, 2.8);
\draw (1.45, 2.2) -- (1.3, 2.8);
\draw (1.55, 2.2) -- (1.7, 2.8);
\node at (-1, 3) {$ \ldots $};
\node at (1, 3) {$ \ldots $};

\end{tikzpicture}
\caption{$ \Lambda $ is a tree order with two elements immediately above each element}
\label{fig:tree-poset}
\end{figure}
\end{example}

\subsection{Discussion of Previous Literature}
When $ \Lambda = \mathbb{N}_0^d $ and $ d = 2 $, the process $ X_t $ is the well-known `corner growth' model
with independent exponential weights $ \{G_\alpha\}_{\alpha \in \Lambda} $
with rates $ \{\lambda_\alpha\}_{\alpha \in \Lambda} $ (cf. \cite{Emr16}, \cite{LS12}, \cite{Sepp09}).
One may also view the times $ \tau_A $ as passage times
in an inhomogeneous `last passage percolation' (LPP) model.
That is, $ \tau_A $ corresponds to the maximal cost $\max_\pi \sum_{i=1}^{\ell(\pi)} G_{\pi_i} $
over `maximal increasing' paths $ \pi $,
where $ \ell(\pi) $ is the length of the path, $ \pi_i $ is the $ i $-th element in the path.
See Section \ref{sect:lpp-relation} which makes this correspondence precise on a general poset $\Lambda$.

With respect to $ \Lambda = \mathbb{N}_0^d $,
for independent, identically distributed (i.i.d.) exponential weights with parameter $ \lambda $
(i.e. $ \lambda_\alpha \equiv \lambda $ is constant),
the process in $ d = 2 $ is part of an `exactly solvable' class.
There are many connections with combinatorics, random matrices, polymer models,
and totally asymmetric exclusion interacting particle systems
through which the limiting scaled, centered statistics of $ \tau_{\xi_n} $  have been found, among other limits.
In particular, when $ \xi_n = (n, \lfloor \gamma n \rfloor) \in \mathbb{N}_0^2 $ for $ \gamma \in (0, 1) $,
the exact orders of the mean and variance of $ \tau_{\xi_n} $ are $ n $ and $ n^{2/3} $, respectively,
the latter corresponding to the well-known `KPZ' relation.
A survey of these and other celebrated results and their histories can be seen for instance
in \cite{AG23}, \cite{Baik05}, \cite{BBS21}, \cite{BCS06}, \cite{BFO20},
\cite{Corw24}, \cite{DRS18}, \cite{Krap21}, \cite{Mart06}, and \cite{Sepp09}.

Further, in $ d \geq 2 $, when the distribution of $ \{G_\alpha\}_{\alpha \in \Lambda} $ is more arbitrary,
there are several results with respect to the law of $ \tau_{\xi_n} $,
although not as complete as when the LPP model is `exactly solvable'.
In $ d \geq 2 $, for a general class of i.i.d. `weights' $\{G_\alpha\}_{\alpha \in \Lambda} $
satisfying a moment condition, the a.s. convergence
of an abstract LLN limit shape $ g(x) = \lim_{n \to \infty} \tau_{\low{\xi_n}} / n $
is known (cf. \cite{DGK01}, \cite{Hamm06}, \cite{Mart04}).
In $ d = 2 $, close to an axis, when $ \xi_n = (n, \lfloor \gamma n\rfloor) $,
the shape has `universal' asymptotics $ g(1, \gamma) = \mu + 2\sigma \sqrt{\gamma} + o(\sqrt{\gamma}) $,
as $ \gamma \to 0^+ $, where $ \mu $ and $ \sigma^2 $ are the mean and variance of $ G_\alpha $
(cf. \cite{LS12}, \cite{Mart04}, and \cite{Mart06}).
Moreover, when $ \xi_n = (n, \lfloor n^a \rfloor) $,
the time $ \tau_{\xi_n} $, centered by $ n\mu + 2\sigma n^{\frac{1+a}{2}} $
and scaled by $ \sigma n^{1/2 - a/6} $,
converges to $ F_2 $, the GUE Tracy-Widom distribution (cf. \cite{BM05}).
For heavy-tailed weights, convergence of the scaled passage times
to those in a continuum last passage model has been shown (cf. \cite{HM07}, \cite{GGN25}).

In $ d = 2 $, for types of inhomogeneous independent exponential weights,
the limit shape $ g(\cdot) $ may also be computed (cf. \cite{Emr16}, \cite{EJ17}, \cite{EJS21}, or \cite{LS12}).
Also, with respect to types of inhomogeneous independent exponential weights,
among other results, variance and central moment bounds of $ \tau_{\xi_n} $ with orders matching those
in the i.i.d. `exactly solvable' model are shown (cf. \cite{EJS23}, \cite{EFL25}, \cite{EGO25}).
We comment also with respect to the homogeneous growth process $ X_t $ in $ d = 3 $ (or LPP with i.i.d. exponential weights),
there are conjectures for the explicit form of the limit shape $ g(\cdot) $ in \cite{OKRM12};
see also \cite{TNM10} for a mapping of the growth process to coupled totally asymmetric exclusions, a `zigzag' process.

Also, in $ d \geq 2 $, for a class of i.i.d. weights,
including gamma distributions (and so exponential distributions),
the variance of $ \tau_{\xi_n} $, where $ \xi_n = n\big(e_1 + \ldots + e_d\big) $ and $ \{e_i\}_{i=1}^d $
are the standard basis vectors of $ \mathbb{Z}^d $,
is bounded $ \var{\tau_{\xi_n}} \leq Cn / \log(n)$ (cf. \cite{Grah12}).
See also \cite{BKS03} and \cite{BR08} in the `first-passage percolation' (FPP) context
as well as \cite{BK12} for related variance bounds in a dependent FPP model.

In $ d = 2 $, the bound $ C \sqrt{\log(n)} \leq \var{\tau_{\xi_n}} $
has been proven for a general class of i.i.d. weights $ \{G_\alpha\}_{\alpha \in \Lambda} $
(cf. \cite{BC20} and, for the FPP context, \cite{DHHX20}).
Given that a positive constant lower bound $ C \leq \var{\tau_{\xi_n}} $ in $ d \geq 3 $ in FPP is known (cf. \cite{Kest93}),
one might formulate a similar bound in LPP for a class of i.i.d. weights,
although it seems such a statement is not extant.

In $ d \geq 2 $, these upper and lower bounds for the variance $ \var{\tau_{\xi_n}} $
in non-exactly solvable i.i.d. weights $ \{G_\alpha\}_{\alpha\in \Lambda} $ LPP models are the best known,
the conjectured order being $ n^{2/3} $ when $ d = 2 $ and the weights have sufficient moments as in exactly solvable models.
However, in $ d \geq 3 $, the order $ n^\beta $ (if a power $ \beta $ exists) of the variance is unresolved,
although there are numerical simulations, including \cite{TFW92}.

Although LPP models have been studied mostly in the $ \Lambda = \mathbb{N}_0^d $ setting,
the work \cite{AC21} formulates the process on posets and discusses a geometrical perspective of passage times.
Also, a few other settings with different schemes have been considered.
The work \cite{FKMR24} considers Barak-Erd\H{o}s graphs on $ \mathbb{Z} $ and related `slabs' on posets,
while the work \cite{WWZ20} features the complete graph with $ N $ vertices.  
Also, the work \cite{BT21} considers a continuous space model.
In the first two works, weights are put on the edges of the graphs
and last passage times are defined via maximal length of (directed or self-avoiding) paths.
The third work considers a generalization of Hammersley's last passage percolation problem on $ [0,1]^2 $.
We note also in passing that there are several works on FPP in settings beyond the standard $ \mathbb{Z}^d $ model,
including \cite{AG23}, \cite{BM22}, \cite{BP94}, \cite{BT15}, \cite{BZ12},
\cite{BDG20}, \cite{CLHJV23}, \cite{DHO24}, \cite{EGHN20}, \cite{HN99}, and \cite{LW14}.

\subsection{Sketch of Results and Methods}

In view of the previous discussion,
our results are summarized as follows.
In a general poset $ \Lambda $,
including the i.i.d. exponential weights Euclidean LPP model $ \Lambda = \mathbb{N}_0^d $,
we provide several non-asymptotic bounds in Section \ref{sect:results}
with respect to the passage times $ \tau_A $ when $ A \in L(\Lambda) $ is an arbitrary set.
These include bounds for $ \var{\tau_A} $, the moments $ \expect{\tau_A^n} $,
the central moments $ \expect{(\tau_A - \expect{\tau_A})^n} $, 
and also the moment generating function $ u \mapsto \expect{e^{u \tau_A}} $.
These estimates are specified in terms of characteristics of the set $ A \in L(\Lambda) $
and the inhomogeneous rates $ \{\lambda_\alpha\}_{\alpha\in \Lambda} $.
We also give a LLN shape limit theorem when the poset $ \Lambda $ is a monoid (Theorem \ref{thm:monoid-shape-func}).

More concretely, we show that for each $ A \in L(\Lambda) $
that $ \lambda_-(A) \var{\tau_A} \leq \expect{\tau_A} $ (see Theorem \ref{thm:var-sublinear})
where $ \lambda_-(A) $ is the minimal rate $ \lambda_\alpha $ for $ \alpha \in A $.
We also bound $ \lambda_-(A) \expect{\tau_A} $
by $ \left(\sqrt{\ell(A)} + \sqrt{\kappa(A) + \eta(A)}\right)^2 $ (see Theorem \ref{thm:mean-upper})
in terms of `geometrical' quantities where $ \ell(A) $ is a `length', $ \kappa(A) $ is a `width',
and $ \eta(A) $ measures a `spread of the rates' for $ A $.
We note that a lower bound $ \lambda_+(A) \expect{\tau_A} \geq \ell(A) $ is given in Proposition \ref{prop:mean-lower}.

In the Euclidean context $ \Lambda = \mathbb{N}_0^d $,
one can show that $ \kappa(A) \leq \log(d) \ell(A) $ (Lemma \ref{lem:kappa-bnd}).
Further, refinements of the upper bound when $ A = \low{\xi_n} $
in terms of the `entropies' of $ \xi_n $ are given in Theorem \ref{thm:lattice-kappa-limit}.
Then, $ \var{\tau_A} $ is of linear order in $ \ell(A) $,
which gives a `diffusive' variance bound.
This is sharp in that it holds for each $ A \in L(\Lambda) $,
including when $ A = \low{\xi_n} $ and $ \xi_n $ lies on an axis.
Recall in this case, $ \var{\tau_A} $ is the variance of
a sum of $ \ell(A) $ independent exponential random variables. 

We also give a lower bound for $ \var{\tau_A} $ in terms of the structure of $ A $
(see Proposition \ref{prop:var-mxml-bnd}).
In the Euclidean context, the bound for $ A = \low{\xi_n} $
reduces to a $ \log(n) $ lower bound in $ d = 2 $ (see \cite{BC20} in a homogeneous context).
Additionally, in $ d \geq 3 $, this estimate gives a positive constant lower bound which,
although understood in the FPP context \cite{Kest93}, seems not to be written in the LPP literature
(see Corollary \ref{coro:var-away-zero} and Corollary \ref{coro:var-2d-bnd}).
We also give an example of a poset $ \Lambda $ where $ \var{\tau_A} $
is bounded below in order by $ |A| $ (Example \ref{exm:low-bnd-poset}).

As previously commented, in the setting of independent exponential weights LPP models in $ d \geq 2 $,
when $ A \in L(\Lambda) $ is an arbitrary set, not necessarily in the form $ \low{\xi_n} $,
the non-asymptotic estimates for the mean $ \expect{\tau_A} $, variance $ \var{\tau_A} $, and other statistics
in terms of properties of $ A $ seem novel.
Of course, as noted earlier, there are better results
when $ A $ is of the form $ \low{\xi_n} $ and $ d = 2 $ in a homogeneous setting.
Also, the LLN shape limit theorem given in a class of monoids
seems to be one of the first generalizations of the shape theorem beyond the Euclidean setting. 
In particular, application of the results to general inhomogeneous independent exponential weight
Euclidean LPP models in $ d \geq 2 $ seems new.

In terms of proofs, the bounds follow from analyzing the backward equation
of the increasing process $ X_t $ in terms of a time-reversed operator $ \Delta $.
In Lemma \ref{lem:diff-deriv-dual}, a useful identity is derived by applying
the backward equation to indicator functions of sets $ \indic[A \subseteq X_t] $.
We show an important comparison inequality Proposition \ref{prop:diff-ineq}
which allows one to bound functions of sets $ g(A) $ by $ f(A) $
when $ (\Delta g)(A) \leq (\Delta f)(A) $ and $ g(\emptyset) \leq f(\emptyset) $.
 
Accordingly, we obtain our estimates by choosing appropriate functions $ g $ and $ f $.
For example, for the variance upper bound in Theorem \ref{thm:var-sublinear},
we will consider functions $ g(A) = \var{\tau_A} $ and $ f(A) = \expect{\tau_A} $.
To bound $ \expect{\tau_A} $, we use
a bound on the moment generating function $ \expect{e^{u \tau_A}} $
(see Proposition \ref{prop:grtr-path-mgf}).
We obtain this bound by considering types of `path functions',
which are discrete surrogates of the exponential function
(cf. \cite{CGGK93} and \cite{Mart02} on bounding `greedy lattice animals' for related objects).
The structure of the poset $ \Lambda $ also plays a role.

Our bounding techniques seem quite different from previous methods.
They in a sense allow the `worst' case of $ A \in L(\Lambda) $,
and so are agnostic to the size of $ \maxim(A) $
which could be as small as $ 1 $ or as large as $ |A| $.
However, our techniques require that $ X_t $ is a Markov process.
That is, the weights $ \{G_\alpha\}_{\alpha\in \Lambda} $ are exponentially distributed, although
we allow their rates $ \{\lambda_\alpha\} $ to be inhomogeneous
and the sets $ A \in L(\Lambda) $ to be arbitrary.

For the monoid LLN shape limit result, as in the known Euclidean context,
we make use of the superadditivity of $ \expect{\tau_{A^n}} / n $.
However, to show the limit is finite in the general setting,
we analyze an effective mean upper bound $ \left(\sqrt{\ell(A^n)} + \sqrt{\kappa(A^n) + \eta(A^n)}\right)^2 $,
via estimates on $ \kappa(A^n) / \ell(A^n) $
and a natural `steadiness' assumption $ \ell(A^n) \leq C\ell_*(A^n) $,
where $ \ell_*(A^n) $ is a `minimal length' that we show is a subadditive sequence.

In part, our results extend to non-exponential independent weights $ \{G_\alpha\} $.
When the weights are all stochastically less or all stochastically more than an exponential,
the bounds for the mean and moment-generating function extend in a natural way.
Similarly, some of the results with respect to the shape limit theorem may be extended
as well (see Subsection \ref{subsect:stoch-less}).

\medskip
\noindent {\bf Plan of the Article}
We begin with a formal definition of the process $ X_t $ in Section \ref{sect:growth-process},
followed by statements of the results in Section \ref{sect:results}.
After introducing preliminary results for the time-reversed generator of $ X_t $ in Section \ref{sect:diff-oper},
we prove our results in Section \ref{sect:proofs}.

\section{Crystal Growth Markov Process}
\label{sect:growth-process}
We define now the class of partial ordered sets $ \Lambda $ considered in the article.
This class includes the Euclidean model where $ \Lambda = \mathbb{N}_0^d $.
For any elements $ \alpha, \beta \in \Lambda $ with $ \alpha < \beta $,
we say $ \beta $ is an \textbf{upper neighbor} of $ \alpha $
or, reciprocally, $ \alpha $ is a \textbf{lower neighbor} of $ \beta $
if there does \emph{not} exist $ x \in \Lambda $ with $ \alpha < x < \beta $.
We may write this relationship as $ \alpha \to \beta $.

\begin{definition} 
\label{defn:ord-local-fin}
We say $ \Lambda $ is \textbf{locally finite} if
\begin{enumerate}
\item for any $ \alpha \in \Lambda $, $ \{x \in \Lambda \;:\; x \leq \alpha\} $ is a finite set,
\item $ \Lambda $ has finitely many minimal elements,
\item for any $ \alpha \in \Lambda $, the element $ \alpha $ has finitely many upper neighbors.
\end{enumerate}
\end{definition}

\begin{figure}[H] \centering
\begin{tikzpicture}[line width=0.5mm]
\draw[->] (-0.5, 0) -- (4.5, 0);
\draw[->] (0, -0.5) -- (0, 4.5);

\filldraw[draw=blue, fill=blue, fill opacity=0.5] (-0.25, -0.25) -- (-0.25, 3.75) -- (0.75, 3.75) -- (0.75, 2.75)
  -- (2.25, 2.75) -- (2.25, 1.75) -- (3.25, 1.75) -- (3.25, 0.75) -- (3.75, 0.75) -- (3.75, -0.25) -- cycle;
\node[blue] at (-0.5, 1.5) {$ A $};

\foreach \y in {0, ..., 9}
  \foreach \x in {0, ..., 9}
    \filldraw (0.5 * \x, 0.5 * \y) circle (1pt);

\begin{scope}[draw=red]
\draw (0.5, 3.5) circle (4pt);
\draw (2, 2.5) circle (4pt);
\draw (3, 1.5) circle (4pt);
\draw (3.5, 0.5) circle (4pt);
\end{scope}

\begin{scope}[draw=green]
\draw (-0.15, 3.85) rectangle (0.15, 4.15);
\draw (0.85, 2.85) rectangle (1.15, 3.15);
\draw (2.35, 1.85) rectangle (2.65, 2.15);
\draw (3.35, 0.85) rectangle (3.65, 1.15);
\draw (3.85, -0.15) rectangle (4.15, 0.15);
\end{scope}

\filldraw (0.5, -0.7) circle (1pt);
\draw[draw=red] (0.5, -0.7) circle (4pt);
\node[red] at (1.3, -0.7) {$ \maxim(A) $};

\filldraw (3, -0.7) circle (1pt);
\draw[draw=green] (2.85, -0.85) rectangle (3.15, -0.55);
\node[green] at (3.85, -0.7) {$ \maxim^*(A) $};

\end{tikzpicture}
\caption{
The maximal elements, $ \maxim(A) $,
and the growth elements, $ \maxim^*(A) $,
for a lower set $ A \subset \Lambda= \mathbb{N}_0^2 $.}
\label{fig:maximal-sets}
\end{figure}
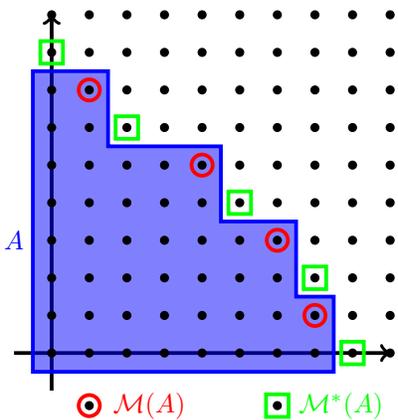

Throughout the article, unless otherwise stated,
we will take $ \Lambda $ to be a locally finite poset.
For example, $ \Lambda = \mathbb{N}_0^d $ is locally finite since $ 0 \in \mathbb{N}_0^d $ is the only minimal element
and every element has $ d $ upper neighbors.
Also, a rooted tree where the degree of a node is finite
but increases, say with its depth, is locally finite.
On the other hand, a tree, where every vertex has an infinite number of vertices below it,
is \emph{not} locally finite.

As before, let $ L(\Lambda) $ be the set of finite lower sets in $ \Lambda $.
Let $ \maxim(A) $ be the set of maximal elements in $ A $
and let $ \maxim^*(A) $ be the set of minimal elements in $ \Lambda \setminus A $.
Note that every element of $ \maxim^*(A) $ will be a minimal element of $ \Lambda $
or an upper neighbor of an element in $ A $.
Note also, when $ \Lambda = \mathbb{N}_0^2 $, that $ \maxim^*(A) = \maxim(A) + 1 $ (as in Figure \ref{fig:maximal-sets});
such a relation does not hold for general $ \Lambda $.

Since $ A $ is finite, $ \maxim(A) $ is finite.
Further, since $ \Lambda $ is locally finite, $ \maxim^*(A) $ is finite.  
Observe that $ \maxim^*(\emptyset) $ consists of the minimal elements of $ \Lambda $.
Also, $ |\maxim(\emptyset)| = 0 $, but for non-empty $ A $, we have $ |\maxim(A)| \geq 1 $
(where $ | \cdot | $ denotes the cardinality).
Also, if $ A \in L(\Lambda) $ and $ A \neq \Lambda $, we have $ |\maxim^*(A)| \geq 1 $.
If $ \Lambda $ is finite then $ \maxim^*(\Lambda) = \emptyset $.

We define a Markov process $ X_t $ on $ L(\Lambda) $,
with respect to rates $ \{\lambda_\alpha\}_{\alpha \in \Lambda} $, with $ \lambda_\alpha \in (0, \infty) $,
so that for any $ A \in L(\Lambda) $ and $ \alpha \in \maxim^*(A) $,
the process $ X_t $ transitions from $ A $ to $ A \cup \alpha $ with rate $ \lambda_\alpha $.
Here, we may think of $ \maxim^*(X_t) $ as the set of `available cells' for $ X_t $ to grow into.
Throughout, we assume $ X_0 = \emptyset $,
though one may consider other initial conditions $ X_0 \in L(\Lambda) $.

For non-empty $ A \subseteq \Lambda $,
let $ \lambda_+(A) := \max_{x \in A} \lambda_x $
and $ \lambda_-(A) := \min_{x \in A} \lambda_x $ be the upper and lower bounds of $ \{\lambda_\alpha\} $ in $ A $.
Since $ \lambda_\alpha \in (0, \infty) $ for $ \alpha \in \Lambda $,
we have for $ A \in L(\Lambda) $ that $ 0 < \lambda_-(A) \leq \lambda_+(A) < \infty $.
When $ A = \emptyset $, our convention is that $ \lambda_-(A) =  \lambda_+(A) = 0 $.
When the set $ A $ is clear from context, we may write $ \lambda_+ = \lambda_+(A) $ and $ \lambda_- = \lambda_-(A) $.

The process $ X_t $ is a continuous time, countable state Markov process with infinitesimal generator
\begin{equation} \label{eq:markov-gener}
(\gener f)(A) := \sum_{\alpha \in \maxim^*(A)} \lambda_\alpha (f(A \cup \alpha) - f(A))
\end{equation}
for functions $ f : L(\Lambda) \to \mathbb{R} $.
We will assume that the rates $ \{\lambda_\alpha\}_{\alpha \in \Lambda} $ are such that 
the process is non-explosive and well-defined for all $ t \geq 0 $;
sufficient condition is that $ \lambda_+(\Lambda) < \infty $, among others.
All of the non-asymptotic statements in the article make use
only of positivity and finiteness of the rates $ 0 < \lambda_\alpha < \infty $ for $ \alpha \in \Lambda $.
However, for the `shape limit' Theorems \ref{thm:lattice-kappa-limit} and \ref{thm:monoid-shape-func},
we will add the assumption that the rates are decreasing, $ \lambda_x \geq \lambda_y $ for $ x \leq y $,
and that $ \lambda_-(\Lambda) > 0 $, which imply uniform bounds $ 0 < \lambda_-(\Lambda) \leq \lambda_+(\Lambda) < \infty $.

The time-reversed Markov process $ X_{-t} $ is a `crystal etching' process
with `backward' generator $ \gener^* =: -\Delta $.
For real functions $ f $ whose domain includes $ A \in L(\Lambda) $
and $ A \setminus \alpha $ for all $ \alpha \in \maxim(A) $,
\begin{equation} \label{eq:defn-delta}
(\Delta f)(A) = \sum_{\alpha \in \maxim(A)} \lambda_\alpha \left( \Big. f(A) - f(A \setminus \alpha) \right) 
\end{equation}
with convention $ (\Delta f)(\emptyset) = 0 $.
Our convention is that empty sums vanish.

We will denote by $ \mathbb{P} $ and $ \mathbb{E} $
the probability measure and expectation when $ X_0 = \emptyset $.
Recall the stopping time $ \tau_A $ defined in eq. (\ref{eq:stop-time-defn}) for $ X_t $ and $ A \in L(\Lambda) $.

\begin{lemma} \label{lem:tau-prob-diff-eq}
For any $ A \in L(\Lambda) $ and $ t \geq 0 $,
\begin{align*}
\frac{\dr}{\dr t} \prob{\tau_A \leq t} = \frac{\dr}{\dr t} \prob{A \subseteq X_t}
& =  \sum_{\alpha \in \maxim(A)} \lambda_\alpha [\prob{\tau_{A \setminus \alpha} \leq t} - \prob{\tau_A \leq t}] \\
& =  -\Delta \prob{\tau_A\leq t},
\end{align*}
where $ A \mapsto \prob{\tau_A \leq t} $ is treated as a function on $ L(\Lambda) $.
\end{lemma}
\begin{proof}
For $ A = \emptyset $, as $ X_0 = \emptyset $, we have $ \prob{\tau_A \leq t} \equiv 1 $ and the statement holds.
Otherwise, we fix non-empty $ A \in L(\Lambda) $ and use $ \indic[\cdot] $ as the indicator of a condition.
We define $ s_A : L(\Lambda) \to \mathbb{R} $ as $ s_A(B) := \indic[A \subseteq B] $.
Then, we can write
\begin{align*}
\frac{\dr}{\dr t} \prob{A \subseteq X_t}
& =  \frac{\dr}{\dr t} \expect{s_A(X_t)}
= \expect{(\gener s_A)(X_t)} \\
& =  \expect{\sum_{\alpha \in \maxim^*(X_t)} \lambda_\alpha \left(\big. s_A(X_t \cup \alpha) - s_A(X_t) \right)} \\
& =  \expect{\sum_{\alpha \in \maxim^*(X_t)} \lambda_\alpha \indic[A \subseteq X_t \cup \alpha, A \nsubseteq X_t]}.
\end{align*}
Notice that $ \{A \subseteq X_t \cup \alpha, A \nsubseteq X_t\} = \{A \setminus \alpha \subseteq X_t, A \nsubseteq X_t\} $.
Also, $ A \subseteq X_t \cup \alpha $ and $ A \nsubseteq X_t $ implies $ \alpha \in \maxim(A) $.
Conversely, when $ \alpha \in \maxim(A) $, if $ A \subseteq X_t \cup \alpha $
and $ A \nsubseteq X_t $ then $ \alpha \in \maxim^*(X_t) $.
Thus,
\begin{align*}
\sum_{\alpha \in \maxim^*(X_t)} \lambda_\alpha \indic[A \subseteq X_t \cup \alpha, A \nsubseteq X_t]
& =  \sum_{\alpha \in \maxim(A)} \lambda_\alpha \indic[A \setminus \alpha \subseteq X_t, A \nsubseteq X_t] \\
& =  \sum_{\alpha \in \maxim(A)} \lambda_\alpha \left( \big. s_{A \setminus \alpha}(X_t) - s_A(X_t) \right).
\end{align*}
Then,
\begin{align*}
\frac{\dr}{\dr t} \prob{A \subseteq X_t}
& =  \expect{\sum_{\alpha \in \maxim(A)} \lambda_\alpha \left(\big. s_{A \setminus \alpha}(X_t) - s_A(X_t) \right)} \\
& =  \sum_{\alpha \in \maxim(A)} \lambda_\alpha \left(\big. \prob{A \setminus \alpha \subseteq X_t} - \prob{A \subseteq X_t}\right).  \qedhere
\end{align*}
\end{proof}

We now demonstrate the integrability of $ f(\tau_A) $
when $ f $ has at most exponential growth, using the tail probabilities of $ \tau_A $.

\begin{corollary} \label{coro:tau-prob-bdd}
For any $ A \in L(\Lambda) $ and all $ t \geq 0 $, we have the bound
\begin{equation} \label{eq:tau-prob-bdd}
\prob{\tau_A > t} \leq (\lambda_+(A) t + 1)^{|A|} e^{-t \lambda_-(A)},
\end{equation}
where $ |A| $ is the cardinality of $ A $.
Further, let $ f : \mathbb{R}_{\geq 0} \to \mathbb{R} $ be such
that there exist $ C, t_0 > 0 $ and $ \mu < \lambda_-(A) $
with $ |f(t)| \leq C e^{\mu t} $ for all $ t \geq t_0 $.
Then, $ \expect{|f(\tau_A)|} < \infty $.
\end{corollary}
\begin{proof}
We will prove the result by induction on $ |A| \geq 0 $.
For $ A = \emptyset $, $ \prob{\tau_A > t} = 0 $ for all $ t \geq 0 $, and so the inequality holds.
When $ |A| \geq 1 $, we consider the derivative of $ e^{t \lambda_-(A)} \prob{\tau_A > t} = e^{t \lambda_-(A)} (1 - \prob{\tau_A \leq t}) $.
Using Lemma \ref{lem:tau-prob-diff-eq} and the fact that $ 1 \leq |\maxim(A)| \leq |A| $, we have
\begin{align*}
\frac{\dr}{\dr t} e^{t \lambda_-(A)} \prob{\tau_A > t}
& =  \lambda_- e^{t \lambda_-} \prob{\tau_A > t} + e^{t \lambda_-} \frac{\dr}{\dr t} \prob{\tau_A > t} \\
& =  \lambda_- e^{t \lambda_-} \prob{\tau_A > t} + e^{t \lambda_-}
\sum_{\alpha \in \maxim(A)} \lambda_\alpha [\prob{\tau_{A \setminus \alpha} > t} - \prob{\tau_A > t}] \\
& \leq  \lambda_- e^{t \lambda_-} \prob{\tau_A > t} - \lambda_- e^{t \lambda_-} \prob{\tau_A > t} \\
&  \qquad + e^{t \lambda_-} \sum_{\alpha \in \maxim(A)} \lambda_\alpha \prob{\tau_{A \setminus \alpha} > t} \\
& \leq  \lambda_+(A) e^{t \lambda_-(A)} \sum_{\alpha \in \maxim(A)} (\lambda_+(A \setminus \alpha) t + 1)^{|A| - 1} e^{-t \lambda_-(A \setminus \alpha)} \\
& \leq  \lambda_+(A) e^{t \lambda_-(A)} |A| \cdot (\lambda_+(A) t + 1)^{|A| - 1} e^{-t \lambda_-(A)} \\
& =  \lambda_+ |A| \cdot (\lambda_+ t + 1)^{|A| - 1} = \lambda_+^{|A|} |A| \left(t + \frac{1}{\lambda_+}\right)^{|A| - 1}.
\end{align*}
Integrating both sides from $ 0 $, we obtain
\begin{equation*}
e^{t \lambda_-} \prob{\tau_A > t} - 1
\leq \lambda_+^{|A|} \left[ \left(t + \frac{1}{\lambda_+}\right)^{|A|} - \frac{1}{\lambda_+^{|A|}} \right]
= (\lambda_+ t + 1)^{|A|} - 1
\end{equation*}
proving eq. (\ref{eq:tau-prob-bdd}).

Next, let $ f : \mathbb{R}_{\geq 0} \to \mathbb{R} $ fulfill the given conditions.
If $ A = \emptyset $ then $ \expect{|f(\tau_A)|} = \expect{|f(0)|} < \infty $.
Otherwise, set $ K = \int_{t=0}^{t_0} |f(t)| \dr \prob{\tau_A \leq t} < \infty $
and note that $ \prob{\tau_A = 0} = 0 $ as $ A \neq \emptyset $ and $ X_0 = \emptyset $.
Then, from Lemma \ref{lem:tau-prob-diff-eq} and eq. (\ref{eq:tau-prob-bdd}), we obtain
\begin{align*}
\expect{|f(\tau_A)|} & =  |f(0)| \prob{\tau_A = 0} - \int_{t=0}^\infty |f(t)| \dr \prob{\tau_A > t} \\
& =  K - \int_{t_0}^{\infty} |f(t)| \frac{\dr}{\dr t} \prob{\tau_A > t} \dr t \\
& =  K - \sum_{\alpha \in \maxim(A)} \lambda_\alpha \int_{t_0}^\infty |f(t)|
\left[ \Big. \prob{\tau_{A \setminus \alpha} > t} - \prob{\tau_A > t} \right] \dr t \\
& \leq  K + \sum_{\alpha \in \maxim(A)} \lambda_\alpha \int_{t_0}^\infty |f(t)| \cdot \prob{\tau_A > t} \dr t \\
& \leq  K + \lambda_+ |A| \int_{t_0}^\infty C e^{t \mu} \cdot (\lambda_+ t + 1)^{|A|} e^{-t \lambda_-} \dr t \\
& =  K + \lambda_+^{|A| + 1} C |A| \int_{t_0}^\infty (t + 1/\lambda_+)^{|A|} e^{t (\mu - \lambda_-)} \dr t < \infty.  \qedhere
\end{align*}
\end{proof}

\subsection{Well-definedness of Functions on $L(\Lambda)$}
\label{subsect:well-defined}
Because of the inhomogeneity of $ \{\lambda_\alpha\}_{\alpha\in \Lambda} $,
we remark that many functions $ A \mapsto \expect{f(\tau_A)} $ of interest
(e.g. those where $ f : \mathbb{R}_{\geq 0} \to \mathbb{R} $ grows exponentially)
may only be defined on a subset of $ \Lambda $.
For instance, for $ f(t) = e^{\mu t} $, the expectation $ \expect{f(\tau_A)} $ is well defined when $ \mu < \lambda_-(A) $,
but if $ \lambda_-(A) \to \mu' < \mu $ as $ A $ grows
then $ \expect{f(\tau_A)} $ will cease to be defined for sufficiently large $ A $.
On the other hand, if $ \mu < \lambda_-(B) $ then we can say $ \expect{f(\tau_A)} $ is well defined for all $ A \in \low{B} $.

So, sometimes our results are stated for restricted-domain functions $ g : \low{D_g} \to \mathbb{R} $
where $ D_g \subseteq \Lambda $ and $ g(A) $ is defined for $ A \in \low{D_g} $.
Of course, the unrestricted case $ g : L(\Lambda) \to \mathbb{R} $, when $g$ is well-defined,
can be recovered using $ D_g = \Lambda $.  
These notions are used in Lemma \ref{lem:diff-deriv-dual} and its applications
as well as in the proofs of Proposition \ref{prop:grtr-path-mgf} and Proposition \ref{prop:lssr-path-mgf}.

\subsection{Relationship to Last Passage Percolation}
\label{sect:lpp-relation}
We now observe the connection between the crystal growth model
and last passage percolation on the locally finite poset $ \Lambda $.
Such a correspondence is well known when $ \Lambda = \mathbb{N}_0^d $.

To this end, we define a \textbf{path} in $ \Lambda $ to be
a (possibly empty) sequence $ \pi = (\pi_1, \ldots, \pi_n) $
where $ \pi_1, \ldots, \pi_n \in \Lambda $, $ \pi_1 $ is minimal,
and for all $ 1 \leq i < n $, $ \pi_i < \pi_{i+1} $ and there exists
no element $ \beta \in \Lambda $ such that $ \pi_i < \beta < \pi_{i+1} $
(i.e. $ \pi_{i+1} $ is an upper neighbor of $ \pi_i $).
Alternatively, for any $ \alpha \in \Lambda $, a path to $ \alpha $
is simply a chain maximal among those bounded by $ \alpha $.
The \textbf{length} of the path $ \pi = (\pi_1, \ldots, \pi_n) $ will be denoted $ \ell(\pi) := n $.
For each $ A \in L(\Lambda) $, we take $ \Pi(A) $ to be
the set of paths $ \pi $ with $ \pi_i \in A $ for all $ i $.
Similarly, we take $ \Pi_m(A) \subseteq \Pi(A) $
to be the set of \textbf{maximal paths} in $ A $.
We also define $ \ell(x) :=  \max \{\ell(\pi) \,:\, \pi \in \Pi_m(\low{x})\} $ for $ x \in \Lambda $.
When $ A = \emptyset $, we have $ \Pi(A) = \Pi_m(A) = \{\emptyset\} $
where $ \emptyset $ is the empty path.
Recall, by convention, empty sums vanish.

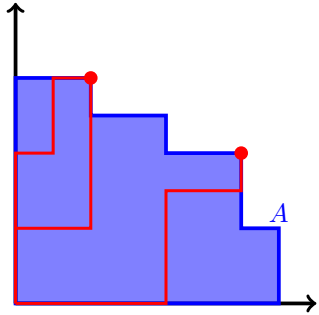
\begin{figure}[H] \centering
\begin{tikzpicture}[line width=0.5mm]

\draw[->] (0, 0) -- (4, 0);
\draw[->] (0, 0) -- (0, 4);

\filldraw[draw=blue, fill=blue, fill opacity=0.5] (0, 3) -- (1, 3) -- (1, 2.5) -- (2, 2.5) -- (2, 2)
  -- (3, 2) -- (3, 1) -- (3.5, 1) -- (3.5, 0) -- (0, 0) --  cycle;
\node[blue] at (3.5, 1.2) {$ A $};

\begin{scope}[draw=red, fill=red, line width=0.4mm]
\draw (0, 0) -- (0, 2) -- (0.5, 2) -- (0.5, 3) -- (1, 3);
\draw (0, 0) -- (0, 1) -- (0.5, 1) -- (1, 1) -- (1, 3);
\filldraw (1, 3) circle (2pt);
\draw (0, 0) -- (2, 0) -- (2, 1.5) -- (3, 1.5) -- (3, 2);
\filldraw (3, 2) circle (2pt);
\end{scope}

\end{tikzpicture}
\caption{Three maximal paths in the lower set $ A \subseteq \Lambda = \mathbb{N}_0^2 $}
\label{fig:paths-in-lower-set}
\end{figure}

\begin{proposition} \label{prop:tau-lpp}
Suppose $ \{G_\alpha\}_{\alpha \in \Lambda} $ is a collection of independent
exponential random variables with rates $ \{\lambda_\alpha\}_{\alpha \in \Lambda} $
and for any $ A \in L(\Lambda) $, define
\begin{equation} \label{eq:tau-lpp-form}
\chi_A := \max_{\pi \in \Pi_m(A)} \sum_{i=1}^{\ell(\pi)} G_{\pi_i}.
\end{equation}
Then, the stochastic process $ Y_t = \bigcup \{A \in L(\Lambda) \,:\, \chi_A \leq t\} $
is a Markov process with generator $ \gener $ from eq. (\ref{eq:markov-gener}).
In particular, $ Y_t $ and $ \chi_A $ have the same distributions as $ X_t $ and $ \tau_A $.
\end{proposition}

We give two proofs of the proposition in Subsection \ref{subsect:proof-tau-lpp},
for the convenience and interest of the reader.

\section{Results}
\label{sect:results}
In the following, we begin with bounds on the
variance $ \var{\tau_A} $, higher moments $ \expect{\tau_A^n} $, and central moments $ \expect{(\tau_A - \expect{\tau_A})^n} $
written in terms of $ \expect{\tau_A} $.
Bounds on the mean $ \expect{\tau_A} $ and moment generating function $ \expect{e^{u \tau_A}} $ are supplied later.
Then, we state a LLN shape limit in a general monoid setting.
Finally, in Subsection \ref{subsect:stoch-less}, we state some extensions to non-exponential weights $ \{G_\alpha\}_{\alpha\in \Lambda} $.

\subsection{Upper Bounds on $ \var{\tau_A} $ and Higher Moments in Terms of $ \expect{\tau_A} $}
\label{sect:mom-bnds}
First, we state a sublinear upper bound for the variance of $ \tau_A $ in terms of its mean.
Estimates of the mean $ \expect{\tau_A} $ are provided in Subsection \ref{subsect:mean-bnd}.

\begin{theorem} \label{thm:var-sublinear}
For any $ A \in L(\Lambda) $, we have $ \expect{\tau_A^2} < \infty $
and $ \lambda_-(A) \cdot \var{\tau_A} \leq \expect{\tau_A} $.
\end{theorem}

Because of Theorem \ref{thm:var-sublinear},
we have $ \var{\tau_A} \leq K \expect{\tau^p_A} $ for $ K = \lambda_-(A)^{-1} $ and $ p = 1 $.
However, it is known that this inequality holds for $ p < 1 $ in certain cases.
Namely, as remarked in the introduction, in the homogeneous two dimensional Euclidean model,
with $ \Lambda = \mathbb{N}_0^2 $ and $ \lambda_\alpha \equiv \lambda $,
the model is exactly solvable.
One obtains for sets $ A = \low{n (e_1 + e_2)} $
that $ \var{\tau_A} \sim n^p \sim \expect{\tau_A}^p $ with $ p = 2/3 $ (see \cite{Baik05}).
Since Theorem \ref{thm:var-sublinear} implies $ \tau_A / \expect{\tau_A} \to 1 $ in probability and $ L^2 $,
by Fatou's lemma, we have $ \liminf \expect{(\tau_A/\expect{\tau_A})^p} = \liminf \expect{\tau_A^p} / \expect{\tau_A}^p \geq 1 $.
So, for large $ n $, one may verify the inequality $ \var{\tau_A} \leq K \expect{\tau_A^p} $.

We now assert that if $ \var{\tau_A} \leq K \expect{\tau_A^p} $ for some $ K > 0 $ and $ p < 1 $
then one can obtain similar bounds on the higher moments of $ \tau_A $. 
Related bounds are obtained for the higher central moments, giving a measure of concentration.

\begin{proposition} \label{prop:high-mom-bnd}
Suppose that there exist real numbers $ K > 0 $ and $ 0 < p \leq 1 $ such that,
for all $ A \in L(\Lambda) $, we have $ \var{\tau_A} \leq K \expect{\tau_A^p} $.
For each integer $ n \geq 0 $ and every $ A \in L(\Lambda) $, we have
\begin{equation} \label{eq:high-mom-bnd}
\expect{\tau_A^n} - \expect{\tau_A}^n \leq K \frac{n (n - 1)^2}{2} \expect{\tau_A^{p + n - 2}}. 
\end{equation}
\end{proposition}

\begin{corollary} \label{coro:central-moms-bnd}
Suppose there exists $ K > 0 $ and $ 0 < p \leq 1 $
such that $ \var{\tau_A} \leq K \expect{\tau_A^p} $ for all $ A \in L(\Lambda) $.
Then, for every $ n $ and every $ A \in L(\Lambda) $, the central moment is bounded by
\begin{equation*}
\left|\big. \expect{(\tau_A - \expect{\tau_A})^n}\right| \leq K \frac{n (n - 1)^2}{2}
\expect{\tau_A^p \left(\tau_A + \expect{\tau_A}\right)^{n-2}}.
\end{equation*}

Also, for any sequence $ A_1, A_2, \ldots \in L(\Lambda) $
with $ \mu_j := \expect{\tau_{A_j}} \to \infty $ as $ j \to \infty $,
we have $ \expect{(\tau_{A_j} - \mu_j)^n} = O(\mu_j^{p + n - 2}) $ as $ j \to \infty $.
\end{corollary}

Note that for $ p = 1 $, the assumption $ \var{\tau_A} \leq K \expect{\tau_A^p} $
with $ K = 1 / \lambda_-(A) $ follows from Theorem \ref{thm:var-sublinear}
and so is unnecessary.

\subsection{Lower Bound on $ \var{\tau_A} $}

We can also obtain lower bounds for the variance
as long as one can limit the growth rate of the lower sets.
In particular, if we know that the number of maximal elements of $ A $ does not grow too quickly
then we can ensure that the variance of $ \tau_A $ does not grow too slowly.

\begin{proposition} \label{prop:var-mxml-bnd}
Suppose there exists an increasing function $ f : \mathbb{R}_{\geq 0} \to \mathbb{R}_{\geq 0} $
such that for all $ A \in L(\Lambda) $, we have $ \sum_{\alpha \in \maxim(A)} \lambda_\alpha \leq f(|A|) $.
Then, for all $ A \in L(\Lambda) $,
\begin{equation*}
\var{\tau_A} \geq \int_1^{|A|+1} \frac{\dr x}{f(x)^2}.
\end{equation*}
\end{proposition}

\begin{corollary} \label{coro:var-away-zero}
There exists $ C > 0 $ such that $ \var{\tau_A} \geq C $ for all non-empty $ A \in L(\Lambda) $.
\end{corollary}
\begin{proof}
Because $ \Lambda $ is locally finite,
every element $ x \in \Lambda $ has $ m_x $ upper neighbors where $ m_x < \infty $.
We define $ B_n := \{x \in \Lambda : \ell(x) \leq n\} $, $ b_n := \sup_{x \in B_n} m_x $,
and prove by induction on $ n \geq 0 $ that $ |B_n|, |b_n|, |\Pi(B_n)| < \infty $.
For $ n = 0 $, we have $ B_0 = \emptyset $ and $ b_0 < \infty $.
Otherwise, $ n \geq 1 $ and every $ x \in B_n $ is the endpoint of a path in $ \Pi(B_n) $,
so $ |B_n| \leq |\Pi(B_n)| $.
Further, any path in $ \Pi(B_n) $ is the extension by one element of a path in $ |\Pi(B_{n-1})| $,
so $ |\Pi(B_n)| < b_{n-1} \cdot |\Pi(B_{n-1})| $.
By the induction hypothesis, $ b_{n-1} < \infty $ and $ |\Pi(B_{n-1})| < \infty $
so that $ |B_n| \leq |\Pi(B_n)| \leq b_{n-1} \cdot |\Pi(B_{n-1})| < \infty $
which implies $ b_n < \infty $ completing the induction.
Because $ B_n $ is finite, $ \low{B_n} $ is finite.

Now, for any $ A \in L(\Lambda) $, if $ |A| \leq n $ then $ \ell(x) \leq |A| \leq n $ for $ x \in A $.
Hence, $ A \subseteq B_n $ or in other words $ A \in \low{B_n} $.
Thus, there are finitely many $ A \in L(\Lambda) $ with $ |A| \leq n $
so we can define a positive increasing function
\begin{equation*}
f(n) = \max_{|A| \leq n} \sum_{\alpha \in \maxim(A)} \lambda_\alpha < \infty.
\end{equation*}
We take $ C = \int_1^2 \frac{\dr x}{f(x)^2} > 0 $.
Then, by Proposition \ref{prop:var-mxml-bnd}, when $ A $ is non-empty, so that $ |A| \geq 1 $, we have 
\begin{equation*}
\var{\tau_A} \geq \int_1^{|A|+1} \frac{\dr x}{f(x)^2} \geq C > 0.  \qedhere
\end{equation*}
\end{proof}

\begin{corollary} \label{coro:var-2d-bnd}
When $ \Lambda = \mathbb{N}_0^2 $ and $ \lambda_+(\Lambda) < \infty $,
then $ \var{\tau_A} \geq \frac{1}{2 \lambda_+^2(\Lambda)} \log(|A| + 1) $ for all $ A \in L(\Lambda) $.
\end{corollary}
\begin{proof}
First, we show $ |\maxim(A)| \leq \sqrt{2 |A|} $ for all $ A \in L(\Lambda) $.
Suppose $ A \in L(\mathbb{N}_0^2) $.
Let $ (x_1, y_1), \ldots, (x_m, y_m) $ be the maximal elements of $ A $
ordered such that $ x_1, \ldots, x_m $ is an ascending sequence.
For any $ i < j $, we observe $ x_i < x_j $, but by maximality $ (x_i, y_i) $
cannot be less than $ (x_j, y_j) $ and so $ y_i > y_j $.
Because $ \{y_j\} $ are integers, we have $ y_{j-1} \geq 1 + y_j $.
Since $ y_m \geq 0 $, by induction, we have $ y_j \geq m - j $.
Note that for every $ 0 \leq b \leq y_i $, the element $ (x_i, b) \in A $ since $ A $ is a lower set.
Thus, taking $ k = m - i + 1 $,
\begin{align*}
|A| & \geq  \left|\bigcup_{i=1}^m \bigcup_{b=0}^{y_i} \{(x_i, b)\} \right|
= \sum_{i=1}^m (y_i + 1) \\
& \geq  \sum_{i=1}^m (m - i + 1)
= \sum_{k=1}^m k = \frac{m (m + 1)}{2}
\geq \frac{1}{2} m^2 = \frac{1}{2} |\maxim(A)|^2
\end{align*}
implying that $ |\maxim(A)|^2 \leq 2 |A| $.
Thus, the result follows from Proposition \ref{prop:var-mxml-bnd} as
\[ \sum_{\alpha \in \maxim(A)} \lambda_\alpha \leq \lambda_+(\Lambda) |\maxim(A)| \leq \sqrt{2} \lambda_+(\Lambda) |A|. \]
\end{proof}

\begin{corollary}
\label{coro:var-diverge}
Suppose $ \lambda_+(\Lambda) < \infty $ and there exists $ b > 0 $
such that $ |\maxim(A)| \leq b $ for all $ A \in L(\Lambda) $.
Then, $ \var{\tau_A} \geq \frac{|A|}{\lambda_+(\Lambda)^2 b^2} $,
and so $ \var{\tau_A} $ diverges linearly in $ |A| $ as $ |A| \to \infty $.
\end{corollary}
\begin{proof}
We define $ f(|A|) := \lambda_+(\Lambda) b \geq \sum_{\alpha \in \maxim(A)} \lambda_\alpha $.
Then, by Proposition \ref{prop:var-mxml-bnd},
\begin{equation*}
\var{\tau_A} \geq \int_1^{|A|+1} \frac{\dr x}{\lambda_+(\Lambda)^2 b^2} = \frac{|A|}{\lambda_+(\Lambda)^2 b^2}.  \qedhere
\end{equation*}
\end{proof}

\begin{example}
\label{exm:low-bnd-poset}
A toy poset $ \Lambda $ satisfying the assumptions of Corollary \ref{coro:var-diverge}
is $ \Lambda = \mathbb{N}_0 \times S $ where $ S $ is a finite set
and for any $ (n_1, s_1), (n_2, s_2) \in \mathbb{N}_0 \times S $,
$ (n_1, s_1) < (n_2, s_2) $ if and only if $ n_1 < n_2 $.
This ordering forces the growth process to fill each layer (e.g. $ \{n\} \times S $) sequentially
so one always has $ |\maxim(A)| \leq |S| $.
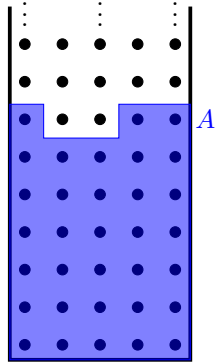
\begin{figure}[H] \centering
\begin{tikzpicture}

\foreach \y in {0, ..., 8}
  \foreach \x in {-2, ..., 2}
    \filldraw (0.5 * \x, 0.5 * \y) circle (2pt);

\node at (-1, 4.5) {$ \vdots $};
\node at (0, 4.5) {$ \vdots $};
\node at (1, 4.5) {$ \vdots $};
\draw[line width=0.5mm] (-1.2, 4.5) -- (-1.2, -0.2) -- (1.2, -0.2) -- (1.2, 4.5);

\filldraw[draw=blue, fill=blue, fill opacity=0.5] (-1.2, -0.2) -- (1.2, -0.2) -- (1.2, 3.2)
  -- (0.25, 3.2) -- (0.25, 2.75) -- (-0.75, 2.75) -- (-0.75, 3.2) -- (-1.2, 3.2) -- cycle;
\node[blue] at (1.4, 3) {$ A $};

\end{tikzpicture}
\caption{The lower set $ A $ in the poset $ \mathbb{N}_0 \times \{1, \ldots, 5\} $
where the ordering is determined by $ \mathbb{N}_0 $ (up is greater)
and $ (n, a) $ is incomparable to $ (n, b) $ whenever $ a \neq b $}
\label{fig:columnar-poset}
\end{figure}
\end{example}

\subsection{Bounds on $ \expect{\tau_A} $}
\label{subsect:mean-bnd}
For any $ A \in L(\Lambda) $, we define the \textbf{length}, $ \ell(A) $, of $ A $
as the maximum length of a path in $ A $:
\begin{equation*}
\ell(A) := \max_{\pi \in \Pi(A)} \ell(\pi) = \max_{\pi \in \Pi_m(A)} \ell(\pi)
\end{equation*}
with $ \ell(\emptyset) := 0 $.
Naturally, one expects that the stopping time $ \tau_A $
would be bounded below in terms of $ \ell(A) $,
since every path in $ A $ must be filled in order for $ A $ to be occupied.

\begin{proposition} \label{prop:mean-lower}
For any $ B \in L(\Lambda) $, we have $ \lambda_+(B) \expect{\tau_B} \geq \ell(B) $.
\end{proposition}

Although simple, this bound is sharp for some posets $ \Lambda $.
In particular, one can take $ \Lambda = \mathbb{N}_0 $
and $ \tau_{[0, n-1]} $ is the sum of $ n $ exponential random variables with rate $ \lambda $.
Hence, $ \expect{\tau_{[0, n-1]}} = \lambda n = \lambda \ell([0, n-1]) $.
One can also consider any model where the poset $ \mathbb{N}_0 $
arises as a lower set of $ \Lambda $, such as when $ \Lambda = \mathbb{N}_0^d $.
However, for more general posets $ \Lambda $,
the lower bound may not be sharp noting the following upper bound.

If we know something about the size of the set $ A $
(i.e. the number of paths in $ A $),
then we can also upper bound the mean $ \expect{\tau_A} $.
To express this bound, define the \textbf{width} $ \kappa(A) $ for any $ A \in L(\Lambda) $
as the logarithm of the number of maximal paths in $ A $
and the \textbf{rate spread} $ \eta(A) $ for any non-empty $ A \subseteq \Lambda $
as the logarithm of the ratio of the maximum and minimum rates $ \lambda_\alpha $ in $ A $:
\begin{equation} \label{eq:kappa-eta-defn}
\kappa(A) := \log |\Pi_m(A)|
\quad \text{and} \quad
\eta(A) := \log\left(\frac{\lambda_+(A)}{\lambda_-(A)}\right).
\end{equation}
We specify that $ \kappa(\emptyset) = \eta(\emptyset) = 0 $.
Note the width and rate spread will typically grow as $ |A| \to \infty $.
However, when $ 0 < \lambda_-(\Lambda) \leq \lambda_+(\Lambda) < \infty $,
we have $ \eta(A) \leq \eta(\Lambda) < \infty $.

\begin{theorem} \label{thm:mean-upper}
For any $ A \in L(\Lambda) $,
\begin{equation*}
\lambda_-(A) \expect{\tau_A} \leq \left(\sqrt{\ell(A)} + \sqrt{\kappa(A) + \eta(A)}\right)^2.
\end{equation*}
\end{theorem}

One can think of $ \kappa(A) / \ell(A) $
as describing a `shape' or `aspect ratio' of $ A $ independent of its size.
To show later `shape limits', one may want to bound $ \kappa(A) $ in terms of $ \ell(A) $
to ensure at most linear growth of the mean $ \expect{\tau_A} $
in terms of $ \ell(A) $.
In general, this may not be possible.
Consider a tree poset where the number of branches for each element
increases as one moves up the tree.
In such settings, though, Theorem \ref{thm:mean-upper} may give a useful bound informed by the geometry.  

\begin{example}
\label{exm:hypercube}
Consider the (Boolean) hypercube $ \Lambda = \{(\epsilon_1, \ldots, \epsilon_n) : \epsilon_i \in \{0,1\} \text{ for } 1 \leq i \leq n\} $
with the usual partial order on $ \mathbb{N}_0^n $.
We focus now on the passage time $ \tau_{\low{(1,\ldots,1)}} $.
Suppose $ \lambda_\alpha \equiv \lambda $ so that $ \eta(\low{(1,\ldots, 1)}) = 0 $.
Note that $ \ell((1, \ldots, 1)) = n $ and $ |\Pi_m(\low{(1,\ldots, 1)})| = n! $.
By Theorem \ref{thm:mean-upper}, we have
$ \expect{\tau_{\low{(1,\ldots, 1)}}} \leq \left(\sqrt{n} + \sqrt{\log n!}\right)^2 $.
Note that the upper bound, of order $ n \log n $,
matches that of the mixing time of a lazy random walk on the hypercube;
see Proposition 7.14 and Example 12.19 in \cite{LP17}.
\end{example}

\begin{lemma} \label{lem:kappa-bnd}
Let $ A \in L(\Lambda) $ and suppose that there exists $ d \in \mathbb{N}_0 $
such that $ A $ has at most $ d $ minimal elements
and for every $ x \in A $, there are at most $ d $ upper neighbors of $ x $ in $ A $.
Then, $ \kappa(A) \leq \log(d) \, \ell(A) $.
\end{lemma}

If we restrict Theorem \ref{thm:mean-upper} to the case
of the $ d $-dimensional lattice where $ \Lambda = \mathbb{N}_0^d $
then we can explicitly calculate $ \kappa(A) $ in terms of $ \ell(A) $
using Stirling's approximation.
In particular, the bounds in Theorem \ref{thm:lattice-kappa-limit} give information on
the shape function $ g(\alpha) = \lim_{n \to \infty} \tau_{\low{n\alpha}}/n $,
known to exist in the homogeneous rate $ \lambda_x \equiv \lambda $ setting (cf. \cite{Mart04}),
but non-explicit in $ d \geq 3 $.
We comment, the limit below extends the known limit in $ d \geq 3 $
for homogeneous weights to an inhomogeneous setting.

\begin{theorem} \label{thm:lattice-kappa-limit}
Suppose $ \Lambda = \mathbb{N}_0^d $ for $ d \geq 1 $, and $ \lambda_-(\Lambda) > 0 $.
For any $ \alpha = (\alpha_1, \ldots, \alpha_d) \neq (0,\ldots, 0)$,
define $ p_i := \frac{\alpha_i}{\ell(\alpha)} $ for each $ i $.  Then,
\begin{equation} \label{eq:kappa-limit}
\lim_{n \to \infty} \frac{\kappa(\langle n \alpha \rangle)}{\ell(\langle n \alpha \rangle)} = -\sum_{i=1}^d p_i \log(p_i),
\end{equation}
\begin{equation} \label{eq:mean-limit}
\limsup_{n \to \infty} \frac{\expect{\tau_{\langle n \alpha \rangle}}}{n \ell(\alpha)}
\leq \frac{1}{\lambda_-(\Lambda) } \left(1 + \sqrt{- \sum_i p_i \log(p_i)}\right)^2.
\end{equation}
Moreover, suppose in addition that $ \lambda_x \geq \lambda_y $ for $ x, y \in \Lambda $ with $ x \leq y $.
Then, by the later Theorem \ref{thm:monoid-shape-func}, the limit of $ \frac{\expect{\tau_{\langle n \alpha \rangle}}}{n \ell(\alpha)} $ exists.
That is, one can replace `$ \limsup $' with `$ \lim $' in eq. (\ref{eq:mean-limit}).
Also, the limit of $ \tau_{\low{n\alpha}} / n $ converges in $ L^2 $.
\end{theorem}

\subsection{Bounds on the Moment Generating Function of $ \tau_A $}
\label{subsect:mom-gen-bnds}
We introduce types of `path functions' that will allow us to estimate the moment generating function of $\tau_A$,
tight up to first order expansion as $ u \to 0^+ $.
Part of the rationale, as will be seen in the proofs,
is that they behave nicely under the operator $ \Delta $.

\begin{definition}
\label{defn:grtr}
For any $ f : \mathbb{N}_0 \to \mathbb{R} $, we define
the \textbf{greater path function} $ \GrtrPath f : L(\Lambda) \to \mathbb{R} $ as
\begin{equation*}
\left(\GrtrPath f \right)(A) := \sum_{\pi \in \Pi(A)} f(\ell(\pi)).
\end{equation*}
\end{definition}

We can apply $ \GrtrPath $ to
the geometric function $ f(n) = r^n $ for a real $ r > 1 $.

\begin{proposition} \label{prop:grtr-path-mgf}
For any $ B \in L(\Lambda) $ and real $ 0 < u < 1 $,
\begin{equation*}
\sum_{\substack{\pi \in \Pi(B) \\ \pi \neq \emptyset}} \frac{u}{(1 - u)^{\ell(\pi)}}
= u \left(\GrtrPath
\left\{\frac{1}{(1 - u)^n}\right\}(B) - 1\right)
\geq \expect{e^{\lambda_-(B) \, u \tau_B}} - 1.
\end{equation*}
\end{proposition}

We now construct a lower bound for the moment generating functions.
To do this, though, we must use information about how quickly $ \Lambda $ branches.

\begin{definition}
 \label{defn:branch-alloc}
We define a \textbf{branching allocation} $ \psi $ to be
a collection of non-negative real numbers
\begin{equation*}
\{\psi_{\alpha \to \beta} \in \mathbb{R}_{\geq 0} \,:\,
\alpha, \beta \in \Lambda, \alpha \text{ is a lower neighbor of } \beta\}
\cup \{\psi_\mu \,:\, \mu \in \Lambda \text{ is minimal}\}
\end{equation*}
such that for each $ \alpha $, $ \sum_\beta \psi_{\alpha \to \beta} \leq 1 $
and $ \sum_{\mu \text{ is minimal}} \psi_\mu \leq 1 $.
Given a branching allocation $ \psi $, we say
the \textbf{weight} of a path $ \pi \in \Pi(\Lambda) $
with respect to a branching allocation $ \psi $ is
\begin{equation*}
\omega_\psi(\pi) := \psi_{\pi_1} \prod_{i=2}^{\ell(\pi)} \psi_{\pi_{i-1} \to \pi_i}
\end{equation*}
and $ \omega_\psi(\emptyset) = 1 $.
Then, for any $ f : \mathbb{N}_0 \to \mathbb{R} $, we define
the \textbf{lesser path function} $ \LssrPath^\psi f : L(\Lambda) \to \mathbb{R} $ as
\begin{equation*}
\left(\LssrPath^\psi f\right)(A) := \sum_{\pi \in \Pi(A)} f(\ell(\pi)) \omega_\psi(\pi).
\end{equation*}
when clear from context, we may omit $ \psi $ and simply write $ \LssrPath f $.
\end{definition}

As with the greater path function,
we can consider the lesser path functions for $ f(n) = r^n $.

\begin{proposition} \label{prop:lssr-path-mgf}
Let $ \psi $ be a branching allocation.
Then, for non-empty $ B \in L(\Lambda) $ and $0<u<\lambda_-(B)/\lambda_+(B)$,
\begin{equation*}
u \left(\LssrPath^\psi
\left\{\frac{1}{(1 - u)^n}\right\}(B) - 1 \right)
\leq \expect{e^{\lambda_+(B) \, u \tau_B}} - 1.
\end{equation*}
\end{proposition}

We remark that the upper bound Proposition \ref{prop:grtr-path-mgf}
is used to make the upper bounds of $ \expect{\tau_A} $ in Theorem \ref{thm:mean-upper}.
On the other hand, the companion lower bound Proposition \ref{prop:lssr-path-mgf}
can be used to give an alternate proof of Proposition \ref{prop:mean-lower} (see Remark \ref{rem:lssr-path-mean-lower}).

\subsection{Partial Orders on Monoids}
\label{sect:monoids}
When $ \Lambda = \mathbb{N}_0^d $, we notice that every element of $ \Lambda $ `looks like' every other element of $ \Lambda $.
In particular, for any $ \alpha \in \Lambda $, we have the transformation $ x \mapsto x + \alpha $
(inherited from the group structure of $ \mathbb{Z} $)
which carries $ 0 \in \Lambda $ and its upper ray $ [0, \infty)_\Lambda = \{x \in \Lambda \,:\, x \geq 0\} $
to $ \alpha $ and its upper ray $ [\alpha, \infty)_\Lambda $.
This transformation structure gives natural sequences of elements and hence lower sets on which to consider their stopping times.
In particular, for any $ A \in L(\mathbb{N}_0^d) $,
we may define $ nA := \{(nx_1, \ldots, nx_d) \,:\, (x_1, \ldots, x_d) \in \mathbb{N}_0^d\} $
and consider the behavior of $ \expect{\tau_{nA}}/n $ as $ n \to \infty $.
The key ingredient here is the order preserving transformation provided by translation.

To generalize, let $ \overline{\Lambda} $ be a group
with a partial ordering $ \leq $.
We say that $ \leq $ is \textbf{compatible} with $ \overline{\Lambda} $ if
for all $ a_1, a_2, b_1, b_2 \in \overline{\Lambda} $,
$ a_1 \leq a_2 $ and $ b_1 \leq b_2 $ implies $ a_1 b_1 \leq a_2 b_2 $.
Then, we define $ \Lambda := \{x \in \overline{\Lambda} \,:\, x \geq 1_{\overline{\Lambda}} \} $.
where $ 1_{\overline{\Lambda}} = 1_\Lambda $ is the identity of $ \overline{\Lambda} $ and $ \Lambda $.
So, being a subset of a group, $ \Lambda $ naturally has the structure of a monoid
(i.e. a set with an associative binary operation and an identity for that operation)
with a partial order.

For the remainder of the section,
we assume that $ \overline{\Lambda} $ is a group with a compatible partial ordering
and $ \Lambda \subseteq \overline{\Lambda} $ is the non-negative cone
so that $ \Lambda := \{x \geq 1_{\Lambda} \,:\, x \in \overline{\Lambda}\} $.
In the later results, we will assume that $ \Lambda $ is locally finite.
We comment that the second condition of Definition \ref{defn:ord-local-fin}
automatically holds as $ 1_\Lambda $ is the unique minimal element in $ \Lambda $.

The length $ \ell(A) $ of $ A \in L(\Lambda) $ may appear superficially similar
to the diameter (as are defined in metric spaces).
However, as seen in Lemma \ref{lem:tau-super-add}, the mapping$A\mapsto \ell(A) $ is superadditive, contradicting this analogy.
We define a related notion of `length', which however is subadditive (cf. eq. (\ref{eq:ell-star-subadd})):
for every $ A \in L(\Lambda) $, let $ \ell_*(x):= \min\{\ell(\pi) \,:\, \pi \in \Pi_m(\low{x})\} $ and
\begin{equation*}
\ell_*(A) := \max_{x \in A} \ell_*(x).
\end{equation*}
In words, $ \ell_*(x) $ is the smallest number of steps to get to $ x $,
and $ \ell_*(A) $ is the largest such number over $ x \in A $.
Note $ \ell_*(\emptyset) = 0 $.
Since $ \ell(x) =  \max \{\ell(\pi) \,:\, \pi \in \Pi_m(\low{x})\} $
and $ \ell(A) = \max_{x \in A} \ell(x) $, clearly $ \ell_*(A) \leq \ell(A) $.

\begin{definition} 
\label{defn:monoid-set-mult}
For any two lower sets $ A, B \in L(\Lambda) $,
we define the set 
\begin{equation*}
AB := \{x \in \Lambda \,:\, \exists a \in A, b \in B \,\st\, x \leq ab\}.
\end{equation*}
Later in Lemma \ref{lem:set-mult-well-defined}, we show $ AB \in L(\Lambda) $ and $ A^n \in L(\Lambda) $.
Then, for any $ n \in \mathbb{N}_0 $, we can define, recursively, $ A^n := A A^{n-1} $.
\end{definition}

We will now state the existence of a shape function
to which $ \tau_{A^n} / n $ converges as $ n \to \infty $.
A sufficient condition we impose is that $ \ell(A) $ does not have faster than linear growth.

For any monoid $ \Lambda $, we say $ \Lambda $ is \textbf{steady}
if there exists $ C \geq 1 $ such that for all $ x \in \Lambda $, we have
\begin{equation*}
\max_{\pi \in \Pi_m(\low{x})} \ell(\pi) \leq C \min_{\pi \in \Pi_m(\low{x})} \ell(\pi).
\end{equation*}
So, steadiness implies that for all $ A \in L(\Lambda) $, $ \ell(A) \leq C \ell_*(A) $.
Because of the subadditivity of $ \ell_*(A) $ (shown later in Lemma \ref{lem:tau-super-add}),
steadiness ensures that $ \ell(A^n) $ has at most linear growth in $ n $.

There are several immediate examples.
For instance, referring to Example \ref{exm:lattice-ord}, in the Euclidean case, when $ \Lambda = \mathbb{N}_0^d $,
or when $ \Lambda $ are the words constructed from $ d $ generators, 
we have $ \ell(\alpha) = \ell_*(\alpha) $ since, for any element $ \alpha \in \Lambda $,
every path from the minimal element to $ \alpha $ has the same length.
See, however, Remark \ref{rem:non-steady} for an example of `non-steadiness'.

\begin{theorem} \label{thm:monoid-shape-func}
Suppose the monoid $ \Lambda $ is locally finite and steady,
and also $ \lambda_-(\Lambda) > 0 $.
Further, assume that $ \lambda_x \geq \lambda_y $ for any $ x, y \in \Lambda $ with $ x \leq y $.
Then, for any non-empty $ A \in L(\Lambda) $,
the following limit converges and
fulfills the inequality:
\begin{equation} \label{eq:shape-ineq}
g(A) := \lim_{n \to \infty} \frac{\expect{\tau_{A^n}}}{n}
\leq \frac{1}{\lambda_-(\Lambda)} \left(\sqrt{\lim_{n \to \infty} \frac{1}{n} \kappa(A^n)}
+ \sqrt{\lim_{n \to \infty} \frac{1}{n} \ell(A^n)}\right)^2.
\end{equation}
Additionally, $ \frac{1}{n} \tau_{A^n} \to g(A) $ in $ L^2 $ as $ n \to \infty $.
\end{theorem}

Note that trivially, as $ \tau_\emptyset = 0 $, we have $ \tau_{\emptyset^n} / n \equiv 0 $.
Also, if $ \Lambda $ is finite then $ \tau_{A^n} \leq \tau_\Lambda < \infty $ almost surely
and so $ \tau_{A^n} / n \to 0 $.
In addition, recall the `width' $ \kappa(A) $ defined in eq. (\ref{eq:kappa-eta-defn})
is the logarithm of the number of maximal paths in $ A $.
If $ S $ is the finite set of upper neighbors of $ 1_\Lambda $,
then one can show that every element of the monoid $ \Lambda $
has at most $ |S| $ upper neighbors (see the beginning of the proof of Theorem \ref{thm:monoid-shape-func}).
Also, as $ \Lambda $ has a unique minimal element, namely $ 1_\Lambda $,
we may apply Lemma \ref{lem:kappa-bnd} to bound $ \kappa(A) $ in terms of $ \ell(A) $,
giving a further bound to eq. (\ref{eq:shape-ineq}).

We say the group $ \overline{\Lambda} $ is \textbf{finitely presented}
if there exist a finite set of generators $ S $
and a finite set of relations $ R $ among them
such that $ \overline{\Lambda} = \langle S \,|\, R \rangle $ (cf. \cite{Jaco12}).
The group $ \langle S \,|\, R \rangle $ is the collection of equivalence classes of
words consisting of elements of $ S $ and $ S^{-1} $
where the equivalence relation is determined by reduction by the relations $ R $.
With respect to the compatible partial ordering $ \leq $, we will assume the corresponding
partially ordered monoid  $ \Lambda = \{x \in \overline{\Lambda} \,:\, x \geq 1_{\overline{\Lambda}}\} $
is finitely generated. 
In this case, steadiness of $ \Lambda $ implies local finiteness.

\begin{lemma} \label{lem:steady-local-fin}
If $ \Lambda $ is finitely generated and steady then $ \Lambda $ is locally finite.
\end{lemma}

\begin{example} 
\label{exm:fin-prsnt}
Both the Euclidean and free generator settings in Example \ref{exm:lattice-ord}
correspond to finitely presented $ \overline{\Lambda} $ and finitely generated $ \Lambda $ which are also steady.

Another example on the lattice may be described by
its group presentation: $ S = \{a, b, c\} $ and $ R = \{bc = a^3\} $.
Then, let $ \overline{\Lambda} = \langle a, b, c \,|\, bc = a^3 \rangle $ be the abelian group generated by $ S $.
Further, let $ \Lambda = \{x \in \overline{\Lambda}: x \geq 0 = 1_{\overline{\Lambda}}\} \subseteq \overline{\Lambda} $ be
the associated abelian monoid generated by $ \{a, b, c\} $ subject to $ bc = a^3 $.
One can see that $ \overline{\Lambda} \cong \mathbb{Z}^2 $
and $ \Lambda $ will have the structure of a poset:
the cone in Figure \ref{fig:cone-in-lattice} with the binary operation written additively and
\begin{equation*}
a = e_1 + e_2, \quad b = 2 e_1 + e_2, \quad \text{and} \quad c = e_1 + 2 e_2.
\end{equation*}
Here, it may be seen that $ \ell(x) > \ell_*(x) $.
Indeed, this is the case say for $ x = b + c = 3 a $ where $ \ell(x) = 3 $ and $ \ell_*(x) = 2 $.
Nevertheless, $ \Lambda $ is steady as seen by Corollary \ref{coro:fin-prsnt-steady}.
\end{example}

We consider a generalization of the last example in the following statement.

\begin{corollary} \label{coro:fin-prsnt-steady}
Suppose $ S $ is a finite set of generators and $ R $ is a finite set of relations on $ S $.
Assume $ \overline{\Lambda} = \langle S \,| R \rangle $ is abelian and $ \Lambda $ is generated by $ S $.
Then, $ \Lambda $ is steady and locally finite
so that for any $ A \in L(\Lambda) $, eq. (\ref{eq:shape-ineq}) holds.
\end{corollary}

We note in passing that $ \overline{\Lambda} $ may be a finitely presented group,
while $ \Lambda $ is not finitely generated.
Indeed, for any real $ \omega > 0 $, take $ \overline{\Lambda} = \mathbb{Z}^2 $
and $ \Lambda = \{(x, y) \in \mathbb{Z}^2 \,:\, x \geq 0 \text{ and } 0 \leq y < \omega x\} $.
Assume for sake of contradiction that $ \Lambda $ had a finite generating set $ \{(x_k, y_k)\}_{k=1}^n $.
Then, $ q := \max_{1 \leq k \leq n} \frac{y_k}{x_k} < \omega $ so, by density of the rationals,
there exist integers $ x, y > 0 $ with $ 0 \leq q < \frac{y}{x} < \omega $.
Thus, $ y < \pi x $ so $ (x, y) \in \Lambda $ and $ (x, y) = \sum_{k=1}^n c_k (x_k, y_k) $ with $ c_k \geq 0 $.
Then, because $ y_k \leq q x_k $, $ y / x = (\sum c_k y_k) / (\sum c_k x_k) \leq (\sum c_k q x_k) / (\sum c_k x_k) = q $
which is a contradiction.

If one places some restrictions on the relations of a non-abelian group
then we can also achieve steadiness.

\begin{corollary} \label{coro:non-comm-steady}
Suppose the (possibly non-abelian) monoid $ \Lambda $ is finitely generated by $ S $
and suppose $ \overline{\Lambda} = \left\langle S \,|\, R \right\rangle $
where $ R $ is finite (possibly empty).
Further, assume that any relation $ (a_1 \ldots a_n = 1_{\overline{\Lambda}}) \in R $
is such that $ |\{i \in [1, n] \,:\, a_i \in S\}| = |\{i \in [1, n] \,:\, a_i \in S^{-1}\}| $.
Then, $ \Lambda $ is steady with $ C = 1 $ and locally finite
so that for any $ A \in L(\Lambda) $, eq. (\ref{eq:shape-ineq}) holds.
\end{corollary}

One may fit the non-abelian free group setting in Example \ref{exm:lattice-ord} (without relations)
into the assumptions of Corollary \ref{coro:non-comm-steady}.
Another simple example is $ \overline{\Lambda} = \langle a, b, c \,|\, abc = cab = bca \rangle $
(the group which allows cyclic permutations of $ a $, $ b $, and $ c $)
and $ \Lambda \subseteq \overline{\Lambda} $ to be the monoid generated by $ \{a, b, c\} $.
The relations $ abc = cab = bca $ may be written as $ abc b^{-1} a^{-1} c^{-1} = abc a^{-1} c^{-1} b^{-1} = 1 $,
fulfilling the conditions of Corollary \ref{coro:non-comm-steady}.

\begin{remark}[On Non-steadiness]
\label{rem:non-steady}
If the relations on a noncommutative group are not so well behaved
then one can display groups where the minimum and maximum length paths
to an element $ x \in \Lambda $ differ quite substantially.
Take $ \overline{\Lambda} = \langle a, b, c \,|\, ab = bac^2, ac = ca, bc = cb \rangle $
so that $ x_n := (ab)^n = b^n a^n c^{n (n + 1)} $.
Then, the minimum length of a path to $ x_n $ (i.e. $ \ell_*(x_n) $) is at most $ 2n $,
while the maximum length of a path (i.e. $ \ell(x_n) $) is at least $ n^2 + 3n $
and $ \frac{n^2 + 3n}{2n} \to \infty $ as $ n \to \infty $.
Hence, the associated poset $ \Lambda $ is not steady.

Yet, we remark that $ \Lambda $ is locally finite.
Since $ \Lambda $ is finitely generated, by the proof of Lemma \ref{lem:steady-local-fin}, it suffices to show that $ \ell(x) $ is finite for $ x \in \Lambda $.
To this end, if $ \pi $ is a path to $ x $ then $ \pi $ generates a word for $ x $.
When this word is put into the `normal' form, $ x = c^{n_1} b^{n_2} a^{n_3} $,
using the relations $ a b = b a c^2 $, $ a c = c a $, $ b c = c b $,
we must have $ n_1 + n_2 + n_3 \geq \ell(\pi) $.
Suppose there are normal form words for $ x $ whose lengths $ n_1 + n_2 + n_3 $ diverge.
Then, there exist two words $ x = c^{n_1} b^{n_2} a^{n_3} = c^{m_1} b^{m_2} a^{m_3} $
where $ n_i \geq m_i $ for $ i = 1, 2, 3 $ with at least one difference $ n_i - m_i > 0 $.
Hence, by cancelling powers using the relations $ a c = c a $ and $ b c = c b $,
one observes that $ 1_\Lambda = c^{n_1-m_1} b^{n_2-m_2} a^{n_3-m_3} > 1_\Lambda $ which is a contradiction.
Thus, there must exist a normal form with maximal finite path length $ n_1 + n_2 + n_3 < \infty $
which bounds the length of any path to $ x $.
\end{remark}

\subsection{Extension to Stochastically Monotone Weights}
\label{subsect:stoch-less}

We comment here on immediate extensions of some of the results
to non-exponential independent weights $ \{G_\alpha\}_{\alpha \in \Lambda} $.
We begin with the bounding of the moments and moment generating function.

\begin{remark} 
\label{rem:lpp-stoch-less}
Suppose $ \{G_\alpha\}_{\alpha\in \Lambda} $ and $ \{H_\alpha\}_{\alpha\in \Lambda} $ are collections of independent random variables
such that $ G_\alpha \preceq H_\alpha $ for every $ \alpha \in \Lambda $
(where $ \preceq $ indicates stochastic ordering).
For every $ A \in L(\Lambda) $, let $ \tau^{(G)}_A $ and $ \tau^{(H)}_A $ be the stopping times
defined by eq. (\ref{eq:tau-lpp-form}) associated with each collection.
Then, because maximum and summation are monotonic,
we have $ \tau^{(G)}_A \preceq \tau^{(H)}_A $ for every $ A $.
Thus, for every $ k \geq 1 $, we have $ \expect{\left(\tau^{(G)}_A\right)^k} \leq \expect{\left(\tau^{(H)}_A\right)^k} $.
Similarly, for every $ u > 0 $, we have $ \expect{e^{u \tau^{(G)}_A}} \leq \expect{e^{u \tau^{(H)}_A}} $.

In particular, if $ \{G_\alpha\}_{\alpha\in \Lambda} $ is a collection of independent random variables
and $ \{\lambda_\alpha\}_{\alpha\in \Lambda} $ is a collection of positive reals
such that for every $ \alpha $, $ G_\alpha $ is stochastically less than
an exponential distribution with rate $ \lambda_\alpha $,
then the upper bounds of the mean and exponential moments
in Theorem \ref{thm:mean-upper}, Theorem \ref{thm:lattice-kappa-limit} and Proposition \ref{prop:grtr-path-mgf}
(using $ \{\lambda_\alpha\}_{\alpha \in \Lambda} $ as the rates)
apply to the stopping times $ \tau^{(G)}_A $ as well.
Correspondingly, if for all $ \alpha $, $ G_\alpha $ is stochastically greater than
an exponential distribution with rate $ \lambda_\alpha $
then the lower bound for the mean and exponential moment Proposition \ref{prop:mean-lower} and Proposition \ref{prop:lssr-path-mgf} apply.
\end{remark}

We now discuss extension of part of the LLN shape limit theorem Theorem \ref{thm:monoid-shape-func}.

\begin{remark} 
\label{rem:monoid-shape-stoch-less}
Suppose $ \{G_\alpha\} $ is a collection of independent random variables,
but not necessarily exponentially distributed.
Instead, we assume that there exists $ \lambda > 0 $
such that for all $ \alpha \in \Lambda $, the variable $ G_\alpha $ is stochastically less
than an exponential distribution with rate $ \lambda $.
Further, assume that $ G_x \preceq G_y $ for any $ x, y \in \Lambda $ with $ x \leq y $.

Because the proof of Lemma \ref{lem:tau-super-add} only relies
on the stochastic ordering of the weights $ \{G_\alpha\}_{\alpha\in \Lambda} $
(as in Remark \ref{rem:lpp-stoch-less}),
it is still true that $ A \mapsto \expect{\tau^{(G)}_A} $ is superadditive.
By Remark \ref{rem:lpp-stoch-less}, as $ \{G_\alpha\}_{\alpha\in \Lambda} $
are all stochastically less than an exponential distribution,
the upper bound on the mean from Theorem \ref{thm:mean-upper}
applies to $ \expect{\tau^{(G)}_A} $ as well.
Hence, the limit in eq. (\ref{eq:shape-ineq}) will converge
and fulfill the given inequality.
\end{remark}

Finally, we comment that these stochastic ordering arguments, however, cannot be easily applied
to the bounds of variances, such as that of $ \var{\tau_A} $ in Theorem \ref{thm:var-sublinear} for instance.
In particular, the $ L^2 $ convergence stated in Theorem \ref{thm:monoid-shape-func}
may not be applicable to $ \tau^{(G)}_A $.

\section{Properties of the Backward Operator $ \Delta $}
\label{sect:diff-oper}
Recall the operator $ \Delta $ given in eq. (\ref{eq:defn-delta}).
We derive several key relations with respect to well-defined functions
(see Subsection \ref{subsect:well-defined}) useful in the sequel.
The first result computes the action of $ \Delta $ on a function $ A \mapsto \expect{f(\tau_A)} $
in terms of the derivative of $ f $.

\begin{lemma} \label{lem:diff-deriv-dual}
Let $ f \in C^1(\mathbb{R}_{\geq 0}) $ be a continuously differentiable function
such that there exist $ D_f \in L(\Lambda) $,
$ C, t_0 > 0 $, and $ \mu < \lambda_-(D_f) $
with $ |f'(t)| \leq C e^{t \mu} $ for all $ t \geq t_0 $.
Then, $ A \mapsto \expect{f(\tau_A)} $ and $ A \mapsto \expect{f'(\tau_A)} $
are well-defined real functions on $ \low{D_f} $
and for all non-empty $ A \in \low{D_f} $, we have
\begin{equation*}
\Delta \expect{f(\tau_A)} = \expect{f'(\tau_A)}.
\end{equation*}
\end{lemma}
\begin{proof}
First, from Corollary \ref{coro:tau-prob-bdd}, we know $ \expect{f'(\tau_A)} $ exists for all $ A \subseteq D_f $.
Similarly, notice that if $ K = |f(0)| + \int_0^{t_0} |f'(t)| \dr t $
then for any $ t \geq t_0 $,
\begin{align*}
|f(t)| & \leq  |f(0)| + \int_0^t |f'(s)| \dr s
\leq K + C \int_{t_0}^t e^{\mu s} \dr s \\
& \leq  K + \frac{C}{\mu} e^{\mu (t - t_0)}
\leq \left(K + \frac{C}{\mu}\right) e^{\mu (t - t_0)}
\end{align*}
implying that $ \expect{f(\tau_A)} $ also exists for all $ A \subseteq D_f $.
We will write $ p_A(t) := \prob{\tau_A \leq t} $.
Then, using Lemma \ref{lem:tau-prob-diff-eq} and noting $ p_A(0) = 0 $ for non-empty $ A $, we can write
\begin{align*}
\expect{f'(\tau_A)} & =  f'(0) p_A(0) + \int_0^\infty f'(t) \dr p_A(t)
= \int_0^\infty f'(t) \frac{\dr}{\dr t} p_A(t) \dr t \\
& =  \sum_{\alpha \in \maxim(A)} \lambda_\alpha \int_0^\infty f'(t) \left[ \big. p_{A \setminus \alpha}(t) - p_A(t) \right] \dr t \\
& =  \sum_{\alpha \in \maxim(A)} \lambda_\alpha \Bigg[
	f(t) [p_{A \setminus \alpha}(t) - p_A(t)] \Big|_{t=0}^\infty \\
&   \qquad - \int_0^\infty f(t) \dr p_{A \setminus \alpha}(t) + \int_0^\infty f(t) \dr p_A(t)
\Bigg] \\
& =  \sum_{\alpha \in \maxim(A)} \lambda_\alpha \Bigg[
\left(f(0) p_A(0) + \int_0^\infty f(t) \dr p_A(t)\right) \\
&   \qquad - \left(f(0) p_{A \setminus \alpha}(0) + \int_0^\infty f(t) \dr p_{A \setminus \alpha}(t)\right)
\Bigg] \\
& =  \sum_{\alpha \in \maxim(A)} \lambda_\alpha \left[ \Big. \expect{f(\tau_A)} - \expect{f(\tau_{A \setminus \alpha})} \right]
= \Delta \expect{f(\tau_A)}.  \qedhere 
\end{align*}
\end{proof}

Next, we derive a product rule for the backward operator $ \Delta $.
In order to state this relation, for any two functions $ f, g : \low{D} \to \mathbb{R} $
where $ D \in L(\Lambda) $ or $ D = \Lambda $,
define the \textbf{quadratic covariance} of $ f $ and $ g $ as
\begin{equation*}
[f, g](A) := \sum_{\alpha \in \maxim(A)} \lambda_\alpha \left[\Big. f(A) - f(A \setminus \alpha)\right]
\cdot \left[\Big. g(A) - g(A \setminus \alpha)\right]
\end{equation*}
for all non-empty $ A \subseteq D $ and $ [f, g](\emptyset) = 0 $.

\begin{lemma} \label{lem:diff-leibniz}
For any $ f, g : \low{D} \to \mathbb{R} $ and any $ A \subseteq D \in L(\Lambda) $, we have
\begin{equation*}
\Delta(f g)(A) = f(A) \cdot (\Delta g)(A) + (\Delta f)(A) \cdot g(A) - [f, g](A).
\end{equation*}
\end{lemma}
\begin{proof}
When $ A = \emptyset $, we have $ \Delta(fg)(A) = (\Delta f)(A) = (\Delta g)(A) = [f, g](A) = 0 $.
Otherwise, we can write
\begin{align*}
\Delta(f g)(A) + [f, g](A)
& =  \sum_{\alpha \in \maxim(A)} \lambda_\alpha \Big[
f(A) g(A) - f(A \setminus \alpha) g(A \setminus \alpha) \\
& +  f(A) g(A) - f(A) g(A \setminus \alpha) - f(A \setminus \alpha) g(A) + f(A \setminus \alpha) g(A \setminus \alpha)
\Big] \\
& =  \sum_{\alpha \in \maxim(A)} \lambda_\alpha \Big[
f(A) g(A) - f(A) g(A \setminus \alpha) \\
&  \quad + f(A) g(A) - f(A \setminus \alpha) g(A)
\Big] \\
& =  f(A) \cdot (\Delta g)(A) + (\Delta f)(A) \cdot g(A).  \qedhere 
\end{align*}
\end{proof}

We will now state the main vehicle for our results,
a difference inequality on well-defined functions.  
We give an analytic argument, although a `martingale' style proof can also be envisioned.

\begin{proposition} \label{prop:diff-ineq}
Consider functions $ f, g : \low{D} \to \mathbb{R} $ with $ D \in L(\Lambda) $
and $ \varphi : \low{D} \times \mathbb{R} \to \mathbb{R} $
such that for non-empty $ A \subseteq D $ and for all $ x, y \in \mathbb{R} $ with $x\neq y$,
we have $ |\varphi(A, x) - \varphi(A, y)| < |x - y| \sum_{\alpha \in \maxim(A)} \lambda_\alpha $.
Suppose $ f(\emptyset) \geq g(\emptyset) $ and for non-empty $ A \subseteq D $ that
\begin{equation*}
(\Delta f)(A) \geq \varphi(A, f(A))
\quad \text{and} \quad
(\Delta g)(A) \leq \varphi(A, g(A)).
\end{equation*}
Then, for all $ A \subseteq D $, we have $ f(A) \geq g(A) $.  
In particular, if $ f(\emptyset) \geq g(\emptyset) $ and for non-empty $ A \subseteq D $,
we have $ (\Delta f)(A) \geq (\Delta g)(A) $, then $ f(A) \geq g(A) $ for all $ A \subseteq D $.
\end{proposition}
\begin{proof}
We will prove by induction on $ |A| $.
If $ |A| = 0 $ then $ A = \emptyset $ and $ f(A) = f(\emptyset) \geq g(\emptyset) = g(A) $.
Otherwise, we consider when $ |A| \geq 1 $.
By the induction hypothesis, we know $ f(A \setminus \alpha) \geq g(A \setminus \alpha) $
for all $ \alpha \in \maxim(A) $.
Then, take $ s := \sum_{\alpha \in \maxim(A)} \lambda_\alpha $ and notice that
\begin{equation*}
(\Delta f)(A) = \sum_{\alpha \in \maxim(A)} \lambda_\alpha [f(A) - f(A \setminus \alpha)]
= s f(A) - \sum_{\alpha \in \maxim(A)} \lambda_\alpha f(A \setminus \alpha).
\end{equation*}
If $ f(A) = g(A) $ then $ f(A) \geq g(A) $.
Otherwise, $ f(A) \neq g(A) $ and we write the difference
\begin{align*}
s f(A) - s g(A)
& =  (\Delta f)(A) - (\Delta g)(A) + \sum_{\alpha \in \maxim(A)} \lambda_\alpha \left[\Big. f(A \setminus \alpha) - g(A \setminus \alpha) \right] \\
& \geq  (\Delta f)(A) - (\Delta g)(A)
\geq \varphi(A, f(A)) - \varphi(A, g(A))
> - s |f(A) - g(A)|.
\end{align*}
Rearranging, we obtain
\begin{align*}
0 & <  s (f(A) - g(A)) + s |f(A) - g(A)| \\
& =  s \left(\Big. \sgn(f(A) - g(A)) + 1 \right) |f(A) - g(A)|
\end{align*}
which implies $ \sgn(f(A) - g(A)) + 1 > 0 $ as $ f(A) \neq g(A) $.
Then, $ \sgn(f(A) - g(A)) > -1 $ and
therefore $ \sgn(f(A) - g(A)) = 1 $ and $ f(A) > g(A) $.
For the last statement, we may take $ \varphi(A,x) \equiv \varphi(A) = (\Delta g)(A) $ (or $ (\Delta f)(A) $),
and note for non-empty $ A $ that since $ x \neq y $,
\begin{equation*}
|\varphi(A, x)-\varphi(A, y)| \equiv 0 < |x - y| \sum_{\alpha\in \maxim(A)} \lambda_\alpha
\end{equation*}
is satisfied.  
\end{proof}

\begin{corollary} \label{coro:mgf-ineq}
Suppose $ h : \low{D} \to \mathbb{R} $ with $ D \in L(\Lambda) $ and $ |u| < \lambda_-(D) $.
If $ (\Delta h)(A) \geq u h(A) $ for non-empty $ A \subseteq D $
then $ h(A) \geq h(\emptyset) \expect{e^{u \tau_A}} $ for all $ A \subseteq D $.
Alternatively, if $ (\Delta h)(A) \leq u h(A) $ for non-empty $ A \subseteq D $
then $ h(A) \leq h(\emptyset) \expect{e^{u \tau_A}} $ for all $ A \subseteq D $.
\end{corollary}
\begin{proof}
For the first case, we apply Proposition \ref{prop:diff-ineq}
with $ f(A) = h(A) $, $ g(A) = h(\emptyset) \expect{e^{u \tau_A}} $,
and $ \varphi(A, x) = ux $.
When $ A = \emptyset $, we have $ f(\emptyset) = h(\emptyset) = g(\emptyset) $, and the claim holds.
Otherwise, $ A $ has at least one maximal element, and
\begin{equation*}
|\varphi(A, x) - \varphi(A, y)| = |u| \cdot |x - y|
< \lambda_-(D) |x - y|
\leq |x - y| \sum_{\alpha \in \maxim(A)} \lambda_\alpha.
\end{equation*}
From Lemma \ref{lem:diff-deriv-dual}, we have $ (\Delta g)(A) = u \expect{h(\emptyset) e^{u \tau_A}} = u g(A) = \varphi(A, g(A)) $.
From the hypothesis, we have $ (\Delta f)(A) = (\Delta h)(A) \geq u h(A) = u f(A) = \varphi(A, f(A)) $.
Hence, the conditions of Proposition \ref{prop:diff-ineq} have been fulfilled
and so $ h(A) = f(A) \geq g(A) = h(\emptyset) \expect{e^{u \tau_A}} $ for $ A \subseteq D $.

Similarly, for the second case, we apply Lemma \ref{lem:diff-deriv-dual}
with $ f(A) = h(\emptyset) \expect{e^{u \tau_A}} $, $ g(A) = h(A) $, and $ \varphi(A, x) = ux $.
Then, we get $ h(A) = g(A) \leq f(A) = h(\emptyset) \expect{e^{u \tau_A}} $ for $ A \subseteq D $.
\end{proof}

\section{Proofs}
\label{sect:proofs}
We now prove the results, mostly in succession, as stated in Section \ref{sect:results}.
We note the exponential bounds in Subsection \ref{subsect:mom-gen-bnds} are proven before those for the means,
as they are used in the arguments for the mean bounds in Subsection \ref{subsect:mean-bnd}.
 
\subsection{Proofs of Variance and Moment Bounds}
\label{subsect:proofs-mom-bnds}
We now supply the arguments for the bounds on
the variance $ \var{\tau_A} $ and moments given in Section \ref{sect:mom-bnds}.
These rely on the difference inequalities from Section \ref{sect:diff-oper}.
The variance bounds are given first as they are shorter.

\pagebreak[1]
\begin{proof}[Proof of Theorem \ref{thm:var-sublinear}]
By Lemma \ref{lem:diff-deriv-dual}, since $ t \mapsto t^2 $ is sub-exponential,
the function $ A \mapsto \expect{\tau_A^2} $ is well-defined on all of $ L(\Lambda) $ with $ D = \Lambda $.
When $A=\emptyset$, 
we know $ \var{\tau_\emptyset} = 0 = \expect{\tau_\emptyset} $, and the desired statement holds.

Otherwise, for non-empty $A$,
we apply $\Delta$ to $ \var{\tau_A} $.  Then, we
use Lemma \ref{lem:diff-deriv-dual} and Lemma \ref{lem:diff-leibniz}, noting $\Delta E[\tau_A]=1$, to obtain
\begin{align} \label{eq:diff-var}
\Delta \var{\tau_A} & =  \Delta \expect{\tau_A^2} - \Delta \left(\expect{\tau_A} \cdot \expect{\tau_A}\right) \\
& =  2 \expect{\tau_A} - \expect{\tau_A} - \expect{\tau_A} + [\expect{\tau_A}, \expect{\tau_A}]
= [\expect{\tau_A}, \expect{\tau_A}]. \nonumber
\end{align}
Now, since $ \expect{\tau_A} \geq \expect{\tau_B} $ whenever $ B \subseteq A $,
we have $ \expect{\tau_A - \tau_{A \setminus \alpha}} \geq 0 $.
Hence, the quadratic variation is bounded by
\begin{align*}
\lambda_-(A) [\expect{\tau_A}, \expect{\tau_A}]
& =  \lambda_- \sum_{\alpha \in \maxim(A)} \lambda_\alpha \expect{\tau_A - \tau_{A \setminus \alpha}}^2 \\
& \leq  \left(\sum_{\alpha \in \maxim(A)} \lambda_\alpha \expect{\tau_A - \tau_{A \setminus \alpha}} \right)^2
= \left(\Delta \expect{\tau_A}\right)^2 = 1 = \Delta \expect{\tau_A}
\end{align*}
implying $ \Delta \var{\tau_A} \leq \frac{1}{\lambda_-(A)} \Delta \expect{\tau_A} $.
So, by Proposition \ref{prop:diff-ineq}, $ \var{\tau_A} \leq \frac{1}{\lambda_-(A)} \expect{\tau_A} $.
\end{proof}

\pagebreak[1]
\begin{proof}[Proof of Proposition \ref{prop:var-mxml-bnd}]
When $A=\emptyset$, the claim holds trivially.  Otherwise, consider a non-empty set $A$.
For $ n \geq 1$, define $ V_n := \min\{\var{\tau_B} \,:\, B \in L(\Lambda), |B| = n\} \in \mathbb{R} $.
Also, define $ s_A := \sum_{\alpha \in \maxim(A)} \lambda_\alpha $.
From Theorem \ref{thm:var-sublinear} and Jensen's inequality, we have
\begin{align*}
\Delta \var{\tau_A} & =  [\expect{\tau_A}, \expect{\tau_A}]
= s_A \sum_{\alpha \in \maxim(A)} \frac{\lambda_\alpha}{s_A} \expect{\tau_A - \tau_{A \setminus \alpha}}^2 \\
& \geq  s_A \left(\sum_{\alpha \in \maxim(A)} \frac{\lambda_\alpha}{s_A} \expect{\tau_A - \tau_{A \setminus \alpha}} \right)^2
= s_A \left(\frac{\Delta \expect{\tau_A}}{s_A} \right)^2
= \frac{1}{s_A}.
\end{align*}
Now, for any $ n $, let $ A \in L(\Lambda) $ with $ |A| = n $.
From the definition of $ \Delta $ and noting that $ |A \setminus \alpha| = |A| - 1 = n - 1 $, we obtain
\begin{align*}
\var{\tau_A} & =  \frac{1}{s_A} \sum_{\alpha \in \maxim(A)} \lambda_\alpha \left(\big. \var{\tau_A} - \var{\tau_{A \setminus \alpha}} \right)
+ \frac{1}{s_A} \sum_{\alpha \in \maxim(A)} \lambda_\alpha \var{\tau_{A \setminus \alpha}} \\
& =  \frac{1}{s_A} \Delta \var{\tau_A}
+ \frac{1}{s_A} \sum_{\alpha \in \maxim(A)} \lambda_\alpha \var{\tau_{A \setminus \alpha}} \\
& \geq  \frac{1}{s_A^2} + \frac{1}{s_A} \sum_{\alpha \in \maxim(A)} \lambda_\alpha V_{|A| - 1}
\geq \frac{1}{f(|A|)^2} + V_{|A| - 1} = \frac{1}{f(n)^2} + V_{n-1}.
\end{align*}
This holds for all $ A $ with $ |A| = n $.
Here, $ V_n \geq \frac{1}{f(n)^2} + V_{n-1} $.
Thus, since $ f(k) $ increases as $ k $ increases, we have as desired,
\begin{equation*}
V_n \geq \sum_{k=1}^n \frac{1}{f(k)^2} \geq \int_1^{n+1} \frac{\dr x}{f(x)^2}.  \qedhere
\end{equation*}
\end{proof}

\pagebreak[1]
\begin{proof}[Proof of Proposition \ref{prop:high-mom-bnd}]
As $ t \mapsto t^n $ is sub-exponential, we may define $ q_n : L(\Lambda) \to \mathbb{R} $
by $ q_n(A) := \expect{\tau_A^n} - \expect{\tau_A}^n $ for every $ A \in L(\Lambda) $ (with $ D = \Lambda $).  When $A=\emptyset$, the claim holds trivially. 

Otherwise, for non-empty $ A $,
applying $ \Delta $ to $ \expect{\tau_A}^n $ yields
\begin{align*}
\Delta \expect{\tau_A}^n & =  \sum_{\alpha \in \maxim(A)} \lambda_\alpha
\left(\big. \expect{\tau_A}^n - \expect{\tau_{A \setminus \alpha}}^n \right) \\
& =  \sum_{\alpha \in \maxim(A)} \lambda_\alpha \left(\big. \expect{\tau_A} - \expect{\tau_{A \setminus \alpha}}\right)
\sum_{k=0}^{n-1} \expect{\tau_A}^k \expect{\tau_{A \setminus \alpha}}^{n - k - 1}.
\end{align*}
Then, using Lemma \ref{lem:diff-deriv-dual},
\begin{align*}
n \expect{\tau_A}^{n-1} \cdot 1 = n \expect{\tau_A}^{n-1} \cdot \Delta \expect{\tau_A}
& =  n \expect{\tau_A}^{n-1} \sum_{\alpha \in \maxim(A)} \lambda_\alpha
\left(\big. \expect{\tau_A} - \expect{\tau_{A \setminus \alpha}}\right) \\
& =  \sum_{\alpha \in \maxim(A)} \lambda_\alpha \left(\big. \expect{\tau_A} - \expect{\tau_{A \setminus \alpha}} \right)
\sum_{k=0}^{n-1} \expect{\tau_A}^{n-1}.
\end{align*}
Hence, the backward operator $\Delta$ applied to $ q_n $ gives
\begin{align} \label{eq:diff-q}
\Delta q_n & =  \Delta \expect{\tau_A^n} - \Delta \expect{\tau_A}^n \\
& =  n \left(\big. \expect{\tau_A^{n-1}} - \expect{\tau_A}^{n-1}\right)
+ n \expect{\tau_A}^{n-1} - \Delta \expect{\tau_A}^n \nonumber \\
& =  n q_{n-1}(A)
+ \sum_{\alpha \in \maxim(A)} \lambda_\alpha \left(\big. \expect{\tau_A} - \expect{\tau_{A \setminus \alpha}}\right) \Bigg[ \nonumber \\
&  \quad\quad\quad \sum_{k=0}^{n-1} \left(\Big. \expect{\tau_A}^{n-1} - \expect{\tau_A}^k \expect{\tau_{A \setminus \alpha}}^{n - k - 1}\right)\Bigg]. \nonumber
\end{align}
Now, we consider just the sum over $ k $,
\begin{align*}
\sum_{k=0}^{n-1} &  \left(\big. \expect{\tau_A}^{n-1} - \expect{\tau_A}^k \expect{\tau_{A \setminus \alpha}}^{n-k-1}\right) \\
& =  \sum_{k=0}^{n-2} \expect{\tau_A}^k \left(\big. \expect{\tau_A}^{n-k-1} - \expect{\tau_{A \setminus \alpha}}^{n-k-1}\right) \\
& =  \left(\big. \expect{\tau_A} - \expect{\tau_{A \setminus \alpha}}\right) \sum_{k=0}^{n-2} \expect{\tau_A}^k
\sum_{j=0}^{n-k-2} \expect{\tau_A}^j \expect{\tau_{A \setminus \alpha}}^{n - j - k - 2} \\
& =  \left(\big. \expect{\tau_A} - \expect{\tau_{A \setminus \alpha}}\right)
\sum_{s=0}^{n-2} \expect{\tau_A}^s \expect{\tau_{A \setminus \alpha}}^{n - s - 2} \sum_{j=0}^s 1 \\
& \leq  \left(\big. \expect{\tau_A} - \expect{\tau_{A \setminus \alpha}}\right)
\sum_{s=0}^{n-2} (s + 1) \expect{\tau_A}^{n - 2} \\
& =  \left(\big. \expect{\tau_A} - \expect{\tau_{A \setminus \alpha}}\right) \frac{n (n - 1)}{2} \expect{\tau_A}^{n - 2}
\end{align*}
where $ s = k + j $ and $ \expect{\tau_{A \setminus \alpha}} \leq \expect{\tau_A} $.
Substituting this back into eq. (\ref{eq:diff-q}) and recalling eq. (\ref{eq:diff-var}) gives
\begin{align} \label{eq:diff-q-bnd}
\Delta q_n & \leq  n q_{n-1}(A)
+ \sum_{\alpha \in \maxim(A)} \lambda_\alpha \left(\big. \expect{\tau_A} - \expect{\tau_{A \setminus \alpha}}\right)^2
\cdot \frac{n (n - 1)}{2} \expect{\tau_A}^{n - 2} \\
& =  n q_{n-1}(A) + \frac{n (n - 1)}{2} \expect{\tau_A}^{n - 2} \Delta \var{\tau_A}. \nonumber
\end{align}
Next, we notice that $ \var{\tau_{A\setminus \alpha}} \geq 0 $ and $ \expect{\tau_A}^{n-2} $ is non-decreasing in $A$.
Then, 
\begin{align*}
\left[\var{\tau_A}, \expect{\tau_A}^{n-2}\right]
& =  \sum_{\alpha \in \maxim(A)} \lambda_\alpha
\left(\var{\tau_A} - \var{\tau_{A \setminus \alpha}}\right)
\left(\expect{\tau_A}^{n - 2} - \expect{\tau_{A \setminus \alpha}}^{n - 2}\right) \\
& \leq  \var{\tau_A} \sum_{\alpha \in \maxim(A)} \lambda_\alpha
\left(\expect{\tau_A}^{n - 2} - \expect{\tau_{A \setminus \alpha}}^{n - 2}\right) \\
& =  \var{\tau_A} \Delta \expect{\tau_A}^{n-2}.
\end{align*}
Applying Lemma \ref{lem:diff-leibniz},
\begin{align*}
\Delta \left(\big. \var{\tau_A} \expect{\tau_A}^{n-2}\right)
& =  \expect{\tau_A}^{n-2} \Delta \var{\tau_A} + \var{\tau_A} \Delta \expect{\tau_A}^{n-2} \\
&  \quad\quad - \left[\var{\tau_A}, \expect{\tau_A}^{n-2}\right] \\
& \geq  \expect{\tau_A}^{n-2} \Delta \var{\tau_A} + \var{\tau_A} \Delta \expect{\tau_A}^{n-2} \\
&  \quad\quad - \var{\tau_A} \Delta \expect{\tau_A}^{n-2} \\
& =  \expect{\tau_A}^{n-2} \Delta \var{\tau_A}.
\end{align*}
So, eq. (\ref{eq:diff-q-bnd}) becomes
\begin{align*}
\Delta q_n & \leq  n q_{n-1}(A) + \frac{n (n - 1)}{2} \expect{\tau_A}^{n - 2} \Delta \var{\tau_A} \\
& \leq  n q_{n-1}(A) + \frac{n (n - 1)}{2} \Delta \left(\var{\tau_A} \expect{\tau_A}^{n - 2}\right).
\end{align*}

Notice that for $ n = 0 $ and $ n = 1 $, we have $ q_0(A) = q_1(A) = 0 $, and so the inequality eq. (\ref{eq:high-mom-bnd}) holds.
Next, we show by induction on $ n \geq 2 $ that $ q_n(A) \leq K \frac{n (n - 1)^2}{2} \expect{\tau_A^{p + n - 2}} $.
For $ n = 2 $, we have by assumption that
\begin{equation*}
q_2(A) = \expect{\tau_A^2} - \expect{\tau_A}^2 = \var{\tau_A}
\leq K \expect{\tau_A^p}
= K \frac{2 (2 - 1)^2}{2} \expect{\tau_A^{p + 2 - 2}}.
\end{equation*}
Then, when $ n \geq 3 $,
\begin{align*}
\Delta q_n(A) & -  \frac{n (n - 1)}{2} \Delta \left(\var{\tau_A} \expect{\tau_A}^{n-2}\right) \\
& \leq  n q_{n-1}(A) \\
& \leq  K n \frac{(n - 1) (n - 2)^2}{2} \expect{\tau_A^{p + n - 3}} \\
& \leq  K n \frac{(n - 1) (n - 2) (n - 2 + p)}{2} \expect{\tau_A^{p + n - 3}} \\
& =  K \frac{n (n - 1) (n - 2)}{2} \Delta \expect{\tau_A^{p + n - 2}},
\end{align*}
using Lemma \ref{lem:diff-deriv-dual}, while
\begin{equation*}
q_n(\emptyset) - \frac{n (n - 1)}{2} \var{\tau_\emptyset} \expect{\tau_\emptyset}^{n-2} = 0
\leq 0 = K \frac{n (n - 1) (n - 2)}{2} \expect{\tau_\emptyset^{p + n - 2}}.
\end{equation*}
Then, from Proposition \ref{prop:diff-ineq}, using $ \expect{\tau_A^p} \leq \expect{\tau_A}^p $ as $ 0 < p \leq 1 $
and $ \expect{\tau_A}^p \expect{\tau_A}^{n-2} \leq \expect{\tau_A^{p + n - 2}} $ as $ p + n - 2 \geq 1 + p $,
we get
\begin{align*}
q_n(A) & \leq  \frac{n (n - 1)}{2} \var{\tau_A} \expect{\tau_A}^{n - 2}
+ K \frac{n (n - 1) (n - 2)}{2} \expect{\tau_A^{p + n - 2}} \\
& \leq  K \frac{n (n - 1)}{2} \expect{\tau_A^{p + n - 2}} + K \frac{n (n - 1) (n - 2)}{2} \expect{\tau_A^{p + n - 2}} \\
& =  K \frac{n (n - 1)^2}{2} \expect{\tau_A^{p + n - 2}}
\end{align*}
completing the induction and proving the desired result.
\end{proof}

\pagebreak[1]
\begin{proof}[Proof of Corollary \ref{coro:central-moms-bnd}]
As in the proof of Proposition \ref{prop:high-mom-bnd}, consider
\begin{equation*}
q_n(A) := \expect{\tau_A^n} - \expect{\tau_A}^n.
\end{equation*}
We know $ q_n(A) \geq 0 $ and,
from Proposition \ref{prop:high-mom-bnd}, that $ q_n(A) \leq K \frac{n (n - 1)^2}{2} \expect{\tau_A^{p + n - 2}} $.
Then, expanding the central moment,
\begin{align*}
\left|\expect{(\tau_A - \expect{\tau_A})^n}\right|
& =  \left|\sum_{k=0}^n \binom{n}{k} (-1)^{n-k} \expect{\tau_A^k} \expect{\tau_A}^{n - k} \right| \\
& =  \Bigg|\sum_{k=0}^n \binom{n}{k} (-1)^{n-k} \expect{\tau_A}^k \expect{\tau_A}^{n - k} \\
&  \qquad + \sum_{k=0}^n \binom{n}{k} (-1)^{n-k} q_k(A) \expect{\tau_A}^{n-k}\Bigg| \\
& =  \left|\left(\expect{\tau_A} - \expect{\tau_A}\right)^n
+ \sum_{k=0}^n \binom{n}{k} (-1)^{n-k} q_k(A) \expect{\tau_A}^{n-k}\right| \\
& \leq  \sum_{k=0}^n \binom{n}{k} K \frac{k (k - 1)^2}{2} \expect{\tau_A^{p + k - 2}} \expect{\tau_A}^{n-k} \\
& =  \frac{K}{2} \sum_{k=2}^n n (n - 1) \binom{n - 2}{k - 2} (k - 1) \expect{\tau_A^{p + k - 2}} \expect{\tau_A}^{n-k} \\
& \leq  K\frac{n (n - 1)^2}{2} \expect{\tau_A^p \sum_{k=0}^{n-2} \binom{n - 2}{k} \tau_A^k \expect{\tau_A}^{n - 2 - k}} \\
& =  K \frac{n (n - 1)^2}{2} \expect{\tau_A^p \left(\tau_A + \expect{\tau_A}\right)^{n-2}}.
\end{align*}

To address the last claim, consider the bound $ |\expect{(\tau_A-\expect{\tau_A})^n}| \leq K \frac{n(n-1)^2}{2} \expect{\tau_A^p (\tau_A + \expect{\tau_A})^{n-2}} $.
Observe that $ \expect{\tau_A^p} \leq \expect{\tau_A}^p $ as $ 0 < p \leq 1 $, and
\begin{align*}
&\expect{\tau_A^p(\tau_A + \expect{\tau_A})^{n-2}} \leq 2^{n-3} \left(\expect{\tau_A^{p+n-2}} + \expect{\tau_A^p} \expect{\tau_A}^{n-2}\right) \nonumber \\
&\ \ \leq 2^{n-3}2^{p+n-3}\big(\expect{|\tau_A - \expect{\tau_A}|^{p+n-2}} +\expect{\tau_A}^{p+n-2}\big)
+  \expect{\tau_A}^{p+n-2}\big).
\end{align*}
Suppose $ n $ is even.
If $ \expect{(\tau_{A_j} - \expect{\tau_{A_j}})^n} $ diverges as $ j \to \infty $
then for every $ 0 < \epsilon < 1 $, for all sufficiently large $ j $,
\[
\expect{|\tau_{A_j} - \expect{\tau_{A_j}}|^{p+n-2}}
\leq \expect{(\tau_{A_j} - \expect{\tau_{A_j}})^n}^{\frac{p+n-2}{n}}
\leq \epsilon \expect{(\tau_{A_j} - \expect{\tau_{A_j}})^n}.
\]
Hence, by choosing $ \epsilon 2^{p+2n-6} K\frac{n(n-1)^2}{2} \leq 1/2 $ and rearranging,
there exists $ C_n > 0 $ such that for all sufficiently large $ j $, $ \left| \expect{(\tau_A - \expect{\tau_A})^n} \right| \leq C_n \mu_j^{p+n-2} $
and so the claim holds.
If $ |\expect{(\tau_{A_j} - \expect{\tau_{A_j}})^n}| $ does not diverge, the claim trivially holds.

For odd $ n $, by Schwarz inequality,
\begin{align*}
\left | \expect{(\tau_{A_j} - \expect{\tau_{A_j}})^n}\right|
&= \left|\expect{(\tau_{A_j} - \expect{\tau_{A_j}})^{\frac{n-1}{2}} \cdot (\tau_{A_j} - \expect{\tau_{A_j}})^{\frac{n+1}{2}}}\right|\\
&\leq \sqrt{\expect{(\tau_A - \expect{\tau_A})^{n+1}} \expect{(\tau_A - \expect{\tau_A})^{n-1}}}.
\end{align*}
Since $ n + 1 $ and $ n - 1 $ are even, we can bound this
by $ \sqrt{C_{n+1} C_{n-1} \mu_j^{p+n+1-2} \mu_j^{p+n-1+2}} = O(\mu_j^{p+n-2}) $.
\end{proof}

\subsection{Proofs of Moment Generating Function Bounds}
\label{subsect:proofs-mgf-bnds}
Recall the definition of the `greater path function' $ \GrtrPath f: L(\Lambda)\to \mathbb{R} $.
We also define the \textbf{backward difference operator} $ \delta f : \mathbb{N}_0 \to \mathbb{R} $ of $ f $ as
\begin{equation*}
(\delta f)(n) := f(n) - f(n - 1) \; \text{for } n \geq 1
\quad \text{and} \quad (\delta f)(0) = f(0).
\end{equation*}
Also, we define the function $ \epsilon : \mathbb{N}_0 \to \mathbb{R} $
by $ \epsilon(0) = 1 $ and $ \epsilon(n) = 0 $ for $ n > 0 $.
Note that $ \delta 1 = \epsilon $ (where $ 1 $ is the constant function).

\begin{lemma} \label{lem:max-path-decomp}
For non-empty $ A \in L(\Lambda) $, we have $ \Pi_m(A) = \bigcup_{\alpha \in \maxim(A)} (\Pi(A) \setminus \Pi(A \setminus \alpha)) $.
\end{lemma}
\begin{proof}
Let $ \pi \in \Pi_m(A) $. Since $ \pi $ is maximal in $ A $
we have $ \beta := \pi_{\ell(\pi)} \in \maxim(A) $.
Since $ \beta \notin A \setminus \beta $,  we have $ \pi \notin \Pi(A \setminus \beta) $.
Hence, $ \pi \in \Pi(A) \setminus \Pi(A \setminus \beta) $.

Conversely, let $ \pi \in \bigcup_{\alpha \in \maxim(A)} (\Pi(A) \setminus \Pi(A \setminus \alpha)) $
then $ \pi \in \Pi(A) \setminus \Pi(A \setminus \beta) $ for some $ \beta \in \maxim(A) $.
Because $ \pi \subseteq \Pi(A) $ and $ \pi \nsubseteq \Pi(A \setminus \beta) $, we must have $ \beta \in \pi $.
Since $ \beta $ is maximal in $ A $, we know $ \pi_{\ell(\pi)} = \beta $ and may conclude $ \pi $ is maximal in $ A $.
\end{proof}

\begin{proposition} \label{prop:grtr-path-diff-ineq}
For any non-decreasing $ f : \mathbb{N}_0 \to \mathbb{R} $ with $ f(0) \geq 0 $
and any $ A \in L(\Lambda) $, we have
\begin{equation*}
\Delta \left(\GrtrPath f\right)(A) \geq \lambda_-(A) \GrtrPath (\delta f)(A).
\end{equation*}
\end{proposition}
\begin{proof}
The display holds for $A=\emptyset$ by our conventions.  Otherwise, for non-empty $A$,
applying $\Delta$ to $ \GrtrPath f $
and recalling Lemma \ref{lem:max-path-decomp}, we have
\begin{align} \label{eq:diff-grtr-path-maximals}
\left(\Delta \GrtrPath f\right)(A) & =  \sum_{\alpha \in \maxim(A)} \lambda_\alpha \left[
\sum_{\pi \in \Pi(A)} f(\ell(\pi)) - \sum_{\pi \in \Pi(A \setminus \alpha)} f(\ell(\pi))
\right] \\
& =  \sum_{\alpha \in \maxim(A)} \lambda_\alpha \sum_{\pi \in \Pi(A) \setminus \Pi(A \setminus \alpha)} f(\ell(\pi))
= \sum_{\pi \in \Pi_m(A)} \lambda_{\pi_{\ell(\pi)}} f(\ell(\pi)). \nonumber
\end{align}
Note that $ \delta f \geq 0 $ since $ f $ is non-decreasing and $ f(0) \geq 0 $.
Because every path in $ A $ is contained in at least one maximal path
and every subpath of a maximal path may be labelled by its length $ k = \ell(\xi) $ for $ \xi \subseteq \pi $,
we may write
\begin{align*}
\lambda_-(A) \left(\GrtrPath \delta f\right)(A) & =  \lambda_- \sum_{\xi \in \Pi(A)} \delta f(\ell(\xi)) \nonumber \\
& \leq  \lambda_- \sum_{\pi \in \Pi_m(A)} \sum_{\substack{\xi \in \Pi(A) \\ \xi \subseteq \pi}} \delta f(\ell(\xi))
= \lambda_- \sum_{\pi \in \Pi_m(A)} \sum_{k=0}^{\ell(\pi)} \delta f(k) \nonumber \\
& =  \lambda_- \sum_{\pi \in \Pi_m(A)} \left(f(0) + \sum_{k=1}^{\ell(\pi)} f(k) - f(k - 1)\right) \nonumber \\
& =  \sum_{\pi \in \Pi_m(A)} \lambda_-(A) f(\ell(\pi))
\leq \left(\Delta \GrtrPath f\right)(A). \qedhere
\end{align*}
\end{proof}

\pagebreak[1]
\begin{proof}[Proof of Proposition \ref{prop:grtr-path-mgf}]
One may verify that the claim holds when $ B = \emptyset $.
Otherwise, we fix a non-empty $ B \in L(\Lambda) $ and will consider functions defined on $ \low{B} $.
All applications of the operator $ \Delta $
will be with respect to $ A \in \low{B} $.
Note that
\begin{equation*}
\GrtrPath (\epsilon)(A) = \sum_{\pi \in \Pi(A)} \epsilon(\ell(\pi)) = \epsilon(\ell(\emptyset)) = \epsilon(0) = 1.
\end{equation*}
By Proposition \ref{prop:grtr-path-diff-ineq}, we have $ (\Delta \GrtrPath 1)(A) \geq \lambda_-(A) (\GrtrPath \epsilon)(A) = \lambda_-(A) $.

For any $ r > 1 $, consider $ f : \mathbb{N}_0 \to \mathbb{R} $
given by $ f(n) = \frac{r^{n+1} - 1}{r - 1} $.
Then, $ \delta f(0) = \frac{r^{0 + 1} - 1}{r - 1} = 1 = r^0 $ and, for $ n \geq 1 $, we have
\begin{equation*}
\delta f(n) = \frac{r^{n + 1} - 1}{r - 1} - \frac{r^n - 1}{r - 1}
= \frac{r^n (r - 1)}{r - 1} = r^n.
\end{equation*}
Denote by $ r^n $ the function $ n \mapsto r^n $.
Using the linearity of $ \Delta $ and $ \GrtrPath $
as well as Proposition \ref{prop:grtr-path-diff-ineq}, for any non-empty $ A \subseteq B $, we have
\begin{align*}
\lambda_-(B) \GrtrPath \{r^n\}(A) & \leq  \lambda_-(A) \GrtrPath \{r^n\}(A)
= \lambda_-(A) \GrtrPath \{\delta f\}(A) \\
& \leq  \Delta \GrtrPath \{f\}(A)
= \frac{r}{r - 1} \Delta \GrtrPath \{r^n\}(A) - \frac{1}{r - 1} \Delta \GrtrPath\{1\}(A) \\
& \leq  \frac{r}{r - 1} \Delta \GrtrPath \{r^n\}(A) - \lambda_-(A) \frac{1}{r - 1} \\
& \leq  \frac{r}{r - 1} \Delta \GrtrPath \{r^n\}(A) - \lambda_-(B) \frac{1}{r - 1}.
\end{align*}
Hence, $ \Delta \GrtrPath \{r^n\}(A) \geq \lambda_-(B) \left(1 - \frac{1}{r}\right) \GrtrPath \{r^n\}(A) + \lambda_-(B) \frac{1}{r} $.

Consider the function $ A \mapsto \frac{1}{r - 1} \expect{r e^{\lambda_-(B) (1 - 1/r) \tau_A} - 1} $
which is well-defined for $ A \in \low{B} $ (see Subsection \ref{subsect:well-defined}).
Let $ u := 1 - 1/r < 1 $. Then, by Lemma \ref{lem:diff-deriv-dual} for non-empty $ A $, we have
\begin{align} \label{eq:path-func-mgf-bdd}
\Delta \expect{\frac{r e^{\lambda_-(B) (1 - 1/r) \tau_A} - 1}{r - 1}}
& =  \lambda_-(B) \left(1 - \frac{1}{r}\right) \expect{\frac{r e^{\lambda_-(B) (1 - 1/r) \tau_A}}{r - 1}} \nonumber \\
& =  \lambda_-(B) \left(1 - \frac{1}{r}\right) \expect{\frac{r e^{\lambda_-(B) u \tau_A} - 1}{r - 1}} + \frac{\lambda_-(B)}{r}
\end{align}
with $ \frac{1}{r - 1} \expect{r e^{\lambda_-(B) \, (1 - 1/r) \tau_\emptyset} - 1} = 1 $
and $ \GrtrPath \{r^n\}(\emptyset) = \sum_{\pi \in \Pi(\emptyset)} r^{\ell(\pi)} = r^{\ell(\emptyset)} = 1 $.

Hence, by Proposition \ref{prop:diff-ineq}, with $ \varphi(A, x)= \lambda_-(B)(1-\frac{1}{r})x + \lambda_-(B)\frac{1}{r} $,
we may conclude that $ \GrtrPath \{r^n\}(A) \geq \frac{1}{r - 1} \expect{r e^{\lambda_-(B) (1 - 1/r) \tau_A} - 1} $.
Taking $ A = B $, we can write
\begin{align*}
u \GrtrPath \left\{\frac{1}{(1 - u)^n}\right\}(B) - u
& =  \frac{r - 1}{r} \GrtrPath \{r^n\} (B) - 1 + \frac{1}{r} \\
& \geq  \expect{e^{\lambda_-(B) \, (1 - 1/r) \tau_B}} - 1
= \expect{e^{\lambda_-(B) \, u \tau_B}} - 1.  \qedhere
\end{align*}
\end{proof}

We will use a similar approach to construct a lower bound for the moment generating functions.
Recall the definitions of a `branching allocation' $ \psi $,
its `weight' $ \omega_\psi $ along a path $ \pi \in \Pi(\Lambda) $,
and `lesser path functions' $ \LssrPath^\psi f : L(\Lambda)\to \mathbb{R} $ in Definition \ref{defn:branch-alloc}.
As with the greater path function, there is another difference inequality
with respect to the lesser path function.
To derive it, we first consider how a branching allocation distributes weight among paths.

\begin{lemma} \label{lem:branch-alloc-weights}
For any branching allocation $ \psi $, all $ \xi \in \Pi(A) $, and all $ A \in L(\Lambda) $, we have
\begin{equation*}
\sum_{\substack{\pi \in \Pi_m(A) \\ \xi \subseteq \pi}} \omega_\psi(\pi) \leq \omega_\psi(\xi).
\end{equation*}
\end{lemma}
\begin{proof}
We will induct on $ \max_\pi(\ell(\pi) - \ell(\xi)) $.
When $ \max_\pi(\ell(\pi) - \ell(\xi)) = 0 $, since $ \xi = \pi $,
we have $ \xi $ is maximal in $ A $ and
\begin{equation*}
\sum_{\substack{\pi \in \Pi_m(A) \\ \xi \subseteq \pi}} \omega_\psi(\pi) = \sum_{\pi = \xi} \omega_\psi(\pi) = \omega_\psi(\xi).
\end{equation*}

Otherwise, when $ \max_\pi (\ell(\pi) - \ell(\xi)) > 0 $, we have that $ \xi $ is not maximal in $ A $.
When $ \xi $ is empty,
we have $ \max_\pi(\ell(\pi) - \ell(\xi)) = \max_\pi \ell(\pi) > \max_\pi (\ell(\pi) - 1) = \max_\pi(\ell(\pi) - \ell((\mu))) $,
where $ (\mu) $ is a singleton path contained in $ \pi $.
Then, using the induction hypothesis,
\begin{align*}
\sum_{\substack{\pi \in \Pi_m(A) \\ \xi \subseteq \pi}} \omega_\psi(\pi) & =  \sum_{\pi \in \Pi_m(A)} \omega_\psi(\pi)
= \sum_{\substack{\mu \in A \\ \mu \text{ is minimal}}} \sum_{\substack{\pi \in \Pi_m(A) \\ (\mu) \subseteq \pi}} \omega_\psi(\pi) \\
& \leq  \sum_{\substack{\mu \in A \\ \mu \text{ is minimal}}} \omega_\psi((\mu))
= \sum_{\substack{\mu \in A \\ \mu \text{ is minimal}}} \psi_\mu \leq 1.
\end{align*}

When $ \xi $ is non-empty and $ \alpha = \xi_{\ell(\xi)} $.
Let $ \{\xi^{(\beta)}\}_{\alpha \to \beta} $ be the paths in $ A $ obtained by extending $ \xi $
by an upper neighbor $ \beta $ of $ \alpha $.
By definition of $ \omega_\psi $, we have $ \omega_\psi(\xi^{(\beta)}) = \psi_{\alpha \to \beta} \cdot \omega_\psi(\xi) $.
Then, because every maximal path contains one of $ \{\xi^{(\beta)}\} $ for $ \alpha \to \beta $,
using the induction hypothesis $ \max_\pi(\ell(\pi) - \ell(\xi)) > \max_\pi(\ell(\pi) - \ell(\xi^{(\beta)})) $
where $ \xi^{(\beta)} \subseteq \pi $, we have
\begin{align*}
\sum_{\substack{\pi \in \Pi_m(A) \\ \xi \subseteq \pi}} \omega_\psi(\pi)
& \leq  \sum_{\alpha \to \beta} \sum_{\substack{\pi \in \Pi_m(A) \\ \xi^{(\beta)} \subseteq \pi}} \omega_\psi(\pi) \\
& \leq  \sum_{\alpha \to \beta} \omega(\xi^{(\beta)})
= \omega_\psi(\xi) \sum_{\alpha \to \beta} \psi_{\alpha \to \beta} \leq \omega_\psi(\xi).
 \qedhere
\end{align*}
\end{proof}

\begin{proposition} \label{prop:lssr-path-diff-ineq}
Let $ \psi $ be a branching allocation.
For any non-decreasing $ f : \mathbb{N}_0 \to \mathbb{R} $ with $ f(0) \geq 0 $ and $ A \in L(\Lambda) $, we have
\begin{equation*}
\Delta \left(\LssrPath^\psi f\right)(A) \leq \lambda_+(A) \LssrPath^\psi (\delta f)(A).
\end{equation*}
\end{proposition}
\begin{proof}
The display holds when $A=\emptyset$ by our conventions.  Otherwise, for non-empty $A\in L(\Lambda)$,
applying the operator $\Delta$ to $ \LssrPath^\psi f $ and using Lemma \ref{lem:max-path-decomp},
we get
\begin{align*}
\left(\Delta \LssrPath^\psi f\right)(A) & =  \sum_{\alpha \in \maxim(A)} \lambda_\alpha \left[
\sum_{\pi \in \Pi(A)} f(\ell(\pi)) \omega_\psi(\pi) - \sum_{\pi \in \Pi(A \setminus \alpha)} f(\ell(\pi)) \omega_\psi(\pi)
\right] \\
& =  \sum_{\pi \in \Pi_m(A)} \lambda_{\pi_{\ell(\pi)}} f(\ell(\pi)) \omega_\psi(\pi).
\end{align*}
Recalling the scheme of proof of Proposition \ref{prop:grtr-path-diff-ineq},
using Lemma \ref{lem:branch-alloc-weights}, we can write
\begin{align*}
\lambda_+(A) \left(\LssrPath^\psi (\delta f) \right)(A) & =  \lambda_+(A) \sum_{\xi \in \Pi(A)} \delta f(\ell(\xi)) \omega_\psi(\xi) \\
& \geq  \lambda_+(A) \sum_{\xi \in \Pi(A)} \sum_{\substack{\pi \in \Pi_m(A) \\ \xi \subseteq \pi}} \delta f(\ell(\xi)) \omega_\psi(\pi) \\
& =  \lambda_+(A) \sum_{\pi \in \Pi_m(A)} \omega_\psi(\pi) \sum_{k=0}^{\ell(\pi)} \delta f(k) \\
& \geq  \sum_{\pi \in \Pi_m(A)} \lambda_{\pi_{\ell(\pi)}} f(\ell(\pi)) \omega_\psi(\pi) = \left(\Delta \LssrPath^\psi f\right)(A).
 \qedhere
\end{align*}
\end{proof}

\pagebreak[1]
\begin{proof}[Proof of Proposition \ref{prop:lssr-path-mgf}]
Consider non-empty $ B \in L(\Lambda) $ and $ A \in \low{B} $.
Note
\[ (\LssrPath^\psi \epsilon)(A) = \sum_{\pi \in \Pi(A)} \epsilon(\ell(\pi)) \omega_\psi(\pi) = \omega_\psi(\emptyset) = 1. \]
Then, by Proposition \ref{prop:lssr-path-diff-ineq}, $ \Delta (\LssrPath^\psi 1)(A) \leq \lambda_+(A) (\LssrPath^\psi \epsilon)(A) = \lambda_+(A) $.
As in the proof of Proposition \ref{prop:grtr-path-mgf},
let $ r > 1 $ and $ f : \mathbb{N}_0 \to \mathbb{R} $ with $ f(n) := \frac{r^{n+1} - 1}{r - 1} $
so that $ \delta f(n) = r^n $.
Then, using the linearity of $ \Delta $ and $ \LssrPath $
as well as Proposition \ref{prop:lssr-path-diff-ineq}, we have
\begin{align*}
\lambda_+(B) \LssrPath^\psi \{r^n\}(A)
& \geq  \lambda_+(A) \LssrPath^\psi \{r^n\}(A) = \lambda_+(A) \LssrPath^\psi (\delta f) \\
& \geq  \Delta \LssrPath^\psi \{f\} (A)
= \frac{r}{r - 1} \Delta \LssrPath^\psi \{r^n\}(A) - \frac{1}{r - 1} \Delta \LssrPath^\psi \{1\}(A) \\
& \geq  \frac{r}{r - 1} \Delta \LssrPath^\psi \{r^n\}(A) - \lambda_+(A) \frac{1}{r - 1} \\
& \geq  \frac{r}{r - 1} \Delta \LssrPath^\psi \{r^n\}(A) - \lambda_+(B) \frac{1}{r - 1}.
\end{align*}
Hence, $ \Delta \LssrPath^\psi \{r^n\}(A) \leq \lambda_+(B) \left(1 - \frac{1}{r}\right) \LssrPath^\psi \{r^n\}(A) + \lambda_+(B) \frac{1}{r} $.

Recall eq. (\ref{eq:path-func-mgf-bdd}) (using $ \lambda_+(B) $ instead of $ \lambda_-(B) $).
Then,
\begin{equation*}
\LssrPath^\psi \{r^n\}(A) \leq \frac{1}{r - 1} \expect{r e^{\lambda_+(B) (1 - 1/r) \tau_A} - 1},
\end{equation*}
the right-hand side function being well-defined for $ u = 1 - 1/r < \lambda_-(B)/\lambda_+(B) $.
Taking now $ A = B $ and rewriting, yields the desired result.
\end{proof}

\subsection{Proofs of Mean Bounds}
\label{subsect:proofs-mean}
We now turn to the estimation of the mean $ \expect{\tau_A} $
based on the bounds in Subsection \ref{subsect:proofs-mgf-bnds}.

\pagebreak[1]
\begin{proof}[Proof of Proposition \ref{prop:mean-lower}]
The desired statement holds for $B=\emptyset$.  Otherwise, consider non-empty
$ B \in L(\Lambda) $ and a non-empty $ A \in \low{B} $.
Either all maximum length paths in $ A $ end at a single element $ \tilde{\alpha} \in \maxim(A) $ or not.
Suppose such a $ \tilde{\alpha} $ exists.
Then, for all $ \alpha \in \maxim(A) \setminus \tilde{\alpha} $,
we have $ \tilde{\alpha} \in A \setminus \alpha $ so $ \ell(A \setminus \alpha) = \ell(A) $.
On the other hand, $ \ell(A \setminus \tilde{\alpha}) = \ell(A) - 1 $
since all maximum paths in $ A \setminus \tilde{\alpha} $ have length at most $ \ell(A) - 1 $
and there exists a maximal path in $ A $ which, after truncating $ \tilde{\alpha} $,
is a maximal path in $ A \setminus \tilde{\alpha} $ with length $ \ell(A) - 1 $.
Applying $ \Delta $ to $ \ell(A) $, we have
\begin{equation*}
\Delta \ell(A) = \sum_{\alpha \in \maxim(A)} \lambda_\alpha [\ell(A) - \ell(A \setminus \alpha)]
= \lambda_{\tilde{\alpha}} [\ell(A) - \ell(A \setminus \tilde{\alpha})]
= \lambda_{\tilde{\alpha}} \leq \lambda_+(B).
\end{equation*}
Alternatively, suppose no such $ \tilde{\alpha} $ exists.
Then, there exist at least two elements in $ \maxim(A) $
that are the terminuses of maximum length paths.
Thus, for any $ \alpha \in \maxim(A) $, the set $ A \setminus \alpha $ will still contain
a path of length $ \ell(A) $ implying that $ \ell(A \setminus \alpha) = \ell(A) $.
Hence, $ \Delta \ell(A) = 0 $.

Therefore, in all cases when $A$ is non-empty, noting $ \Delta \expect{\tau_A} = 1 $ (see Lemma \ref{lem:diff-deriv-dual}),
we have $ \Delta \ell(A) \leq \lambda_+(B) \leq \Delta \left(\lambda_+(B) \expect{\tau_A}\right) $.
Since $ \ell(\emptyset) = 0 = \lambda_+(B) \expect{\tau_\emptyset} $,
we conclude $ \ell(A) \leq \lambda_+(B) \expect{\tau_A} $ by Proposition \ref{prop:diff-ineq}.
The result follows by taking $ A = B $.
\end{proof}

\begin{remark} 
\label{rem:lssr-path-mean-lower}
We give an alternate proof of Proposition \ref{prop:mean-lower} for the reader's interest.
Consider the inequality in Proposition \ref{prop:lssr-path-mgf}.
Dividing by $ u $, the right-hand side converges to $\lambda_+(B) \expect{\tau_B} $ as $ u \to 0^+ $.
For the left-hand side, consider a branching allocation $ \psi $
such that $ \sum_\beta \psi_{\alpha \to \beta} = 1 $
and $ \sum_{\mu \text{ is minimal}} \psi_\mu = 1 $.
Indeed, one may take $ \psi_{\alpha \to \beta} = 1 / d_\alpha $
where $ d_\alpha $ is the number of upper neighbors of $ \alpha $
and $ \psi_\mu = 1 / d_{\mathrm{min}} $ where $ d_{\mathrm{min}} $ is the number of minimal elements.
Then, we may write
\begin{align} \label{eq:lssr-path-mean-lower}
\lim_{u \to 0^+} \LssrPath^\psi \left\{\frac{1}{(1-u)^n}\right\} - 1
& =  \lim_{u \to 0^+} \left(\sum_{\pi \in \Pi(B)} \frac{1}{(1-u)^{\ell(\pi)}} \omega_\psi(\pi) \right) - 1 \\
& =  -1 + \sum_{\pi \in \Pi(B)} \omega_\psi(\pi)
= -1 + \sum_{k=0}^{\ell(B)} \sum_{\substack{\pi\in \Pi(B) \\ \ell(\pi) = k}} \omega_\psi(\pi). \nonumber
\end{align}
Now, as $ \omega_\psi(\pi) = \psi_{\pi_1} \prod_{i=2}^{\ell(\pi)} \psi_{\pi_{i-1} \to \pi_i} $,
summing successively over the possible points $ \pi_i $,
we have
\[ \sum_{\substack{\pi \in \Pi(B) \\ \ell(\pi) = k}} \omega_\psi(\pi) = 1 \]
and so eq. (\ref{eq:lssr-path-mean-lower}) equals $ -1 + \ell(B) + 1 = \ell(B) $ as desired.
\end{remark}

\pagebreak[1]
\begin{proof}[Proof of Theorem \ref{thm:mean-upper}]
The claim follows for $A=\emptyset$ by our conventions.  Otherwise, consider a non-empty set $A$.  For any $ 0 < u < 1 $, as the mapping $ t \mapsto e^{\lambda_-(A) \, ut} $ is convex,
by Jensen's inequality, $ \expect{e^{\lambda_-(A) \, u \tau_A}} \geq e^{u \lambda_-(A) \expect{\tau_A}} $.
Consider $ r := \frac{1}{1 - u} $.
Then, using Proposition \ref{prop:grtr-path-mgf} and Proposition \ref{prop:grtr-path-diff-ineq} and recalling eq. (\ref{eq:diff-grtr-path-maximals}),
we have
\begin{align*}
e^{u \lambda_-(A) \expect{\tau_A}} - 1 & \leq  u \left(\GrtrPath\left\{\frac{1}{(1 - u)^n}\right\}(A) - 1\right)
= \left(1 - \frac{1}{r}\right) \left(\GrtrPath\{r^n\}(A) - 1\right) \\
& \leq  \frac{1}{\lambda_-(A)} \left(1 - \frac{1}{r}\right) \Delta \GrtrPath\left\{\frac{r^{n + 1} - 1}{r - 1}\right\}(A) - 1 + \frac{1}{r} \\
& =  -1 + \frac{1}{r} + \frac{r - 1}{r} \frac{1}{\lambda_-(A)} \sum_{\pi \in \Pi_m(A)} \lambda_{\pi_{\ell(\pi)}} \frac{r^{\ell(\pi) + 1} - 1}{r - 1} \\
& \leq  -1 + \frac{1}{r} + \frac{1}{\lambda_-(A)} \sum_{\pi \in \Pi_m(A)} \lambda_+(A) \left[ r^{\ell(A)} - \frac{1}{r} \right] \\
& =  -1 + \frac{1}{r} \left(1 - \frac{\lambda_+(A)}{\lambda_-(A)} |\Pi_m(A)|\right) + \frac{\lambda_+(A)}{\lambda_-(A)} |\Pi_m(A)| r^{\ell(A)} \\
& \leq  -1 + \frac{\lambda_+(A)}{\lambda_-(A)} |\Pi_m(A)| r^{\ell(A)}.
\end{align*}
Hence, as $ r > 1 $, we have $ e^{u \lambda_-(A) \expect{\tau_A}} \leq \frac{\lambda_+(A)}{\lambda_-(A)} |\Pi_m(A)| r^{\ell(A)} = e^{\eta(A) + \kappa(A)} r^{\ell(A)} $.
After taking the logarithm, we obtain
\begin{align*}
\lambda_-(A) \expect{\tau_A} & \leq  \frac{\kappa(A) + \eta(A) + \ell(A) \log(r)}{u} \\
& =  \frac{1}{u} \left( \kappa(A) + \eta(A) + \ell(A) \log\left(1 + \frac{u}{1 - u}\right)\right) \\
& \leq  \frac{1}{u} \left(\kappa(A) + \eta(A) + \ell(A) \frac{u}{1 - u}\right).
\end{align*}
We may optimize our choice of $ u $ to minimize this value.
Differentiate $ \frac{\kappa(A) + \eta(A)}{u} + \frac{\ell(A)}{1 - u} $ and consider
\begin{equation*}
0 = - \frac{\kappa(A) + \eta(A)}{u_0^2} + \frac{\ell(A)}{(1 - u_0)^2}
\quad \text{and} \quad
\sqrt{\frac{\kappa(A) + \eta(A)}{\ell(A)}} = \frac{u_0}{1 - u_0} = \frac{1}{1 - u_0} - 1
\end{equation*}
implying $ u_0 = \frac{\sqrt{\kappa(A) + \eta(A)}}{\sqrt{\kappa(A) + \eta(A)} + \sqrt{\ell(A)}} $.
Hence, the desired result follows from
\begin{align*}
\lambda_-(A) \expect{\tau_A}
& \leq  (\kappa(A) + \eta(A)) \frac{\sqrt{\kappa(A) + \eta(A)} + \sqrt{\ell(A)}}{\sqrt{\kappa(A) + \eta(A)}} \\
&  \qquad + \ell(A) \frac{\sqrt{\kappa(A) + \eta(A)} + \sqrt(\ell(A))}{\sqrt{\ell(A)}} \\
& \leq  \sqrt{\kappa(A) + \eta(A)} \left( \sqrt{\kappa(A) + \eta(A)} + \sqrt{\ell(A)} \right) \\
&  \qquad + \sqrt{\ell(A)} \left( \sqrt{\kappa(A) + \eta(A)} + \sqrt{\ell(A)} \right) \\
& =  \left(\sqrt{\kappa(A) + \eta(A)} + \sqrt{\ell(A)}\right)^2.  \qedhere
\end{align*}
\end{proof}

\pagebreak[1]
\begin{proof}[Proof of Lemma \ref{lem:kappa-bnd}]
We consider a branching allocation $ \psi $ specified by
\begin{align*}
\psi_\mu & = 
 \begin{cases} \frac{1}{d} & \text{when } \mu = \mu_i \\ 0 & \text{otherwise} \end{cases}
\quad \text{ for minimal elements } \mu \in \Lambda
\quad \text{and} \\
\psi_{\alpha \to \beta} & =  \begin{cases} \frac{1}{d} & \text{when } \alpha, \beta \in A \\ 0 & \text{otherwise.} \end{cases}
\quad \text{ for any } \alpha, \beta \in \Lambda
\text{ where $ \beta $ is an upper neighbor of $ \alpha $}
\end{align*}
Then, we have $ \sum_\mu \psi_\mu \leq 1 $ since $ A $ has at most $ d $ minimal elements.
Also, for each $ \alpha \in \Lambda $, we have $ \sum_\beta \psi_{\alpha \to \beta} \leq 1 $
because $ \alpha \in A $ has at most $ d $ upper neighbors in $ A $.
Now, by Lemma \ref{lem:branch-alloc-weights}, we write
\begin{align*}
|\Pi_m(A)| \cdot d^{-\ell(A)} & =  \sum_{\pi \in \Pi_m(A)} \left(\frac{1}{d}\right)^{\ell(A)} \\
& \leq  \sum_{\pi \in \Pi_m(A)} \left(\frac{1}{d}\right)^{\ell(\pi)}
= \sum_{\pi \in \Pi_m(A)} \omega_\psi(\pi)
\leq \omega_\psi(\emptyset) = 1.
\end{align*}
Hence, $ |\Pi_m(A)| \leq d^{\ell(A)} $, giving $ \kappa(A) = \log |\Pi_m(A)| \leq \log(d) \, \ell(A) $.
\end{proof}

\pagebreak[1]
\begin{proof}[Proof of Theorem \ref{thm:lattice-kappa-limit}]
First, as every maximal path in $ A = \langle n \alpha \rangle $
must contain $ n \alpha_i $ steps in the $ i $-th dimension, we have
 $ \ell(\langle n \alpha \rangle) = n \ell(\alpha) $.
Thus, every maximal path may be enumerated by counting
the labelings of the $ n \ell(\alpha) $ steps across the different dimensions. Hence,
there are $ |\Pi_m(\langle n \alpha \rangle)| = \binom{n \ell(\alpha)}{n \alpha_1, \ldots, n \alpha_d} $ different maximal paths.

Using Stirling's approximation, as $ n \to \infty $,
\begin{align*}
\binom{n \ell{\alpha}}{n \alpha_1, \ldots, n \alpha_d} = \frac{(n \ell(\alpha))!}{(n \alpha_1)! \cdots (n \alpha_d)!}
& \sim  \frac{\sqrt{2\pi n \ell(\alpha)} \left(\frac{n \ell(\alpha)}{e}\right)^{n \ell(\alpha)}}
{\prod_{i=1}^d \sqrt{2\pi n \alpha_i} \left(\frac{n \alpha_i}{e}\right)^{n \alpha_i}} \\
& =  \sqrt{2\pi n \ell(\alpha)}^{1 - d} \left(\prod_{i=1}^d p_i^{-1/2} \right)
\left(\prod_{i=1}^d p_i^{p_i} \right)^{-n \ell(\alpha)}
\end{align*}
where $ \sim $ means that the ratio of the two quantities approaches one.
Taking the logarithm, we obtain
\begin{align*}
\log(|\Pi_m(\langle n \alpha \rangle)|) & -  \frac{1 - d}{2} \log(2 \pi n \ell(\alpha))
+ \sum_{i=1}^d \left[ \frac{1}{2} \log(p_i) + n \ell(\alpha) p_i \log(p_i) \right] \to 0  \quad \text{and} \\
\frac{\kappa(\langle n \alpha \rangle)}{\ell(\langle n \alpha \rangle)} & =  \frac{\log(|\Pi_m(\langle n \alpha \rangle)|)}{n \ell(\alpha)}
\to - \sum_{i=1}^d p_i \log(p_i),
\end{align*}
proving eq. (\ref{eq:kappa-limit}).

Since $ \alpha \mapsto \lambda_\alpha $ is decreasing,
we have $ \lambda_+(\Lambda) < \infty $.
Then, as $ 0 < \lambda_-(\Lambda) \leq \lambda_+(\Lambda) < \infty $,
we have $ \eta(A) \leq \eta(\Lambda) < \infty $, and so $ \eta(A) / \ell(A) \to 0 $.
By applying Theorem \ref{thm:mean-upper},
we may derive eq. (\ref{eq:mean-limit}):
\begin{align*}
\limsup_{n \to \infty} \frac{\lambda_-(\Lambda) \expect{\tau_{\low{n \alpha}}}}{n \ell(\alpha)}
& \leq  \lim_{n \to \infty} \frac{\left(\sqrt{\kappa(\low{n \alpha}) + \eta(\Lambda)} + \sqrt{\ell(\low{n \alpha})}\right)^2}{n \ell(\alpha)} \\
& =  \lim_{n \to \infty} \left(\sqrt{\frac{\kappa(\low{n \alpha}) + \eta(\Lambda)}{\ell(\low{n \alpha})}} + 1\right)^2 \\
& =  \left(1 + \sqrt{- \sum_i p_i \log(p_i)}\right)^2. \qedhere
\end{align*}
\end{proof}

\subsection{Proofs for the Shape Theorem for Partially Ordered Monoids}
\label{subsect:proofs-monoids}
We would like to establish the existence of a shape function.
In particular, we would like $ \tau_{A^n} / n $ to converge in some appropriate sense.
Recall the definition of $ AB $ from Definition \ref{defn:monoid-set-mult}.
First, we prove that $ AB $ and, hence, $ A^n $ lie in $ L(\Lambda) $.

\begin{lemma} \label{lem:set-mult-well-defined}
For any $ A, B \in L(\Lambda) $ and $ n \in \mathbb{N}_0 $, we have $ AB, A^n \in L(\Lambda) $.
\end{lemma}
\begin{proof}
If the result holds for $ AB $
then, by induction, the result follows for $ A^n $.
Since $ AB $ is a lower set by construction,
it suffices to show that $ AB $ is finite.

To this end, let $ \gamma \in \maxim(AB) $.  Then,
there exist $ a \in A $ and $ b \in B $ such that $ \gamma \leq ab $.
Since $ A $ and $ B $ are finite, there exist $ \alpha \in \maxim(A) $ and $ \beta \in \maxim(B) $
with $ a \leq \alpha $ and $ b \leq \beta $.
By the assumed compatibility, $ \gamma \leq ab \leq \alpha \beta \in AB $
which, by maximality of $ \gamma $, gives $ \gamma = \alpha \beta $.
Thus, $ \maxim(AB) \subseteq \{\alpha \beta \,:\, \alpha \in \maxim(A), \beta \in \maxim(B)\} $.  Hence,
\begin{equation*}
AB = \bigcup_{\gamma \in \maxim(AB)*} \low{\gamma}
\subseteq \bigcup_{\alpha \in \maxim(A)} \bigcup_{\beta \in \maxim(B)} \low{\alpha \beta},
\end{equation*}
where $ \maxim(A) $ and $ \maxim(B) $ are finite, since $ A $ and $ B $ are finite,
and $ \low{\alpha \beta} $ is finite by local finiteness of $ \Lambda $.
Therefore, $ AB $ is finite and $ AB \in L(\Lambda) $.
\end{proof}

Via Theorem \ref{thm:var-sublinear}, we may obtain $ L^2 $ convergence
of $ \tau_{A^n} / n $ as long as $ \expect{\tau_{A^n}} / n $ is convergent.
We will use the superadditivity of $ \expect{\tau_A} $, $ \ell(A) $ and $\kappa(A)$,
as well as the subadditivity of $ \ell_*(A) $ to establish
the desired convergence of $ \expect{\tau_{A^n}} / n $.

\begin{lemma} \label{lem:tau-super-add}
Let $ A, B \in L(\Lambda) $.
Then, 
\begin{equation*}
\kappa(A) + \kappa(B) \leq \kappa(AB), \quad
\ell(A) + \ell(B) \leq \ell(AB), \quad
\text{and} \quad
\ell_*(AB) \leq \ell_*(A) + \ell_*(B).
\end{equation*}
Additionally, suppose for all $ x, y \in \Lambda $ with $ x \leq y $ that $ \lambda_x \geq \lambda_y $.
Then, for any $ A, B \in L(\Lambda) $, we have $ \expect{\tau_A} + \expect{\tau_B} \leq \expect{\tau_{AB}} $.
\end{lemma}
\begin{proof}
If either $A$ or $B$ are empty, the claims hold.  Otherwise, let both $A$ and $B$ be non-empty sets.
For any $ \pi = (\pi_1, \ldots, \pi_k) \in \Pi(A) $
and $ \xi = (\xi_1, \ldots, \xi_n) \in \Pi(B) $, define $ \pi \circ \xi \in \Pi(AB) $ as
\begin{equation*}
\pi \circ \xi := (\xi_1, \ldots, \xi_n, \pi_1 \xi_n, \ldots, \pi_k \xi_n).
\end{equation*}
If $ \xi \in \Pi_m(B) $, notice that any extension of $ \xi $ is not in $ \Pi(B) $,
and so $ (\xi_1, \ldots, \xi_n, \pi_1 \xi_n) \notin \Pi(B) $.
Thus, $ \max \{i \in [1, k+n] \,:\, (\pi \circ \xi)_i \in B\} = n = \ell(\xi) $.
So, one can recover the location of the division between $ \xi $ and $ \pi $ from $ \pi \circ \xi $.
Further, from the group structure of $\overline{\Lambda} $ by multiplying by $ \xi_n^{-1} $,
we conclude $ (\pi, \xi) \mapsto \pi \circ \xi $ is injective for $ (\pi, \xi) \in \Pi(A) \times \Pi_m(B) $.

We will define $ \Pi_m(A) \circ \Pi_m(B) := \{\pi \circ \xi \,:\, \pi \in \Pi_m(A), \xi \in \Pi_m(B)\} $.
Due to the injectivity, $ |\Pi_m(A) \circ \Pi_m(B)| = |\Pi_m(A) \times \Pi_m(B)| = |\Pi_m(A)| \cdot |\Pi_m(B)| $.

First, we have
\begin{align*}
\ell(A) + \ell(B) & =  \max_{\pi \in \Pi_m(A)} \ell(\pi) + \max_{\xi \in \Pi_m(B)} \ell(\xi)
= \max_{\substack{\pi \in \Pi_m(A) \\ \xi \in \Pi_m(B)}} (\ell(\pi) + \ell(\xi)) \\
& =  \max_{\zeta \in \Pi_m(A) \circ \Pi_m(B)} \ell(\zeta)
\leq \max_{\zeta \in \Pi(AB)} \ell(\zeta) = \ell(AB),
\end{align*}
implying $ \ell $ is superadditive.

For any $ x \in \Lambda $, we define $ |x| := \min_{\pi \in \Pi_m(\low{x})} \ell(\pi) $
so that $ \ell_*(A) = \max_{x \in A} |x| $.
For $ x, y \in \Lambda $, we have
\begin{align*}
|x| + |y|
& = \min_{\pi \in \Pi_m(\low{x})} \ell(\pi) + \min_{\xi \in \Pi_m(\low{y})} \ell(\xi)
= \min_{\substack{\pi \in \Pi_m(\low{x}) \\ \xi \in \Pi_m(\low{y})}} (\ell(\pi) + \ell(\xi)) \\
& =  \min_{\zeta \in \Pi_m(\low{x}) \circ \Pi_m(\low{y})} \ell(\zeta)
\geq \min_{\zeta \in \Pi(\low{xy})} \ell(\zeta) = |xy|.
\end{align*}
Then, for any $ A, B \in L(\Lambda) $,
using the fact that $ |x| \leq |y| $ whenever $ x \leq y $,
\begin{align} \label{eq:ell-star-subadd}
\ell_*(AB) & =  \max \{|z| \,:\, z \in AB\}
= \max \{|xy| \,:\, x \in A, y \in B\} \nonumber \\
& \leq  \max\{|x| + |y| \,:\, x \in A, y \in B\} \\
& =  \max\{|x| \,:\, x \in A\} + \max\{|y| \,:\, y \in B\} \nonumber \\
& \leq  \ell_*(A) + \ell_*(B), \nonumber
\end{align}
demonstrating that $ \ell_* $ is subadditive.

Second, we show that no path in $ \Pi_m(A) \circ \Pi_m(B) $ is a subpath of another.
Suppose $ \pi^{(1)}, \pi^{(2)} \in \Pi_m(A) $ and $ \xi^{(1)}, \xi^{(2)} \in \Pi_m(B) $
with $ \zeta^{(1)} := \pi^{(1)} \circ \xi^{(1)} $ and $ \zeta^{(2)} := \pi^{(2)} \circ \xi^{(2)} $
where $ \zeta^{(1)} \subseteq \zeta^{(2)} $.
Then,
\begin{equation*}
\xi_{\ell(\xi^{(1)})}^{(1)} = \max \left\{\zeta_i^{(1)} \in B \,:\, i \in [1, \ell(\zeta^{(1)})]\right\}
\leq \max \left\{\zeta_i^{(2)} \in B \,:\, i \in [1, \ell(\zeta^{(2)})]\right\} = \xi_{\ell(\xi^{(2)})}^{(2)}
\end{equation*}
which, by the maximality of $ \xi^{(1)} $ and $ \xi^{(2)} $, implies $ \xi^{(1)}_{\ell(\xi^{(1)})} = \xi^{(2)}_{\ell(\xi^{(2)})} $
and so $ \xi^{(1)} = \xi^{(2)} $.
Then, using $ n := \ell(\xi^{(1)}) = \ell(\xi^{(2)}) $,
we have $ \zeta^{(1)}_n = \xi^{(1)}_n = \xi^{(2)}_n = \zeta^{(2)}_n $.  Hence, 
\begin{equation*}
\pi^{(1)} = \left(\zeta^{(1)}_{n+1} (\zeta^{(1)}_n)^{-1}, \ldots, \zeta^{(1)}_{\ell(\zeta^{(1)})} (\zeta^{(1)}_n)^{-1}\right)
\subseteq \left(\zeta^{(2)}_{n+1} (\zeta^{(2)}_n)^{-1}, \ldots, \zeta^{(2)}_{\ell(\zeta^{(2)})} (\zeta^{(2)}_n)^{-1}\right) = \pi^{(2)},
\end{equation*}
implying $ \pi^{(1)} = \pi^{(2)} $ by maximality of $ \pi^{(1)} $ in $ A $.

Therefore, none of the elements of $ \Pi_m(A) \circ \Pi_m(B) $ are strict subpaths of each other.
Thus, we can extend each to a unique maximal path in $ AB $.
Then, $ |\Pi_m(A) \circ \Pi_m(B)| \leq |\Pi_m(AB)| $,
and $ \log |\Pi_m(A)| + \log |\Pi_m(B)| \leq \log |\Pi_m(AB)| $.
Hence, $\kappa$ is superadditive.

Third, we assume $ \lambda_x \geq \lambda_y $ for all $ x, y \in \Lambda $ with $ x \leq y $.
For every $ \pi = (\pi_1, \ldots, \pi_k) \in \Pi_m(A) $ and $ \xi = (\xi_1, \ldots, \xi_n) \in \Pi_m(B) $,
consider $ \zeta := \pi \circ \xi $.
Then, since the identity is a minimum, $ 1_\Lambda \leq \xi_n $
and by compatibility $ \pi_i \leq \pi_i \xi_n = \zeta_{i+n} $,
we have $ \lambda_{\pi_i} \geq \lambda_{\zeta_{i+n}} $.

Now, suppose $ \{G_\alpha\}_{\alpha \in \Lambda} $ are defined as in Proposition \ref{prop:tau-lpp},
and $ \{\tilde{G}_\alpha\}_{\alpha \in \Lambda} $ is an independent copy of the process
(on a common probability space, using the same symbol $ \mathbb{E} $ for the expectation).
Then, $ \lambda_{\pi_i} \geq \lambda_{\zeta_{i+n}} $ implies $ \tilde{G}_{\pi_i} \preceq G_{\zeta_{i+n}} $
where we recall $ \preceq $ represents stochastic ordering.
Also, as $ \zeta_i = \xi_i $ for all $ i \in [1, n] $, we have $ G_{\xi_i} \preceq G_{\zeta_i} $.
From the independence of $ \{G_\alpha\}_{\alpha\in \Lambda} $ and $ \{\tilde{G}_\alpha\}_{\alpha\in \Lambda} $
as well as the monotonicity of the maximum and summation, we obtain
\begin{equation*}
\max_{\pi \in \Pi_m(A)} \max_{\xi \in \Pi_m(B)} \left[\sum_{i=1}^{\ell(\pi)} \tilde{G}_{\pi_i} + \sum_{i=1}^{\ell(\xi)} G_{\xi_i}\right]
\preceq \max_{\pi \in \Pi_m(A)} \max_{\xi \in \Pi_m(B)} \sum_{i=1}^{\ell(\zeta)} G_{\zeta_i} \leq \tau_{AB}.
\end{equation*}
Using Proposition \ref{prop:tau-lpp}, this implies $ \tilde{\tau}_A + \tau_B \preceq \tau_{AB} $
where $ \tilde{\tau}_A $ is the stopping time of $ A $
for the process associated with $ \{\tilde{G}_\alpha\}_{\alpha\in \Lambda} $.
Hence, $ \expect{\tau_A} + \expect{\tau_B} = \expect{\tilde{\tau}_A + \tau_B} \leq \expect{\tau_{AB}} $.
\end{proof}

\pagebreak[1]
\begin{proof}[Proof of Theorem \ref{thm:monoid-shape-func}]
Let $ S $ be the finite set of upper neighbors of $ 1_\Lambda $.
Consider $ x, y \in \Lambda $ where $ y $ is an upper neighbor of $ x $.
Regarding $ x, y \in \overline{\Lambda} $, we have $ y x^{-1} > 1_\Lambda $, and so $ y x^{-1} \in \Lambda $ by definition of $\Lambda$.
For any $ z $ with $ 1_\Lambda < z \leq y x^{-1} $, we have $ x < zx \leq y $,
implying $ zx = y $ as $ y $ is an upper neighbor of $ x $.
Thus, $ y x^{-1} $ is an upper neighbor of $ 1_\Lambda $ and $ y x^{-1} \in S $.
Therefore, $ x $ has at most $ |S| $ upper neighbors.
Because this holds for all $ x \in \Lambda $
and $ 1_\Lambda $ is the unique minimal element of $ \Lambda $,
by Lemma \ref{lem:kappa-bnd} with $ d = |S| \geq 1 $,
for all $ A \in L(\Lambda) $, we have $ \kappa(A) \leq \ell(A) \, \log |S| $.

From Lemma \ref{lem:tau-super-add}, we know $ \{\ell(A^n)\}_{n=1}^\infty $, $ \{\kappa(A^n)\}_{n=1}^\infty $,
and $ \{\expect{\tau_{A^n}}\}_{n=1}^\infty $ are superadditive sequences.
Also, by subadditivity of $ \ell_* $ and steadiness of $ \Lambda $,
we have $ \ell(A^n) \leq C \ell_*(A^n) \leq nC \ell_*(A) $.
Then, the sequence $ \{\ell(A^n) / n\}_{n=1}^\infty $ is bounded by $ C \ell_*(A) $.
and the sequence $ \{\kappa(A^n) / n\}_{n=1}^\infty $ is bounded by $ C \ell_*(A) \, \log |S| $.
Applying Theorem \ref{thm:mean-upper}, for $ n \geq 1 $, we have
\begin{align*}
\frac{\lambda_-(\Lambda) \expect{\tau_{A^n}}}{n}
& \leq  \frac{\left(\sqrt{\kappa(A^n) + \eta(A^n)} + \sqrt{\ell(A^n)}\right)^2}{n} \\
& \leq  C \ell_*(A) \left(1 + \sqrt{\log |S| + \eta(\Lambda)}\right)^2.
\end{align*}
Since, $ \alpha \mapsto \lambda_\alpha $ is decreasing, $ \lambda_+(\Lambda) < \infty $.
Then, as $ 0 < \Lambda_-(\Lambda) \leq \Lambda_+(\Lambda) < \infty $,
we have $ \eta(\Lambda) / n \to 0 $ and so, by superadditivity (Fekete's lemma), the following limits exist:
\begin{align*}
\ell_\infty & :=  \lim_{n \to \infty} \frac{\ell(A^n)}{n} \leq C \ell_*(A),
\quad \quad
\kappa_\infty := \lim_{n \to \infty} \frac{\kappa(A^n)}{n} \leq C \ell_*(A) \, \log |S|, \\
&  \qquad\qquad \text{and} \quad \lim_{n \to \infty} \frac{\lambda_-(\Lambda) \expect{\tau_{A^n}}}{n}
\leq \left(\sqrt{\kappa_\infty} + \sqrt{\ell_\infty}\right)^2.
\end{align*}
Lastly, define $ g(A) := \lim_{n \to \infty} \expect{\tau_{A^n}} / n $
so that, by Theorem \ref{thm:var-sublinear},
\begin{equation*}
\lim_{n \to \infty} \expect{\left(\frac{\tau_{A^n}}{n} - g(A)\right)^2}
= \lim_{n \to \infty} \frac{\var{\tau_{A^n}}}{n^2}
\leq \lim_{n \to \infty} \frac{1}{n} \cdot \frac{\expect{\tau_{A^n}}}{n}
= \lim_{n \to \infty} \frac{g(A)}{n} = 0,
\end{equation*}
proving $ \tau_{A^n} / n \to g(A) $ in $ L^2 $.
\end{proof}

\pagebreak[1]
\begin{proof}[Proof of Lemma \ref{lem:steady-local-fin}]
First, $ 1_\Lambda $ is the unique minimal element of $ \Lambda $
fulfilling the second condition of Definition \ref{defn:ord-local-fin}.
Second, take $ S $ to be the upper neighbors of $ 1_\Lambda $.
Because $ \Lambda $ is finitely generated, $ S $ is finite.
By the first part of the proof of Theorem \ref{thm:monoid-shape-func},
there are at most $ |S| $ upper neighbors of any $ \alpha \in \Lambda $,
fulfilling the third condition of Definition \ref{defn:ord-local-fin}.

Third, for any $ \alpha \in \Lambda $,
we can write $ \alpha = s_1 \cdots s_n $ for $ s_1, \ldots, s_n \in S $.
Thus, the path $ \zeta = (\zeta_1, \ldots, \zeta_{n+1}) $ with $ \zeta_k = s_1 \cdots s_{k-1} $ is maximal in $ \low{\alpha} $
so $ \min_{\pi \in \Pi_m(\low{\alpha})} \ell(\pi) \leq n + 1 $.
By steadiness,
\begin{equation*}
\max_{\pi \in \Pi_m(\low{\alpha})} \ell(\pi) \leq C \min_{\pi \in \Pi_m(\low{\alpha})} \leq C(n + 1).
\end{equation*}
Then, for every $ x \in \low{\alpha} $, we can write $ x = \tilde{s}_1 \cdots \tilde{s}_m $
with $ \tilde{s}_1, \ldots, \tilde{s}_m \in S $.
We define the path $ \tilde{\xi}_x = (1_\Lambda, \tilde{s}_1, \tilde{s}_1 \tilde{s}_2, \ldots, x) \in \Pi(\low{\alpha}) $.
Then, we can extend $ \tilde{\xi}_x $ to a maximal path $ \xi_x \in \Pi_m(\low{\alpha}) $.
Thus, every element of $ \low{\alpha} $ can be mapped into $ \Pi_m(\low{\alpha}) $
and for every $ \pi \in \Pi_m(\low{\alpha}) $, there are $ \ell(\pi) \leq C(n+1) $ elements which map to it.
Hence, $ |\low{\alpha}| \leq C(n+1) \cdot |\Pi_m(\low{\alpha})| $.
Because the maximum length of a path in $ \Pi_m(\low{\alpha}) $ is $ C(n+1) $
and there are $ |S| $ different generators,
we also know that $ |\Pi_m(\low{\alpha})| \leq |S|^{C (n+1)} $.
Finally,
\begin{equation*}
|\low{\alpha}| \leq C(n+1) \cdot |\Pi_m(\low{\alpha})| \leq C(n + 1) \cdot |S|^{C (n + 1)} < \infty
\end{equation*}
fulfilling the first condition of Definition \ref{defn:ord-local-fin}.
\end{proof}

\pagebreak[1]
\begin{proof}[Proof of Corollary \ref{coro:fin-prsnt-steady}]
First, we will construct a group homomorphism $ \varphi : \overline{\Lambda} \to \mathbb{R} $
such that $ \varphi(x) > 0 $ for all $ x \in \Lambda \setminus \{1_{\overline{\Lambda}}\} $.
Because of the commutativity of $ \overline{\Lambda} $, we will write the group operation additively.

Let $ S $ and $ R $ be a finite set of generators and relations, respectively, of $ \overline{\Lambda} $.
Without loss of generality, we may choose $ R $ to not contain the trivial relation $ 0 = 0 $
and $ S $ so that it does not contain the identity and is a generating set for $ \Lambda $ as a monoid.
We enumerate the generators as $ S = \{s_1, \ldots, s_d\} $.
Similarly, we enumerate the relations $ R = \{r^{(1)}, \ldots, r^{(n)}\} $
such that for each $ i $, $ r_1^{(i)} s_1 + \ldots + r_d^{(i)} s_d = 0 $
where $ r_1^{(i)}, \ldots, r_d^{(i)} \in \mathbb{Z} $.

Fix $ j_0 \in [1, d] $.
Now, we construct a matrix $ M $ from the relations
excluding the column corresponding to $ j_0 $
so $ M := (r_j^{(i)})_{i, j \neq j_0} \in \mathbb{Z}^{n \times (d - 1)} $.
Then, we define the vector $ b = (-r_{j_0}^{(i)})_i \in \mathbb{Z}^n $.
By the Farkas' lemma for rational matrices (see Corollary 3.5 of \cite{Scow06}),
either $ M c = b $ has a solution $ c = (c_j)_{j \neq j_0} \in \mathbb{Q}^{d-1} $ with $ c_j \geq 0 $
or $ M^T w \leq 0 $ has a solution $ w \in \mathbb{Q}^n $ with $ b^T w > 0 $, but not both.
Assume for sake of contradiction that there exists a solution $ w = (w_i) \in \mathbb{Q}^n $ to the latter
so $ M^T w \leq 0 $ and $ b^T w > 0 $.
Without loss of generality, by scaling $ w $ appropriately, we may assume $ w \in \mathbb{Z}^n $.
Then, $ M^T w \leq 0 $ means that $ w_1 r_j^{(1)} + \ldots + w_n r_j^{(n)} \leq 0 $ for all $ j \neq j_0 $.
Additionally, $ b^T w > 0 $ means that $ -w_1 r_{j_0}^{(1)} - \ldots - w_n r_{j_0}^{(n)} > 0 $
and, because it is an integer, it must be greater than or equal to one.
Now, using the ordering on $ \Lambda $,
\begin{align*}
0 & =  \sum_{i=1}^n w_i \cdot 0 = \sum_{i=1}^n w_i \sum_{j=1}^d r_j^{(i)} s_j \\
& =  \left(w_1 r_{j_0}^{(1)} + \ldots + w_n r_{j_0}^{(n)}\right) s_{j_0}
+ \sum_{j \neq j_0} \left(w_1 r_j^{(1)} + \ldots + w_n r_j^{(n)}\right) s_j
\leq_\Lambda -s_{j_0} <_\Lambda 0,
\end{align*}
which is a contradiction.
Therefore, we conclude there exists a solution $ c = (c_j)_{j \neq j_0} \in \mathbb{Q}^{d-1} $
with $ Mc = b $ and $ c_j \geq 0 $.
Define $ \varphi_{j_0} : \Lambda \to \mathbb{Q} $ as
\begin{equation*}
\varphi_{j_0}\left(\sum_{j=1}^d x_j s_j\right) = x_{j_0} + \sum_{j \neq j_0} x_j c_j
\end{equation*}
for any $ x_1, \ldots, x_d \in \mathbb{Z} $.
Because $ Mc = b $, we know $ \varphi_{j_0}\left(r_1^{(i)} s_1 + \ldots + r_d^{(i)} s_d\right) = 0 $ for all $ i $.
So, $ \varphi_{j_0} $ is well-defined on $ \overline{\Lambda} $.
For any $ x = x_1 s_1 + \ldots + x_d s_d \in \Lambda $ as $ x_1, \ldots x_d \geq 0 $,
we have $ \varphi_{j_0}(x) \geq 0 $.
Also, $ \varphi_{j_0}(s_{j_0}) = 1 > 0 $.

Now, we define $ \varphi : \Lambda \to \mathbb{R} $
as $ \varphi(x) = \varphi_1(x) + \ldots + \varphi_d(x) $
so that $ \varphi(s) > 0 $ for all $ s \in S $.
Define $ m_- := \min_{s \in S} \varphi(s) > 0 $ and $ m_+ := \max_{s \in S} \varphi(s) > 0 $.
For any $ x \in \Lambda $, let $ \pi \in \Pi_m(\low{x}) $.
Then, for each $ i \in [1, \ell(\pi)] $, $ \pi_i - \pi_{i-1} $ is minimal in $ \Lambda \setminus \{0\} $
and so $ \pi_i - \pi_{i-1} \in S $.
Thus,
\begin{equation*}
\varphi(x) = \sum_{i=1}^{\ell(\pi)} \varphi(\pi_i - \pi_{i-1}) \geq \sum_{i=1}^{\ell(\pi)} m_- = m_- \ell(\pi)
\end{equation*}
and, similarly, $ \varphi(x) \leq m_+ \ell(\pi) $.
Hence, $ \frac{1}{m_+} \varphi(x) \leq \ell(\pi) \leq \frac{1}{m_-} \varphi(x) $.
Therefore,
\begin{equation*}
\max_{\pi \in \Pi_m(\low{x})} \ell(\pi) \leq \frac{1}{m_-} \varphi(x)
\leq \frac{m_+}{m_-} \min_{\pi \in \Pi_m(\low{x})} \ell(\pi)
\end{equation*}
and so $ \Lambda $ is steady with $ C = \frac{m_+}{m_-} $.
By Lemma \ref{lem:steady-local-fin}, we know $ \Lambda $ is also locally finite.
We conclude, by Theorem \ref{thm:monoid-shape-func}, eq. (\ref{eq:shape-ineq}) holds for all $ A \in L(\Lambda) $.
\end{proof}

\begin{remark}
 \label{rem:strict-pos-homo}
A concrete strictly positive homomorphism $ \varphi $ with respect to the cone
in Example \ref{exm:fin-prsnt} given in Figure \ref{fig:cone-in-lattice} can be found.
One could choose $ \varphi(x_1 a + x_2 b + x_3 c) = (2/3, 1, 1) \cdot (x_1, x_2, x_3) = (2/3) x_1 + x_2 + x_3 $
which is strictly positive on the non-zero elements of $ \Lambda = \langle a, b, c \,|\, 3a = b + c \rangle $.
Since $ \varphi(3a - b - c) = 0 $, the function $ \varphi $ is well-defined.
\end{remark}

\pagebreak[1]
\begin{proof}[Proof of Corollary \ref{coro:non-comm-steady}]
For each word $ w = w_1 \ldots w_n \in \overline{\Lambda} $
composed of elements of $ S \cup S^{-1} $,
consider $ f(w) = |\{i \in [1, n] \,:\, w_i \in S\}| - |\{i \in [1, n] \,:\, w_i \in S^{-1}\}| $,
the number of positive generators minus the number of negative generators.
Then, by assumption, $ f(r) = 0 $ for any relation $r\in R$.
Hence, if two words $ w_1 $ and $ w_2 $ are equivalent under $R$,
then $ f(w_1) = f(w_2)$.
Therefore, $ f $ is well-defined as a function on $ \overline{\Lambda} $.

Suppose now $ x \in \Lambda $.
For every $ \pi = (\pi_1, \ldots, \pi_n) \in \Pi_m(\low{x}) $,
we have
\begin{equation*}
x = \pi_1 \cdot (\pi_1^{-1} \pi_2) \cdots (\pi_{n-1}^{-1} \pi_n)
\end{equation*}
where $ \pi_i^{-1} \pi_{i+1} \in S $ and $ \pi_1 \in S $.
Thus, $ w = \pi_1 \cdot (\pi_1^{-1} \pi_2) \cdots (\pi_{n-1}^{-1} \pi_n) $
is a word and $ f(x) = f(w) = n = \ell(\pi) $.
Hence, $ \ell(\pi) $ is the same for all $ \pi \in \Pi_m(\low{x}) $.
We conclude
\begin{equation*}
\max_{\pi \in \Pi_m(\low{x})} \ell(\pi) = \min_{\pi \in \Pi_m(\low{x})} \ell(\pi),
\end{equation*}
meaning $ \Lambda $ is steady with $ C = 1 $.
By Lemma \ref{lem:steady-local-fin}, we know $ \Lambda $ is also locally finite.
Moreover, by Theorem \ref{thm:monoid-shape-func}, eq. (\ref{eq:shape-ineq}) holds for all $ A \in L(\Lambda) $.
\end{proof}

\subsection{Proofs of Proposition \ref{prop:tau-lpp}}
\label{subsect:proof-tau-lpp}

We now give two proofs of Proposition \ref{prop:tau-lpp}.
The first is more `standard', following from the construction of Markov jump processes,
inspired by the proof of Proposition 1.1 of \cite{Sepp09}.
for the discrete time growth model on $ \mathbb{N}_0^2 $ with independent Geometric weights.
The second is self-contained and of a different character.
Let $ \tilde{\mathbb{P}} $ and $ \tilde{\mathbb{E}} $
denote the probability measure and expectation operator, respectively,
for the probability space of the process $ Y_t $.

\begin{proof}[Proof of Proposition \ref{prop:tau-lpp} via Markov Jump Process Construction]
With respect to the process $ Y_t $, let $ \{T_i\}_{i \geq 0} $ and $ \{Z_i\}_{i \geq 0} $ be
the event times and the states visited at these event times.
Here, $ T_0 = 0 $ and $ Z_0 = \emptyset $.  
By the construction of Markov jump processes on countable state spaces,
to show that $ Y_t $ is a Markov process with generator $ \gener $,
it is enough to verify the following:
\begin{itemize}
\item[(a)] $ \{Z_i\} $ is a discrete time Markov chain with transition probability
\begin{equation*}
\probTWhen{\big. Z_{i+1} = A \cup \alpha}{Z_i = A} = \lambda_\alpha / \left(\sum_{\beta\in \maxim^*(A)} \lambda_\beta\right)
\end{equation*}
for $ \alpha \in \maxim^*(A) $.
\item[(b)] $ \{T_{i+1} - T_i\}_{i \geq 0} \,|\, \{Z_i\}_{i \geq 0} $ are independent exponential random variables

with parameters $ \left\{\sum_{\beta \in \maxim^*(Z_i)} \lambda_\beta \right\}_{i \geq 0} $.
\end{itemize}

Let $\emptyset = A_0 \subset A_1 \subset A_2 \subset \cdots \subset A_k $ be lower sets in $ \Lambda $
such that $ A_j \setminus A_{j-1} = \alpha_j $.
In particular, $ A_j = \bigcup_{1 \leq i \leq j} \alpha_i $ for $ 1 \leq j \leq k $.
Items (a) and (b) will follow from verifying
\begin{align}
\label{eq:skeleton-prob}
&  \probT{Z_k=Z_{k-1}\cup\alpha_k, T_k-T_{k-1}>t_k, \ldots, Z_2=Z_1\cup\alpha_2, T_2-T_1>t_2, Z_1=\alpha_1, T_1>t_1} \nonumber \\
&  \qquad\qquad = \prod_{j=1}^k \frac{\lambda_{\alpha_j}}{\sum_{\beta \in \maxim^*(A_{j-1})} \lambda_\beta}
\exp\left(-t_j\sum_{\beta\in \maxim^*(A_{j-1})} \lambda_\beta \right).
\end{align}

Indeed, by taking $ t_j \equiv 0 $, we see that item (a) would hold.
Also, item (b) would hold since then, for any $ L \geq k $,
\begin{align*}
&  \probTWhen{\bigcup_{j=1}^k \{T_j - T_{j-1} > t_j\}}{\{Z_j\}_{j=0}^L} \\
&  \quad = \probTWhen{\bigcup_{j=1}^k \{T_j - T_{j-1}> t_j\} \cup \bigcup_{j=k+1}^L\{T_j - T_{j-1} > 0\}}{\{Z_j\}_{j=0}^L} \\
&  \quad = \prod_{j=1}^k \exp\left( -\sum_{\beta\in \maxim^*(Z_{j-1})} \lambda_\beta \right).
\end{align*}

We now verify (a) and (b).
Note for $ \alpha \in \Lambda $ that $ G_\alpha = \chi_{\low{\alpha}} - \chi_{\low{\alpha} \setminus \alpha} $
is the time between when $ \alpha $ becomes first accessible and when it is achieved.
Also, when $ \alpha \in \maxim^*(\emptyset) $, we observe $ \chi_{\low{\alpha} \setminus \alpha} = \chi_\emptyset = 0 $.

The left-hand side of eq. (\ref{eq:skeleton-prob}) equals
\begin{align}
\label{eq:expanded-skeleton-prob}
&  \tilde{\mathbb{P}}\Bigg(\chi_{\low{\alpha_k}} - \chi_{A_{k-1}} 
= \min_{\beta \in \maxim^*(A_{k-1})} \{\chi_{\low{\beta}}- \chi_{A_{k-1}})\} > t_k, \nonumber \\
&  \qquad\qquad \chi_{\low{\alpha_{k-1}}} - \chi_{A_{k-2}}
= \min_{\beta \in \maxim^*(A_{k-2}) \setminus \maxim^*(A_{k-1})} \{\chi_{\low{\beta}} - \chi_{A_{k-2}}\} > t_{k-1}, \nonumber \\
&  \qquad\qquad \ldots, \, \chi_{\low{\alpha_1}} = \min_{\beta \in \maxim^*(A_0) \setminus \maxim^*(A_{k-1})} \chi_{\low{\beta}} > t_1 \Bigg) \\
& =  \tilde{\mathbb{P}}\Bigg(G_{\alpha_k} - (\chi_{A_{k-1}} - \chi_{\low{\alpha_k}\setminus \alpha_k})
= \min_{\beta\in \maxim^*(A_{k-1})}\{G_\beta - (\chi_{A_{k-1}}-\chi_{\low{\beta}\setminus \beta})\} > t_k, \nonumber \\
&  \,\quad G_{\alpha_{k-1}} - \left(\chi_{A_{k-2}} - \chi_{\low{\alpha_{k-1} \setminus \alpha_{k-1}}} \right) \nonumber \\
&  \qquad\qquad = \min_{\beta\in \maxim^*(A_{k-2}) \setminus \maxim^*(A_{k-1})} \{G_\beta - (\chi_{A_{k-2}} - \chi_{\low{\beta} \setminus \beta})\} > t_{k-1}, \nonumber \\
&  \quad \ldots, \, G_{\alpha_1} = \min_{\beta \in \maxim^*(A_0) \setminus \maxim^*(A_{k-1})} G_\beta > t_1 \Bigg). \nonumber
\end{align}
By definition, for $ \beta \in \maxim^*(A_j) $ and $ 1 \leq j \leq k-1 $,
the difference $ \chi_{A_j} - \chi_{\low{\beta} \setminus \beta} $
is expressible in terms of $ \{G_\alpha \,:\, \alpha \in A_{k-1}\} $.
Therefore, the event inside the probability in eq. (\ref{eq:expanded-skeleton-prob})
is expressible all in terms of $ \{G_\alpha\}_{\alpha \in A_{k-1} \cup \maxim^*(A_{k-1})} $.

We now decompose the event.
Observe
\begin{equation*}
B_0 = \{G_{\alpha_k} - (\chi_{A_{k-1}} - \chi_{\low{\alpha_k} \setminus \alpha_k})
= \min_{\beta \in \maxim^*(A_{k-1})}\{G_\beta - (\chi_{A_{k-1}}-\chi_{\low{\beta}\setminus \beta})\}>t_k\big\}
\end{equation*}
is a subset of $ B_1 = \cap_{\beta \in \maxim^*(A_{k-1})} \{G_\beta > \chi_{A_{k-1}}-\chi_{\low{\beta}\setminus\beta}\} $.
Let
\begin{align*}
C & =  \Big\{G_{\alpha_{k-1}} - (\chi_{A_{k-2}} - \chi_{\low{\alpha_{k-1}} \setminus \alpha_{k-1}}) \\
&  \qquad = \min_{\beta \in \maxim^*(A_{k-2}) \setminus \maxim^*(A_{k-1})} \{G_\beta - (\chi_{A_{k-2}} - \chi_{\low{\beta} \setminus \beta})\} > t_{k-1}, \\
&  \qquad \ldots, G_{\alpha_1} = \min_{\beta \in \maxim^*(A_0) \setminus \maxim^*(A_{k-1})} G_\beta > t_1 \Big\}.
\end{align*}
Observe that $ C $ involves only $ \mathcal{F} = \sigma \{G_\alpha : \alpha\in A_{k-1}\} $,
and $ B_0, B_1 $ involve only $ \sigma\{G_\alpha : \alpha \in \maxim^*(A_{k-1})\} $.
Since, $ \mathcal{F} $ is independent of $\sigma\{G_\beta : \beta \in \maxim^*(A_{k-1})\} $,
we have eq. (\ref{eq:expanded-skeleton-prob}) equals
\begin{align*}
\probT{B_0 \cap C} = \expectT{B_0 \cap B_1 \cap C}
& =  \expectT{1_C \probTWhen{B_0 \cap B_1}{\mathcal{F}}} \\
& =  \expectT{1_C \probTWhen{B_0}{B_1, \mathcal{F}} \probTWhen{B_1}{\mathcal{F}}}.
\end{align*}
Since $ \{G_\alpha : \alpha \in \maxim^*(A_{k-1})\} $ are independent random variables,
and also independent of $ \mathcal{F} $, by properties of exponential distributions,
we have that
\begin{equation*}
\probTWhen{B_0}{B_1, \mathcal{F}} = \frac{\lambda_{\alpha_k}}{\sum_{\beta \in \maxim^*(A_{k-1})} \lambda_\beta} \exp\left(-t_k \sum_{\beta\in \maxim^*(A_{k-1})} \lambda_\beta\right).
\end{equation*}
Hence, we have
\begin{equation*}
\probT{B_0 \cap B_1 \cap C} = \frac{\lambda_{\alpha_k}}{\sum_{\beta \in \maxim^*(A_{k-1})} \lambda_\beta}
\exp\left(-t_k \sum_{\beta \in \maxim^*(A_{k-1})} \lambda_\beta\right) \probT{C \cap B_1}.
\end{equation*}

Now, the event $ C \cap B_1 $, expressed in terms of $ \{G_\alpha : \alpha\in A_{k-1}\} \cup \{G_\beta : \beta \in \maxim^*(A_{k-1})\} $,
states that the evolution up to time $ T_{k-1} $ fills in $ \alpha_1, \ldots, \alpha_{k-1} $ in order
and that $ \beta \in \maxim^*(A_{k-1}) $ (those states which are still accessible) are not filled by time $ T_{k-1} $,
and also that $ T_j - T_{j-1} \geq t_j $ for $ 1 \leq j \leq k-1 $.
The event $ C \cap B_1 $ can be re-expressed as
\begin{equation*}
C \cap B_1 = \{Z_{k-1} = Z_{k-2} \cup \alpha_{k-1}, T_{k-1} - T_{k-2} > t_{k-1}, \ldots, Z_1 = \alpha_1, T_1 > t_1\}.
\end{equation*}
Hence, we may iterate and verify the claim in eq. (\ref{eq:skeleton-prob}).
\end{proof}

For the reader's interest, we give an alternate self-contained proof of Proposition \ref{prop:tau-lpp}.

\begin{proof}[Alternate proof of Proposition \ref{prop:tau-lpp}]
Fix some (not necessarily finite) lower set $ D \subseteq \Lambda $.
For any $ t \geq 0 $, let $ \mathcal{F}_t := \sigma(Y_s \,:\, s \leq t) $
be the natural filtration for $ Y_t $
and $ \sigma(Y_t \cap D) $ be the $ \sigma $-algebra generated by $ Y_t \cap D $.
Also, for any $ A \in L(\Lambda) $, let $ \mathcal{G}_B := \sigma(\chi_A \,:\, A \subseteq B) $.
We will now show that $ Y_t \cap D $ is Markov with respect to the filtration $ \mathcal{F}_t $.
That is, for all $ s \leq t $, $ (\sigma(Y_t \cap D) \indep \mathcal{F}_s) \,|\, \sigma(Y_s \cap D) $.
In words, $ \sigma(Y_t \cap D) $ is independent of $ \mathcal{F}_s $ given $ \sigma(Y_s \cap D) $.
Note that $ \sigma(Y_t \cap D) = \sigma(\chi_A \leq t \,:\, A \in L(\Lambda), A \subseteq D) $
and, similarly, $ \mathcal{F}_t = \sigma(\chi_A \leq s \,:\, A \in L(\Lambda), s \leq t) $.

We prove by induction on $ |A| $ that for all lower sets $ A \subseteq D $
and all $ s \geq 0 $, $ (s \lor \chi_A \indep \mathcal{F}_s) \,|\, \sigma(Y_s \cap D) $.
When $ |A| = 0 $, we have $ A = \emptyset $
so $ s \lor \chi_A = s \lor 0 = s $ which is independent of $ \mathcal{F}_s $.
For $ |A| \geq 1 $, either $ |\maxim(A)| > 1 $ or $ |\maxim(A)| = 1 $.
If $ |\maxim(A)| > 1 $ then for all $ \alpha \in A $, $ |\low{\alpha}| < |A| $
so, by the induction hypothesis, $ (s \lor \chi_{\low{\alpha}} \indep \mathcal{F}_s) \,|\, \sigma(Y_s \cap D) $.
Thus, $ (s \lor \chi_A = \max_{\alpha \in A} (s \lor \chi_{\low{\alpha}}) \indep \mathcal{F}_s) \,|\, \sigma(Y_s \cap D) $.

Otherwise, $ |\maxim(A)| = 1 $ so $ \maxim(A) = \{\alpha\} $ and $ A = \low{\alpha} $.
Note $ G_\alpha = \chi_{\low{\alpha}} - \chi_{\low{\alpha} \setminus \alpha} $.
We condition on $ \{\chi_A > s\} = \{G_\alpha > s - \chi_{A \setminus \alpha}\} $,
Note that $ \mathcal{F}_s $ is generated by events of the form $ E := \{\chi_{B_1} \leq s_1, \ldots, \chi_{B_n} \leq s_n\} $
where $ B_1, \ldots, B_n \in L(\Lambda) $ with $ B_1 \subseteq \ldots \subseteq B_n $
and $ s_1, \ldots, s_n \geq 0 $ with $ s_1 \leq \ldots \leq s_n = s $.
Then, $ \chi_{\low{\alpha}} = \chi_A > s $ and $ s \geq \chi_{B_n} $
implies $ \alpha \notin B_n $ and $ B_n \subseteq A \setminus \alpha $.
So, if $ B_n \nsubseteq A \setminus \alpha $ then $ \probTWhen{E}{\chi_A > s} = 0 $
meaning $ (s \lor \chi_A \indep E) \,|\, (\chi_A > s) $.
Otherwise, $ B_n \subseteq A \setminus \alpha $ implying $ E \in \mathcal{G}_{A \setminus \alpha} $
and so $ G_\alpha \indep E $.
Also, $ G_\alpha \indep (s - \chi_{A \setminus \alpha}) $ and
so, by memorylessness of $ G_\alpha $, for all $ t \geq s $,
\begin{align*}
\probTWhen{\big. s \lor \chi_A = \chi_A > t}{\chi_A > s, E}
& =  \probTWhen{\big. G_\alpha > t - \chi_{A \setminus \alpha}}{G_\alpha > s - \chi_{A \setminus \alpha}, E} \\
& =  \probT{\big. G_\alpha > t - s} = \probTWhen{\big. s \lor \chi_A = \chi_A > t}{\chi_A > s}.
\end{align*}
Thus, $ (s \lor \chi_A \indep E) \,|\, (\chi_A > s) $
and, because this holds for all $ E $,
we get $ (s \lor \chi_A \indep \mathcal{F}_s) \,|\, (\chi_A > s) $.
Conditioning on $ \chi_A \leq s $, we also have $ (s \lor \chi_A = s \indep \mathcal{F}_s) \,|\, (\chi_A \leq s) $
so $ (s \low \chi_A \indep \mathcal{F}_s) \,|\, \sigma(\{\chi_A > s\}) $.
Since $ A \subseteq D $, we know $ \sigma(\{\chi_A > s\}) \subseteq \sigma(Y_s \cap D) \subseteq \mathcal{F}_s $.
So, by the weak union property of conditional independence,
we can extend the conditioning on $ \sigma(\{\chi_A > s\}) $
to conditioning on all of $ \sigma(Y_s \cap D) $.
Therefore, we may conclude $ (s \lor \chi_A \indep \mathcal{F}_s) \,|\, \sigma(Y_s \cap D) $, completing the induction.

Observe that for all $ t > s $,
we have $ \sigma(Y_t \cap D) = \sigma(\chi_A \leq t \,:\, A \in L(\Lambda), A \subseteq D) = \sigma(s \lor \chi_A \leq t \,:\, A \in L(\Lambda), A \subseteq D) $.
Thus, $ (\sigma(Y_t \cap D) \indep \mathcal{F}_s) \,|\, \sigma(Y_s \cap D) $
meaning $ Y_t \cap D $ is a Markov process with respect to $ \mathcal{F}_t $.
In particular, $ Y_t = Y_t \cap \Lambda $ is Markov.
Lastly, we note that for any $ A \in L(\Lambda) $,
let $ \alpha \in \maxim^*(A) $ and take $ D = \low{\alpha} $.
Then, since $ Y_t \cap D $ is Markov and $ \{Y_t = \low{\alpha} \setminus \alpha\}, \{Y_t = A\} \in \mathcal{F}_t $,
we can interchange conditioning on $ Y_t = \low{\alpha} \setminus \alpha $ or $ Y_t = A $
with conditioning on $ Y_t \cap D = \low{\alpha} \setminus \alpha $:
\begin{align*}
\probTWhen{\big. A \cup \alpha \subseteq Y_{t+h}}{Y_t = A}
& =  \probTWhen{\big. \low{\alpha} \subseteq Y_{t+h} \cap D}{Y_t = A} \\
& =  \expectTWhen{\Big. \probTWhen{\big. \low{\alpha} \subseteq Y_{t+h} \cap D}{\mathcal{F}_t}}{Y_t = A} \\
& =  \probTWhen{\big. \low{\alpha} \subseteq Y_{t+h} \cap D}{Y_t \cap D = \low{\alpha} \setminus \alpha} \\
& =  \probTWhen{\big. G_\alpha + \chi_{\low{\alpha} \setminus \alpha} = \chi_\alpha \leq t + h}{\chi_{\low{\alpha} \setminus \alpha} \leq t < \chi_\alpha} \\
& =  1 - \probTWhen{\big. G_\alpha > h + (t - \chi_{\low{\alpha} \setminus \alpha})}{G_\alpha > t - \chi_{\low{\alpha} \setminus \alpha} \geq 0} \\
& =  1 - e^{-\lambda_\alpha h}.
\end{align*}
Therefore, the generator for $ Y_t $ is $ \gener $ from eq. (\ref{eq:markov-gener}).
\end{proof}

\noindent
{\bf Funding.}
S.S. was supported in part by a Simons Sabbatical Grant.

\pagebreak[1]
\providecommand{\bysame}{\leavevmode\hbox to3em{\hrulefill}\thinspace}
\providecommand{\MR}{\relax\ifhmode\unskip\space\fi MR }
\providecommand{\MRhref}[2]{%
  \href{http://www.ams.org/mathscinet-getitem?mr=#1}{#2}
}
\providecommand{\href}[2]{#2}

\end{document}